\documentclass{article}

\usepackage{arxiv}

\usepackage{tikz}
\usepackage[utf8]{inputenc} 
\usepackage[T1]{fontenc}    
\usepackage{hyperref}       
\usepackage{url}            
\usepackage{booktabs}       
\usepackage{amsfonts}       
\usepackage{verbatim}
\usepackage{amsmath,amssymb}
\usepackage{amsthm}
\usepackage{thmtools, thm-restate}
\usepackage{mathrsfs}

\usepackage[backend=biber,style=alphabetic,texencoding=ascii]{biblatex}    

\usepackage{bm}
\usepackage{bbm}

\usepackage{subcaption}
\usepackage{enumitem}
\usepackage{verbatim}
\usepackage{mathtools}
\usepackage{nicefrac}       
\usepackage{microtype}      
\usepackage{lipsum}
\renewcommand{\vec}[1]{\textbf{#1}}
\newcommand{\defeq}{\vcentcolon=}
\newcommand{\expect}{\mathbb{E}}
\newcommand{\prob}{\mathbb{P}}
\usepackage[english]{babel}
\newtheorem{theorem}{Theorem}[section]
\newtheorem{lemma}[theorem]{Lemma}

\newtheorem{defn}[theorem]{Definition}

\theoremstyle{remark}
\newtheorem{remark}[theorem]{Remark}

\newtheoremstyle{draftstyle}%
{3pt}
{3pt}
{\itshape\color{red}}
{}
{\bfseries\color{red}}
{.}
{.5em}
{}

\newcommand{\jstat}[1]{\textcolor{black}{#1}}
\theoremstyle{draftstyle}
\newtheorem{draftnote}[theorem]{Draft Note}

\numberwithin{equation}{section}

\newcommand{\datax}{\mathcal{X}}
\newcommand{\CNk}{C_{N,k}}
\newcommand{\CN}{C_{N}}

\newcommand{\grad}{\nabla}
\newcommand{\ind}[1]{i_{\leq {#1}}}
\newcommand{\indic}{\mathbbm{1}}

\newcommand{\expectGOE}{\expect_{GOE}^{N}}
\newcommand{\Tr}{\text{Tr}}
\usetikzlibrary{decorations.pathmorphing}
\usetikzlibrary{calc}

\title{The Loss Surfaces of Neural Networks with General Activation Functions}

\usepackage{authblk}

\tikzset{
  branch cut/.style={
    decorate,decoration=snake,
    to path={
      (\tikztostart) -- (\tikztotarget) \tikztonodes
    },
  }
}
\usepackage{csquotes}

\begin{document}

\author[1]{\large Nicholas P. Baskerville}
\author[2]{\large Jonathan P. Keating}
\author[1]{\large Francesco Mezzadri}
\author[1]{\large Joseph Najnudel}

\affil[1]{\small\textit{School of Mathematics, University of Bristol, Fry Building, Bristol, BS8 1UG, UK}}
\affil[2]{\small\textit{Mathematical Institute, University of Oxford, Oxford, OX2 6GG, UK}}
\affil[ ]{\texttt{\{n.p.baskerville, F.Mezzadri, joseph.najnudel\}@bristol.ac.uk}, \texttt{Jon.Keating@maths.ox.ac.uk}}

\maketitle


\begin{abstract}
\jstat{    The loss surfaces of deep neural networks have been the subject of several studies, theoretical and experimental, over the last few years. One strand of work considers the complexity, in the sense of local optima, of high dimensional random functions with the aim of informing how local optimisation methods may perform in such complicated settings. Prior work of Choromanska et al (2015) established a direct link between the training loss surfaces of deep multi-layer perceptron networks and spherical multi-spin glass models under some very strong assumptions on the network and its data. In this work, we test the validity of this approach by removing the undesirable restriction to \texttt{ReLU} activation functions. In doing so, we chart a new path through the spin glass complexity calculations using supersymmetric methods in Random Matrix Theory which may prove useful in other contexts. Our results shed new light on both the strengths and the weaknesses of spin glass models in this context.}
\end{abstract}

\tableofcontents

\section{Introduction}
\jstat{Neural networks continue to have substantial success when applied to an increasingly long list of machine learning problems: computer vision, speech processing, natural language processing, reinforcement learning, media generation etc. We refer the interested reader to the excellent website \cite{paperswithcode} where they will find links to published literature detailing the success of neural networks in all fields of machine learning. Despite this success, and the rapid pace of  progress in the development and application of neural network models, the theoretical study and understanding of them is still rather underdeveloped. Deep neural networks appear to subvert many of the classical ideas in machine learning theory. Networks are trained using stochastic gradient-based optimisation methods on very high-dimensional, strongly non-convex surfaces for which no formal convergence or performance guarantees exist, and yet excellent practical performance is routinely obtained with little concern for whether the optimisation problem has been solved. Extremely over-parametrised models can be trained with  large numbers of passes through the data without overfitting. Models with equivalent training performance can have radically different generalisation performance depending on complicated interactions between design choices such as learning rate size (and scheduling) and weight-decay \cite{loshchilov2018decoupled}.}

\jstat{One strand of theoretical work focuses on studying properties of the loss surfaces of large neural networks and the behaviour of gradient descent algorithms on those surfaces. \cite{sagun2014explorations} presented experimental results pointing to a similarity between the loss surfaces of multi-layer networks and spherical multi-spin glasses \cite{mezard1987spin}. \cite{choromanska2015loss} built on this work by presenting modeling assumptions under which the training loss of multi-layer perceptron neural networks with \texttt{ReLU} activations can be shown to be equivalent to a spherical multi-spin glass (with network weights corresponding to spin states). The authors then applied spin glass results of \cite{auffinger2013random} to obtain precise asymptotic results about the complexity\footnote{The notion of complexity will be precised in subsequent sections.} of the training loss surfaces. Crudely, the implication of this work is that the unreasonable efficacy of gradient descent on the high-dimensional and strongly non-convex loss surfaces of neural network models can in part be explained by favourable properties of their geometry that emerge in high dimensions. Relationships between simpler neural networks and spin glasses have been known since \cite{kanter1987associative, gardner1988space, engel2001statistical} and, more generally, connections between spin glass theory and computer science were studied in \cite{nishimori2001statistical} in the context of signal processing (image reconstruction, error correcting codes).}

\jstat{More recent work has dispensed with deriving explicit links between neural networks and spin glasses, instead taking spin glass like objects as a tractable playground for gradient descent in complex high-dimensional environments. In particular, \cite{baity2019comparing} compare empirically the dynamics of state-of-the-art deep neural networks and glassy systems, while \cite{mannelli2019afraid, ros2019complex, arous2019landscape, mannelli2019passed} 
study random tensor models containing some `spike' to represent other features of machine learning problems (some `true signal' to be recovered) and perform explicit complexity calculations as well as gradient descent dynamical calculations revealing phase transitions and landscape trivialisation. \cite{maillard2019landscape} simplify the model in favour of explicitly retaining the activation function non-linearity and performing complexity calculations \`{a} la \cite{auffinger2013random, fyodorov2007replica,  fyodorov2004complexity} for a single neuron. \cite{pennington2017geometry} study the loss surface of random single hidden layer neural networks by applying the generalised Gauss-Newton matrix decomposition to their Hessians and modelling the two components as freely-additive random matrices from certain ensembles. \cite{pennington2017nonlinear, benigni2019eigenvalue} consider the loss surfaces of single layer networks by computing the spectrum of the Gram matrix of network outputs. These works demonstrate the value of studying simplified, randomised neural networks for understanding networks used in practice. }

\jstat{The situation at present is far from clear. The spin glass correspondence and consequent implications for gradient descent based learning from \cite{choromanska2015loss, sagun2014explorations} are tantalising, however there are significant challenges. Even if the mean asymptotic properties of deep neural network loss surfaces were very well described by corresponding multi-spin glass models, the question would still remain whether these properties are in fact relevant to gradient-based algorithms running for sub-exponential time, with some evidence that the answer is negative \cite{baity2019comparing, mannelli2019passed, folena2019rethinking}. Another challenge comes from recent experimental studies  of deep neural network Hessians \cite{papyan2018full, ghorbani2019investigation, granziol2020beyond, granziol2019towards} which reveal spectra with several large outliers and considerable rank degeneracy, deviating significantly from the Gaussian Orthogonal Ensemble semi-circle law implied by a spin glass model. Bearing all this in mind, there is a long and illustrious history in the physics community of fruitfully studying quite unrealistic simplified models of complicated physical systems and still obtaining valuable insights into aspects of the true systems.}

\jstat{Several of the assumptions used in \cite{choromanska2015loss} to obtain a precise spherical multi-spin glass expression are undesirable, as outlined clearly in \cite{choromanska2015open}. Assuming i.i.d. Gaussian data and random labels is clearly a going to greatly simplify the problem, however it is also the case that many of the properties of deep neural networks during training are not specific to any particular dataset, and there may well be phases of training to which such assumptions are more applicable than one might first expect. Gaussian and independence assumptions are commonplace when one is seeking to analyse theoretically very complicated systems, so while they are strong, they are not unusual and it is not unreasonable to expect some important characteristics of real networks to persist. By contrast, the restriction of the arguments in \cite{choromanska2015loss} to exclusively \texttt{ReLU} activations seems innocuous, but we argue quite the opposite is true. There are deep mathematical reasons why Gaussian and independence assumptions are required to make progress in the derivation in \cite{choromanska2015loss}, while the restriction to \texttt{ReLU} activations appears to be an obscure peculiarity of the calculations. The \texttt{ReLU} is certainly a very common choice in practice, but it is by no means the only valid choice, nor always the best; see e.g. leaky \texttt{ReLU} in state-of-the-art image generation \cite{karras2019style} and GELU in state-of-the-art language models \cite{devlin2018bert}. It would not be at all surprising if a spin glass correspondence along the lines of \cite{choromanska2015loss} were impossible without Gaussian and/or independence assumption on the data, however it would be extremely concerning if such a correspondence specifically required \texttt{ReLU} activations. If the conclusions drawn in \cite{choromanska2015loss} about deep neural networks from this correspondence are at all relevant in practice, then they must apply equally to all activation functions used in practice. On the other hand, if the conclusions were \emph{precisely} the same for all reasonable activation functions, it would reveal a limitation of the multi-spin glass correspondence, since activation function choice can have significant implications for training neural networks in practice.}

\jstat{In this work, we return to the modeling assumptions and methodology of \cite{choromanska2015loss} and extend their results to multi-layer perceptron models with any activation function. We demonstrate that the general activation function has the effect of modifying the exact multi-spin glass by the addition of a certain extra deterministic term. We then extend the results of \cite{auffinger2013random} to this new high-dimensional random function. At the level of the logarithmic asymptotic complexity of the loss surface, we obtain precisely the same results as \cite{choromanska2015loss}, however the presence of a general activation function is felt in the sharp asymptotic complexity. On the one hand, our results strengthen the case for \cite{choromanska2015loss} by showing that their derivation is not just an accident in the case of \texttt{ReLU} networks. On the other hand, we have shown that this line of reasoning about neural networks is insensitive to an important design feature of real networks that can have significant impacts on training in practice.}

\jstat{The main calculation in this paper uses a Kac-Rice formula to compute landscape complexity of the modified multi-spin glass model we encounter. Kac-Rice formulae have a long history in the Physics literature \cite{bray1980metastable, bray1981metastable} and more specifically to perform complexity calculations \cite{fyodorov2004complexity, fyodorov2005counting, auffinger2013random}. Complexity calculations in spiked matrix and tensor models in \cite{ros2019complex, arous2019landscape} have addressed spin glass objects with specific rank-1  deterministic additive terms, however those calculations do not extend to the case encountered here since those deterministic terms create a single distinguished direction --- parallel to the gradient of that term everywhere on the sphere --- which is critical to their analysis; our extra deterministic term creates no such single distinguished direction. We chart a different course using supersymmetric methods in Random Matrix Theory. Supersymmetric methods have been used before in spin glass models and complexity calculations \cite{cavagna1999quenched,annibale2003supersymmetric,crisanti2003complexity,fyodorov2004complexity}, often using the replica trick. We show how the full logarithmic complexity results of \cite{auffinger2013random} can be obtained using a supersymmetric approach quite different to the approach used in that and similar works. By moving to this approach, we can make progress despite the presence of the extra deterministic term in the multi-spin glass. Our approach to the supersymmetric calculations most closely follows \cite{fyodorov2015random, nock}, but several steps require approximations due to the extra term. Some of our intermediate results in the supersymmetric and RMT calculations are stronger than required here, but may well be useful in future calculations, e.g. spiked spherical multi-spin glass models with any fixed number of spikes. Finally, our approach computes the total complexity summed over critical points of any index and then uses large deviations principles to obtain the complexity with specified index. This is the reverse order of the approach taken in \cite{auffinger2013random} and may be more widely useful when working with perturbations of matrices with known large deviations principles.}

\subsection{Multi-layer perceptron neural networks}
Let $f:\mathbb{R}\rightarrow\mathbb{R}$ be a suitably well-behaved (e.g. differentiable almost everywhere and with bounded gradient) non-linear \emph{activation function} which is taken to applied entry-wise to vectors and matrices. 
We study multi-layer perceptron neural networks of the form \begin{equation}
    \vec{y}(\vec{x}) = f(W^{(H)}f(W^{(H-1)}f(\ldots f(W^{(1)}\vec{x})\ldots)))\label{eq:nn_def}
\end{equation}
where the input data vectors $\vec{x}$ lie in $\mathbb{R}^d$ and the \emph{weight matrices} $\{W^{(\ell)}\}_{\ell=1}^H$ have any shapes compatible with $\vec{x}\in\mathbb{R}^d$ and $\vec{y}(\vec{x})\in\mathbb{R}^c$. Note that, as in \cite{choromanska2015loss}, we do not consider biases in the network.

\subsection{Outline of results and methods}\label{subsec:outline_results}
Following \cite{choromanska2015loss}, we view $\vec{y}$ as a random function over a high-dimensional weight-space and explore its critical points, i.e. vanishing points of its gradient. The randomness will come from taking the input data to be random. We define the following key quantities\footnote{Recall that the \emph{index} of a critical points is the number of negative eigenvalues of the Hessian at that point.}:

\begin{align}
        C_{k,H}(u) =  &\textrm{expected number of critical points of } \vec{y}  \textrm{ of index } k \textrm{ taking values at most } u, \label{eq:c_kh_rough}\\
        C_{H}(u) =  &\textrm{expected number of critical points of } \vec{y} \textrm{ taking values at most } u.\label{eq:c_h_rough}
\end{align}

In Section \ref{sec:nns_random_funcs} we make precise our heuristic definitions in (\ref{eq:c_kh_rough})-(\ref{eq:c_h_rough}). Following \cite{auffinger2013random} we obtain precise expressions for $C_{k,H}$ and $C_H$ as expectations under the Gaussian Orthogonal Ensemble (GOE) and use them to study the asymptotics in the large-network limit. Our results reveal almost the same `banded structure' of critical points as first found in \cite{choromanska2015loss}. In particular we establish the existence of the same critical values $E_0 > E_1 >  ... >E_{\infty}$ such that, with overwhelming probability, critical points taking (scaled) values in $(-E_k, -E_{k+1})$ have index at-most $k+2$, and that there are exponentially many such critical points. We further obtain the exact leading order terms in the expansion of $C_H(u)$, this being the only point at which the generalised form of the activation function $f$ affects the results. In passing, we also show that the network can be generalised to having any number of output neurons without much affecting the calculations of \cite{choromanska2015loss} who only consider single-output networks.\\

In Section \ref{sec:nns_random_funcs} we extend the derivation of \cite{choromanska2015loss} to general activation functions by leveraging piece-wise linear approximations, and we extend to multiple outputs and new loss functions with a simple extension of the corresponding arguments in \cite{choromanska2015loss}. In Section \ref{sec:goe_expressions} we obtain expressions for the complexities $C_{k,H}, C_H$ using a Kac-Rice formula as in \cite{auffinger2013random, fyodorov2007replica, fyodorov2004complexity} but are forced to deal with a perturbed GOE matrix, preventing the replication of the remaining calculations in that work. Instead, in Section \ref{sec:asymptotic_evaluation} we use the supersymmetric method following closely the work of \cite{nock, fyodorov2015random} and thereby reach the asymptotic results of \cite{auffinger2013random} by entirely different means.

\section{Neural networks as random functions}\label{sec:nns_random_funcs}
In this section we show that, under certain assumptions, optimising the loss function of a neural network is approximately equivalent to minimising the value of a random function on a high dimensional hypersphere, closely related to the spin glass. Our approach is much the same as \cite{choromanska2015loss} but is extended to a general class of activation functions and also to networks with multiple output neurons.

\subsection{Modelling assumptions}\label{subsec:modelling_assumptions}
We make the following assumptions, all of which are required for the specific analytic framework of this paper, and are taken either exactly from, or by close analogy with \cite{choromanska2015loss}. We defer a discussion of their plausibility  and necessity to Section \ref{subsec:discussion_assumptions}.

\begin{enumerate}
    \item Components of data vectors are i.i.d. standard Gaussians.\label{item: assumption_gaussian}
    \item \label{item: assumption_sparse}The neural network can be well approximated as a much sparser\footnote{As in \cite{choromanska2015loss}, a network with $N$ weights is sparse if it has $s$ unique weight values and $s\ll N$.} network that achieves very similar accuracy.
    \item \label{item: assumption_uniform_weights}The unique weights of the sparse network are approximately uniformly distributed over the graph of weight connections.
    \item \label{item: assumption_diff}The activation function is twice-differentiable almost everywhere in $\mathbb{R}$ and can be well approximated as a piece-wise linear function with finitely many linear pieces.
        \item \label{item: assumption_bernoulli}The action of the piece-wise linear approximation to the activation function on the network graph can be modelled as i.i.d. discrete random variables, independent of the data at each node, indicating which linear piece is active.
    \item \label{item: assumption_sphere}The unique weights of a the sparse neural network lie on a hyper-sphere of some radius.
\end{enumerate}
\begin{remark}
An alternative to assumption \ref{item: assumption_bernoulli} would be to take the activation function to be \emph{random} (and so too its piece-wise linear approximation). In this paradigm, we consider the ensuing analysis of this paper to be a study of the \emph{mean properties} of the induced ensemble of neural networks. Resorting to studying mean properties of complicated stochastic systems is a standard means of simplifying the analysis. We do not develop this remark further, but claim that the following calculations are not much affected by switching to this interpretation.
\end{remark}

\subsection{Linearising loss functions}\label{subsec:linear_loss}
In \cite{choromanska2015loss} the authors consider networks with a single output neuron with either $L_1$ or hinge loss and show that both losses are, in effect, just linear in the network output and with positive coefficient, so that minimising the loss can be replaced with minimising the network output. Our ensuing analysis can just as well be applied to precisely these situations, but here we present arguments to extend the applicability to multiple output neurons for $L_1$ regression loss and the widely-used cross-entropy loss \cite{DBLP:journals/corr/JanochaC17} for classification.\\

\textbf{$\textrm{L}_1$ loss.}
The $L_1$ loss is given by \begin{equation}
    \mathcal{L}_{L_1}(\vec{y}(\vec{X}), \vec{Y}) \defeq \sum_{i=1}^c| y_i(\vec{X}) - Y_i|\label{eq:l1_def}
\end{equation}
where $\vec{X}$ is a single random data vector and $\vec{Y}$ a single target output.
Following \cite{choromanska2015loss}, we assume that the absolute values in (\ref{eq:l1_def}) can be modelled as Bernoulli random variables, $M_i$ say,  taking values in $\{-1, 1\}$. We do not expect $\vec{X}, \vec{Y}$ and the $M_i$ to be independent, however it may be reasonable to assume that $\vec{X}$ and the $M_i$ are conditionally independent conditioned on $\vec{Y}$. We then have  \begin{align}
    \expect_{M | \vec{Y}}\mathcal{L}_{L_1}(\vec{y}(\vec{X}), \vec{Y}) = \expect_{M | \vec{Y}}  \sum_{i=1}^c M_i (y_i(\vec{X}) - Y_i) &=  \sum_{i=1}^c (2\pi_i-1) y_i(\vec{X}) -   \sum_{i=1}^c \expect_{M | \vec{Y}} M_iY_i\notag\\
    &=  \sum_{i=1}^c (2\pi_i-1) y_i(\vec{X}) -   \sum_{i=1}^c (2\pi_i - 1)Y_i
    \label{eq:l1_linear_result}
\end{align}
where the $M_i$ are Bernoulli random variables with $\prob(M_i = 1) = \pi_i$. Observe that the second term in (\ref{eq:l1_linear_result}) is independent of the parameters of the network.

\textbf{Cross-entropy loss.}
The cross-entropy loss is given by \begin{equation}
    \mathcal{L}_{\text{entr}}(\vec{y}(\vec{X}), \vec{Y}) \defeq -\sum_{i=1}^c Y_i \log\left(\text{SM}[\vec{y}(\vec{X})]_i\right)\label{eq:ce_def}
\end{equation}
where $\text{SM}$ is the \emph{soft-max} function: \begin{align}
    \text{SM} : &\mathbb{R}^c \rightarrow \mathbb{R}^c,\notag\\
    &\vec{z} \mapsto \frac{\exp(\vec{z})}{\sum_{i=1}^m \exp(z_i)} \label{eq:softmax_def}
\end{align}
and $\exp(\cdot)$ is understood to be applied entry-wise. Note that we are applying the standard procedure of mapping network outputs onto the simplex $\Delta^{c-1}$ to allow us to calculate a mutual entropy. Restricting to $c$-class classification problems and using one-hot label vectors \cite{one-hot}, we obtain \begin{align}
   \mathcal{L}_{\text{entr}}(\vec{y}(\vec{X}), \vec{Y}) &= -\sum_{i=1}^c Y_i\left\{y_i(\vec{X}) - \log\left(\sum_{j=1}^c \exp(y_j(\vec{X}))\right)\right\}\label{eq:ce_linear_part}
\end{align}
We note that classification networks typically produce very `spiked' soft-max outputs \cite{DBLP:journals/corr/GuoPSW17}, therefore we make the approximation \begin{equation}
    \sum_{i=1}^c \exp(y_i(\vec{X})) \approx \max_{i=1,\ldots, c} \{\exp(y_i(\vec{X}))\}\label{eq:ce_approx}
\end{equation}
and so we obtain from (\ref{eq:ce_linear_part}) and (\ref{eq:ce_approx}) \begin{align}
  \mathcal{L}_{\text{entr}}(\vec{y}(\vec{X}), \vec{Y}) &\approx -\sum_{i=1}^c\left\{ Y_i y_i(\vec{X}) - Y_i\max_{j=1,\ldots,c}\{y_j(\vec{X})\}\right\}\label{eq:ce_linear_approx_pre_categorical}
\end{align}
We now model the max operation in (\ref{eq:ce_linear_approx_pre_categorical}) with a categorical variable, $M''$ say, over the indices $i=1,\ldots, c$ and take expectations (again assuming conditional independence of $\vec{X}$ and $M''$) to obtain \begin{equation}
    \expect_{M'' | \vec{Y} }\mathcal{L}_{\text{entr}}(\vec{y}(\vec{X}), \vec{Y}) = -\sum_{i=1}^c  Y_i \left(y_i(\vec{x}) - \sum_{j=1}^c \pi_j'' y_j(\vec{X})\right)\label{eq:ce_linear_final}
\end{equation}
Now $\vec{Y}$ is a one-hot vector and so (\ref{eq:ce_linear_final}) in fact reduces to \begin{equation}\label{eq:ce_linear_result}
      \expect_{M'' | \vec{Y} }\mathcal{L}_{\text{entr}}(\vec{y}(\vec{X}), \vec{Y}) = \sum_{j=1}^c \pi_j'' y_j(\vec{x}) - y_i(\vec{x})
\end{equation} 
for some $i$.\\

\begin{remark}
The arguments in this section are not intended to be anything more than heuristic, so as to justify our study of $\vec{a}^T\vec{y}$ for some constant $\vec{a}$ instead of the actual loss function of a neural network. The modelling assumptions required are no stronger than those used in \cite{choromanska2015loss}.
\end{remark}

\subsection{Network outputs as spin glass-like objects}
We assume that the activation function, $f$, can be well approximated by a piece-wise linear function with finitely many linear pieces. To be precise, given any $\epsilon > 0$ there exists some positive integer $L$ and real numbers $\{\alpha_i, \beta_i\}_{i=1}^L$ and real $a_1 < a_2 < \ldots < a_{L-1}$ such that \begin{align}
    |f(x) - (\alpha_{i+1} x + \beta_{i+1})| &< \epsilon ~~~ \forall x\in(a_i, a_{i+1}], ~ 1 \leq i \leq L-2,\notag\\
    |f(x) - (\alpha_1 x + \beta_1)| &< \epsilon ~~~ \forall x\in(-\infty, a_1],\label{eq:piecewise_lin_def}\\
    |f(x) - (\alpha_L x + \beta_L)| &< \epsilon ~~~ \forall x\in(a_{L-1}, \infty).\notag
\end{align}
Note that the $\{\alpha_i, \beta_i\}_{i=1}^L$ and $\{a_i\}_{i=1}^{L-1}$ are constrained by $L-1$ equations to enforce continuity, viz. \begin{align}\label{eq:pwise_cont_constraint}
    \alpha_{i+1}a_i + \beta_{i+1} = \alpha_{i}a_i + \beta_i, ~~~~ 1\leq i \leq L-1
\end{align}
\begin{defn}
A continuous piece-wise linear function with $L$ pieces $\hat{f}\left(x; \left\{\alpha_i, \beta_i\right\}_{i=1}^L , \left\{a_i\right\}_{i=1}^{L-1}\right)$ is an $(L,\epsilon)$-\emph{approximation} to to a function $f$ if $\left|f(x) - \hat{f}\left(x; \left\{\alpha_i, \beta_i\right\}_{i=1}^L , \left\{a_i\right\}_{i=1}^{L-1}\right)\right| < \epsilon$ for all $x\in\mathbb{R}$.
\end{defn}

Given the above definition, we can establish the following.

\begin{lemma}\label{lemma:linear_approx}
Let $\hat{f}\left(\cdot ; \left\{\alpha_i, \beta_i\right\}_{i=1}^L , \left\{a_i\right\}_{i=1}^{L-1}\right)$ be a $(L, \epsilon)$-approximation to $f$. Assume that all the $W^{(i)}$ are bounded in Frobenius norm\footnote{Recall assumption \ref{item: assumption_sphere}, which is translated here to imply bounded Frobenius norm.}. Then there exists some constant $K>0$, independent of all $W^{(i)}$, such that\begin{equation}
    \left\Vert f(W^{(H)}f(W^{(H-1)}f(\ldots f(W^{(1)}\vec{x})\ldots))) -  \hat{f}(W^{(H)}\hat{f}(W^{(H-1)}\hat{f}(\ldots \hat{f}(W^{(1)}\vec{x})\ldots)))\right\Vert_2 < K\epsilon\label{eq:pwise_lemma}
\end{equation}
for all $\vec{x}\in\mathbb{R}^d.$
\end{lemma}

\begin{proof}
Suppose that (\ref{eq:pwise_lemma}) holds with $H-1$ in place of $H$.
Because $\hat{f}$ is piece-wise linear and continuous then we clearly have \begin{equation}
   |\hat{f}(x) - \hat{f}(y)| \leq \max_{i=1,\ldots, L}\{|\alpha_i|\} |x-y|\equiv K'|x-y|\label{eq:fhat_lipschitz}
\end{equation} 
which can be seen by writing \begin{equation}
    \hat{f}(x) - \hat{f}(y)  = (\hat{f}(x) - \hat{f}(a_i)) + (\hat{f}(a_i) - \hat{f}(a_{i-1})) + \ldots + (\hat{f}(a_{j+1}) - \hat{f}(a_j)) + (\hat{f}(a_j) - \hat{f}(y))
\end{equation}
for all intermediate points $a_j, \ldots, a_i \in (y, x)$. Using (\ref{eq:fhat_lipschitz}) and our induction assumption we obtain \begin{align}
    &\left\Vert \hat{f}(W^{(H)}f(W^{(H-1)}f(W^{(H-2)}f(\ldots f(W^{(1)}\vec{x})\ldots))) - \hat{f}(W^{(H)}\hat{f}(W^{(H-1)}\hat{f}(W^{(H-2)}\hat{f}(\ldots \hat{f}(W^{(1)}\vec{x})\ldots)))\right\Vert_2\notag \\
    \leq &cK'\left\Vert W^{(H)}\left[f(W^{(H-1)}f(W^{(H-2)}f(\ldots f(W^{(1)}\vec{x})\ldots))) - \hat{f}(W^{(H-1)}\hat{f}(W^{(H-2)}\hat{f}(\ldots \hat{f}(W^{(1)}\vec{x})\ldots)))\right]\right\Vert_2\notag\\
    \leq &cKK'\left\Vert W^{(H)}\right\Vert_F \epsilon \notag \\
    \leq &K''\epsilon,\notag
\end{align}
for some $K''$, where on the last line we have used the assumption that the network weights are bounded to bound $\Vert W^{(H)}\Vert_F$. The result for $H=1$ follows immediately from (\ref{eq:fhat_lipschitz}).\\
\end{proof}

\begin{remark}
One could be more explicit in the construction of the piece-wise linear approximation $\hat{f}$ from $f$ given the error tolerance $\epsilon$ by following e.g. \cite{berjon2015optimal}. We do not develop this further here as we do not believe it to be important to the practical implications of our results.
\end{remark}

In much the same vein as \cite{choromanska2015loss} (c.f. Lemma 8.1 therein), we now use the following general result for classifiers to further justify our study of approximations to a neural network in the rest of the paper.

\begin{theorem}\label{thm:correlation}
Let $Z_1$ and $Z_2$ be the outputs of two arbitrary $c$-class classifiers on a dataset $\mathcal{X}$. That is, $Z_1(x),Z_2(x)$ take values in $\{1,2,\ldots, c\}$ for $x\in \mathcal{X}$. If $Z_1$ and $Z_2$ differ on no more than $\epsilon|\mathcal{X}|$ points in $\mathcal{X}$, then \begin{equation}
    \text{corr}(Z_1, Z_2) = 1 - \mathcal{O}(\epsilon)
\end{equation}
where, recall, the correlation of two random variables is given by \begin{equation}
   \frac{ \mathbb{E}(Z_1Z_2) - \mathbb{E}Z_1\mathbb{E}Z_2}{std(Z_1)std(Z_2)}.
\end{equation}
\end{theorem}
\begin{proof}
Let $\datax_i\subset\datax$ be the set of data points for which $Z_1=i$ for $i=1,2, \ldots, c$. Let $\datax_{i,j}\subset\datax_i$ be those points for which $Z_1 = i$ but $Z_2 = j$ where $j\neq i$. Define the following: \begin{equation}
    p_i = \frac{|\datax_i|}{|\datax|}, ~~~ \epsilon_i^+ = \sum_{j\neq i}\frac{|\datax_{i,j}|}{|\datax|}, ~~~ \epsilon_i^- = \sum_{j\neq i}\frac{|\datax_{j,i}|}{|\datax|}.
\end{equation}
We then have \begin{align}
    \expect Z_1 &= \sum_{i=1}^c i p_i,\label{eq:corr_Ez1}\\
    \expect Z_2 &= \sum_{i=1}^c i (p_i - \epsilon^+_i + \epsilon^-_i)\label{eq:corr_Ez2}\\
    \expect Z_1Z_2 &= \sum_{i=1}^c i^2 (p_i - \epsilon_i^+) + \sum_{1\leq i < j \leq c} ij\frac{|\datax_{i,j}| + |\datax_{j,i}|}{|\datax|}\label{eq:corr_Ez1z2} \\
    std(Z_1) &= \left[\sum_{i=1}^c i^2p_i - \sum_{i,j}ijp_ip_j\right]^{1/2}\label{eq:corr_stdZ1}\\
     std(Z_2) &= \left[\sum_{i=1}^c i^2(p_i -\epsilon_i^+ + \epsilon_i^-) - \sum_{i,j}ij(p_i - \epsilon_i^+ + \epsilon_i^-)(p_j - \epsilon_j^+ + \epsilon_j^-)\right]^{1/2}.\label{eq:corr_stdZ2}
\end{align}

Now, by assumption $\sum_i \epsilon_i^{\pm} \leq \mathcal{O}(\epsilon)$ and so $\epsilon_i^{\pm} \leq \mathcal{O}(\epsilon)$ for all $i$. Similarly, $|\datax_{i,j}|/|\datax| \leq \mathcal{O}(\epsilon)$ and so we quickly obtain from (\ref{eq:corr_Ez1})-(\ref{eq:corr_Ez1z2}) \begin{equation}
    cov(Z_1, Z_2) = \sum_{i=1}^c i^2p_i - \sum_{i,j}ijp_ip_j + \mathcal{O}(\epsilon).\label{eq:corr_cov}
\end{equation}
Finally, combining (\ref{eq:corr_stdZ1}) - (\ref{eq:corr_cov}) we obtain \begin{equation}
    corr(Z_1, Z_2) = \frac{1 + \mathcal{O}(\epsilon)}{(1  + \mathcal{O}(\epsilon))^{1/2}} = 1 + \mathcal{O}(\epsilon).
\end{equation}
\end{proof}

The final intermediate result we require gives an explicit expression for the output of a neural network with a piece-wise linear activation function.

\begin{lemma}\label{lemma:pwise_network}
Consider the following neural network \begin{equation}
    \hat{\vec{y}}(\vec{x}) = \hat{f}(W^{(H)}\hat{f}(\ldots \hat{f}(W^{(1)}\vec{x})\ldots))
\end{equation}
where $\hat{f}\left(\cdot; \left\{\alpha_i, \beta_i\right\}_{i=1}^L , \left\{a_i\right\}_{i=1}^{L-1}\right)$ is a piece-wise linear function with $L$ pieces. Then there exist $A_{i,j}$ taking values in \begin{equation}
   \mathcal{A} \defeq \left\{\prod_{i=1}^H \alpha_{j_i}\ ~:~ j_1,\ldots, j_H \in \{1,\ldots, L\}\right\}
\end{equation} and $ A^{(\ell)}_{i,j}$ taking values in \begin{equation}
    \mathcal{A}^{(\ell)} \defeq\left\{\beta_k\prod_{r=1}^{H-\ell} \alpha_{j_r} ~:~ j_1,\ldots, j_{H-\ell}, k \in \{1,\ldots, L\}\right\}\end{equation}
    such that \begin{equation}
        \hat{y_i}(\vec{x}) = \sum_{j=1}^d \sum_{k\in\Gamma_i}x_{j,k} A_{j,k} \prod_{l=1}^H w_{j,k}^{(l)} + \sum_{\ell=1}^H\sum_{j=1}^{n_{\ell}} \sum_{k\in\Gamma_i^{(\ell)}} A_{j,k}^{(\ell)} \prod_{r=\ell +1}^H w_{j,k}^{(r)}
    \end{equation}
    where $\Gamma_i$ is an indexing of all paths through the network to the $i$-th output neuron, $\Gamma_i^{(\ell)}$ is an indexing of all the paths through the network from the $\ell$-th layer to the $i$-th output neuron, $w_{j,k}^{(l)}$ is the weight applied to the $j$-th input on the $k$-th path in the $l$-th layer, $x_{j,k} = x_j$, and $n_{\ell}$ is the number of neurons in layer $\ell$.
\end{lemma}
\begin{proof}
Firstly, for some $j=1,\ldots, L$ \begin{equation}
    \hat{f}(W^{(1)}\vec{x})_i = \alpha_j (W^{(1)}\vec{x})_i + \beta_j
\end{equation}
and so there exist $j_1, j_2,\ldots \in\{1,\ldots, L\}$ such that  \begin{equation}
    [W^{(2)} \hat{f}(W^{(1)}\vec{x})]_i = \sum_k W^{(2)}_{ik}( \alpha_{j_k}(W^{(1)}\vec{x})_k  + \beta_{j_k}) = \sum_k \alpha_{j_k}W^{(2)}_{ik} \sum_l W^{(1)}_{kl}x_l + \sum_k  W^{(2)}_{ik}\beta_{j_k}.\label{eq:pwise_network}
\end{equation}
Continuing in the vein of (\ref{eq:pwise_network}), there exist $k_1,k_2, \ldots\in\{1, \ldots, L\}$ such that \begin{equation}
\hat{f}(W^{(2)} \hat{f}(W^{(1)}\vec{x}))_i = \alpha_{k_i}\sum_r \alpha_{j_r}W_{ir}^{(2)}\sum_l W_{kl}^{(1)}x_l + \alpha_{k_i}\sum_{r} W^{(2)}_{ir} \beta_{j_r} + \beta_{k_i}
\end{equation}
from which we can see that the result follows by re-indexing and induction.
\end{proof}

We now return to the neural network $\vec{y}(\cdot)$. Fix some small $\epsilon>0$, let $\hat{f}\left(\cdot; \{\alpha_i, \beta_i\}_{i=1}^L, \{x_i\}_{i=}^{L-1}\right)$ be a $(L,\epsilon)$-approximation to $f$ and let $\hat{\vec{y}}$ be the same network as $\vec{y}$ but with $f$ replaced by $\hat{f}$. By Lemma \ref{lemma:linear_approx}, we have\footnote{Here we use the standard notation that, for a function $p$ on $\mathcal{B}$, $p\lesssim \epsilon$ if there exists a constant $K$ such that $p(x) \leq K\epsilon$ for all $x\in\mathcal{B}$.}  \begin{equation}\label{eq:y_yhat_epsilon}
\Vert\vec{y}(\vec{x}) - \hat{\vec{y}}(\vec{x})\Vert_2 \lesssim \epsilon\end{equation} for all $\vec{x}\in\mathbb{R}^d$, and so we can adjust the weights of $\hat{\vec{y}}$ to obtain a network with accuracy within $\mathcal{O}(\epsilon)$ of $\vec{y}$. We then apply Lemma \ref{lemma:pwise_network} to $\hat{\vec{y}}$ and assume\footnote{This assumption is the natural analogue of the assumption used in \cite{choromanska2015loss}.} that the $A_{i,j}$ and $A_{i,j}^{(\ell)}$ can be modelled as i.i.d. discrete random variables with \begin{align}\label{eq:rho_def}
    \expect A_{i,j} = \rho, ~~~ \expect A_{i,j}^{(\ell)} = \rho_{\ell}
\end{align}
and then \begin{equation}\label{eq:spin_glass_pre_sparse}
\expect \hat{y}_i(\vec{X}) = \rho \expect_{\vec{x}} \sum_{j=1}^d \sum_{k\in\Gamma_i}X_{j,k} \prod_{l=1}^H w_{j,k}^{(l)} + \sum_{\ell=1}^H \rho_{\ell}\sum_{j=1}^{n_{\ell}} \sum_{k\in\Gamma_i^{(\ell)}} \prod_{r=\ell +1}^H w_{j,k}^{(r)}.
\end{equation}
Our reasoning is now identical to that in Section 3.3 of \cite{choromanska2015loss}. We use the assumptions of sparsity and uniformity (Section \ref{subsec:modelling_assumptions}, assumptions \ref{item: assumption_sparse}, \ref{item: assumption_uniform_weights}) and some further re-indexing to replace (\ref{eq:spin_glass_pre_sparse}) by \begin{equation}\label{eq:y_tilde_results}
    \expect \tilde{y}_i(\vec{X}) = \rho \expect_{\vec{X}} \sum_{i_1, \ldots, i_H = 1}^{\Lambda} X_{i_1, \ldots, i_H} \prod_{k=1}^H w_{i_k} + \sum_{\ell=1}^H\rho_{\ell} \sum_{i_{\ell + 1}, \ldots, i_H=1}^{\Lambda} \prod_{k=\ell +1}^H w_{i_k}
\end{equation}
where $\Lambda$ is the number of unique weights of the network and, in particular, the sparsity and uniformity assumptions are chosen to give \begin{align}\label{eq:sparse_uniform_epsilon}
    \expect_{\vec{X}}\left\Vert \tilde{\vec{y}}(\vec{X}) -  \hat{\vec{y}}(\vec{X})\right\Vert_2 \lesssim \epsilon.
\end{align}
(\ref{eq:y_yhat_epsilon}) and (\ref{eq:sparse_uniform_epsilon}) now give \begin{equation}\label{eq:y_ytwid_eps}
       \expect_{\vec{X}} \left\Vert   \tilde{\vec{y}}(\vec{X}) -  \vec{y}(\vec{X})\right\Vert_2 \lesssim \epsilon
\end{equation}
and in the case of classifiers, (\ref{eq:y_ytwid_eps}) ensures that the conditions for  Theorem \ref{thm:correlation} are met, so establishing that \begin{equation}\label{eq:spin_glass_corr}
corr(\tilde{\vec{y}}(\vec{X}) ,\vec{y}(\vec{X})) = 1 - \mathcal{O}(\epsilon).
\end{equation}
As in \cite{choromanska2015loss}, we use these heuristics to justify studying $\tilde{\vec{y}}$ hereafter in place of $\vec{y}$.

Recalling the results of Section \ref{subsec:linear_loss}, in particular (\ref{eq:l1_linear_result}) and (\ref{eq:ce_linear_result}) we conclude that to study the loss surface of $\tilde{\vec{y}}$ under some loss function it is sufficient to study quantities of the form $\sum_{i=1}^c \eta_i \tilde{y}_i$ and, in particular, we study the critical points. The $X$ are centred Gaussian random variables and so any finite weighted sum of some $X$ is a centred Gaussian variable with some variance. We can re-scale variances and absorb constants into the $\rho_{\ell}$ and thereby replace $\sum_i \eta_i \tilde{y}_i(\vec{X})$ with $\tilde{y}_i$(\vec{X}).


Note that we assumed an $L_2$ constraint on the network weights (Section \ref{subsec:modelling_assumptions}, point 6)  and that now carries forward as \begin{equation}\label{eq:weight_norm}
    \frac{1}{\Lambda}\sum_{i=1}^{\Lambda} w_i^2 = \mathcal{C}
\end{equation} 
for some constant $\mathcal{C}$. For ease of notation in the rest of the paper, we define \begin{equation}\label{eq:g_def}
    g(\vec{w}) = \sum_{i_1, \ldots, i_H = 1}^{\Lambda} X_{i_1, \ldots, i_H} \prod_{k=1}^H w_{i_k} + \sum_{\ell=1}^H\rho_{\ell}' \sum_{i_{\ell + 1}, \ldots, i_H=1}^{\Lambda} \prod_{k=\ell +1}^H w_{i_k}
\end{equation}
where $\rho_{\ell}' \defeq  \rho_{\ell}/\rho$. Finally, recall that we assumed the data entries  $X_i$ are i.i.d standard Gaussians. To allow further analytic progress to be made, we follow \cite{choromanska2015loss} and now extend this assumption to $X_{i_1, \ldots, i_H}\overset{\text{i.i.d}}{\sim} \mathcal{N}(0,1)$. The random function $g$ is now our central object of study and, without loss of generality, we take $\mathcal{C}=1$ in (\ref{eq:weight_norm}) so that $g$ is a random function on the ($\Lambda$-1)-sphere of radius $\sqrtsign{\Lambda}$.

Observe that the first term in (\ref{eq:g_def}) is precisely the form of an $H$-spin glass as found in \cite{choromanska2015loss} and the second term is deterministic and contains (rather obliquely) all the dependence on the activation function. Having demonstrated the link between our results and those in \cite{choromanska2015loss}, we now set $\Lambda = N$ for convenience and to make plain the similarities between what follows and \cite{auffinger2013random}. We also drop the primes on $\rho_{\ell}'$.

\subsection{Validity of the modelling assumptions.}\label{subsec:discussion_assumptions}
The authors of \cite{choromanska2015loss} discuss the modelling assumptions in \cite{choromanska2015open}. We add to their comments that the hyper-sphere assumption \ref{item: assumption_sphere} seems easily justifiable as merely $L_2$ weight regularisation.\\

Assumption \ref{item: assumption_bernoulli} from Section \ref{subsec:modelling_assumptions} is perhaps the least palatable, as the section of a piece-wise linear activation function in which a pre-activation value lies is a deterministic function of that pre-activation value and so certainly not i.i.d. across the network and the data items. It is not clear how to directly test the assumption experimentally, but we can certainly perform some experiments to probe its plausibility.

\jstat{For the sake of clarity, consider initially a  \texttt{ReLU} activation function. Let $\mathscr{N}$ be the set of all nodes (neurons) in a neural network, and let $\mathscr{D}$ be a dataset of inputs for this network. Assumption \ref{item: assumption_bernoulli} says that we can model the action of the activation function at any neuron  $\mathfrak{n}\in\mathscr{N}$ and any data point $\vec{x}\in\mathscr{D}$ as i.i.d. Bernoulli random variables. In particular, this is why the the expectations over the activation function indicators and the data distribution can be taken independently in (\ref{eq:spin_glass_pre_sparse}). If one fixes some neuron $\mathfrak{n}\in\mathscr{N}$, and observes its pre-activations over all data points in $\mathscr{D}$, one will observe some proportion $\rho^{\mathfrak{n}}$ of positive values. Assumption \ref{item: assumption_bernoulli} implies that this proportion should be approximately the same for each $\mathfrak{n}\in\mathscr{N}$, namely $p$, where $p$ is the success probability of the Bernoulli. Taking all of the $\rho^{\mathfrak{n}}$ together, their empirical distribution should have low variance and be centred on $p$. More precisely, for large $|\mathscr{D}|$ each $\rho^{\mathfrak{n}}$ should be close in distribution to i.i.d. Gaussian with mean $p$ and variance of order $|\mathscr{D}|^{-1}$, a fact that can be derived simply from the central limit theorem applied to i.i.d. Bernoulli random variables. Similarly, assumption \ref{item: assumption_bernoulli} implies that one can exchange data points and neurons in the previous discussion and so observe proportions $\bar{\rho}^{\vec{x}}$ for each $\vec{x}\in\mathscr{D}$, which again should have an empirical distribution centred on $p$ and with low variance. The value of $p$ is not prescribed by any of our assumptions and nor is it important, all that matters is that the distributions of $\{\rho^{\mathfrak{n}}\}_{\mathfrak{n}\in\mathscr{N}}$  and $\{\bar{\rho}^{\vec{x}}\}_{\vec{x}\in\mathscr{D}}$ are strongly peaked around some common mean.}

We will now generalise the previous discussion to the case of any number of linear pieces of the activation function. Suppose that the activation function is piece-wise linear in $L$ pieces and denote by $I_1, \ldots, I_L$ the disjoint intervals on which the activation function is linear; $\{I_i\}_{i=1}^L$ partition $\mathbb{R}$. Let $\iota(\vec{x}, \mathfrak{n})$ be defined so that the pre-activation to neuron $\mathfrak{n}\in\mathscr{N}$ when evaluating at $\vec{x}\in\mathscr{D}$ lies in $I_{\iota(\vec{x}, \mathfrak{n})}.$ We consider two scenarios, \emph{data averaging} and \emph{neuron averaging}. Under data averaging, we fix a neuron and observe the pre-activations observed over all $\mathscr{D}$, i.e. define for $j=1,\ldots, L$ the counts \begin{align}
    \chi_j^{\mathfrak{n}} = |\{\vec{x}\in\mathscr{D} ~:~ \iota(\vec{x}, \mathfrak{n}) = j\}|
\end{align}
and thence the $L-1$ independent ratios \begin{align}
    \rho_j^{\mathfrak{n}} = \frac{ \chi_j^{\mathfrak{n}} }{\sum_{i=1}^L \chi_1^{\mathfrak{n}}}
\end{align}
for $j=2,\ldots, L$. Similarly, in neuron averaging we define \begin{align}
     \bar{\chi}_j^{\vec{x}} &= |\{\mathfrak{n}\in\mathscr{N} ~:~ \iota(\vec{x}, \mathfrak{n}) = j\}|,\\
       \bar{\rho}_j^{\vec{x}} &= \frac{ \bar{\chi}_j^{\vec{x}} }{ \sum_{i=1}^L\bar{\chi}_1^{\vec{x}} }.
\end{align}

We thus have the sets of observed real quantities \begin{align}
   R_j &= \{ \rho_j^{\mathfrak{n}} ~:~ \mathfrak{n}\in\mathscr{N}\},\\
   \bar{R}_j &= \{ \bar{\rho}_j^{\vec{x}} ~:~ \vec{x}\in\mathscr{D}\}.\\
\end{align}
Under assumption \ref{item: assumption_bernoulli}, the empirical variance of the values in $R_j$ and $\bar{R}_j$ should be small. We run experiments to interrogate this hypothesis under a variety of conditions. In particular: \begin{enumerate}
    \item  Standard Gaussian i.i.d. data vs. `real' data (MNIST digits \cite{lecun-mnisthandwrittendigit-2010}).
    \item Multi-layer perceptron (MLP) vs. convolutional (CNN) architecture.
    \item Trained vs. randomly initialised weights.
    \item Various piece-wise linear activation functions.
\end{enumerate}

In particular:

\begin{enumerate}
    \item We generate 10000 i.i.d. Gaussian data vectors of length 784 (to match the size of MNIST digits).
    \item We fix a MLP architecture of 5 layers and a CNN architecture with 3 convolutional layers and 2 fully-connected. The exact architecture details are given in the Appendix.
    \item We train all networks to test accuracy of at least $97\%$ and use dropout with rate $0.1$ during training.
    \item We test \texttt{ReLU} (2 pieces), \texttt{HardTanh} (3 pieces) and a custom 5 piece function. Full details are given in Appendix \ref{ap:experiments}.
\end{enumerate}

To examine the $R_j$ and $\bar{R}_j$, we produce histograms of $R_2$ for $L=2$ (i.e. \texttt{ReLU}), joint density plots of $(R_2, R_3)$ for $L=3$ (i.e. \texttt{HardTanh}) and pair-plots of $(R_2, R_3, R_4, R_5)$ for $L=5$.  We are presently only interested in the size of the variance shown, but these full distribution plots are included in-case any further interesting observations can be made in the future. Figures \ref{fig:probe_agg_data_random_weights}-\ref{fig:probe_agg_neuron_trained_weights} show the results for \texttt{ReLU} activations and Figures \ref{fig:probe_agg_data_random_weights_tanh}-\ref{fig:probe_agg_neuron_trained_weights_tanh} show the results for \texttt{HardTanh}. The qualitative trends are much the same for all three activation functions, but the plots for the 5-piece function are very large and so are relegated to the supplementary material\footnote{\url{https://github.com/npbaskerville/loss-surfaces-general-activation-functions/blob/master/Loss_surfaces_of_neural_networks_with_general_activation_functions___supplimentary.pdf}}. We make the following observations:
  \begin{figure*}[p]
        \centering
        \begin{subfigure}[b]{0.236\textwidth}
            \centering
            \includegraphics[width=\textwidth]{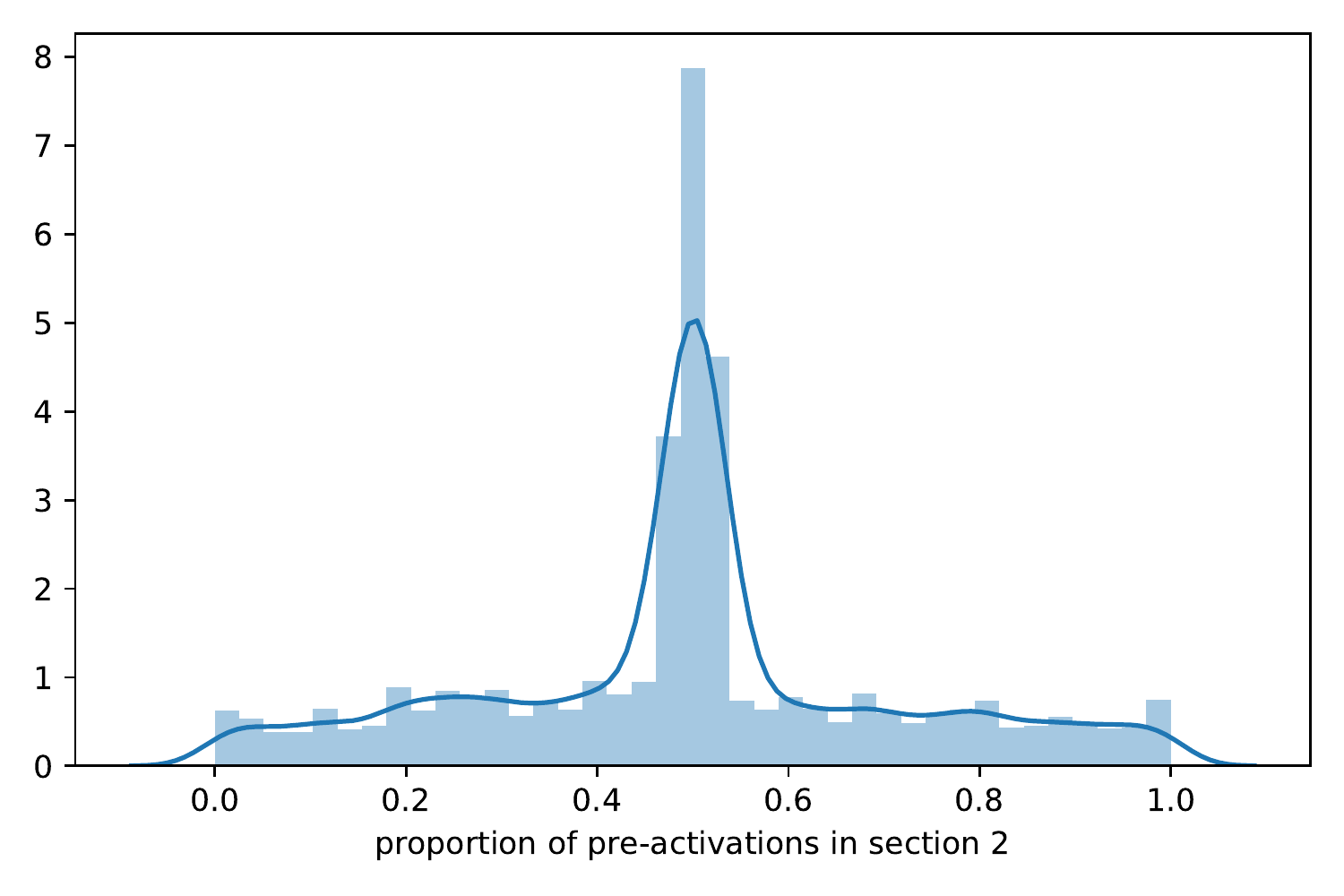}
            \caption{MLP, i.i.d. normal data.} 
            \label{fig:probe_agg_data_random_weights_mlp_iid}
        \end{subfigure}
        \begin{subfigure}[b]{0.236\textwidth}
            \centering
            \includegraphics[width=\textwidth]{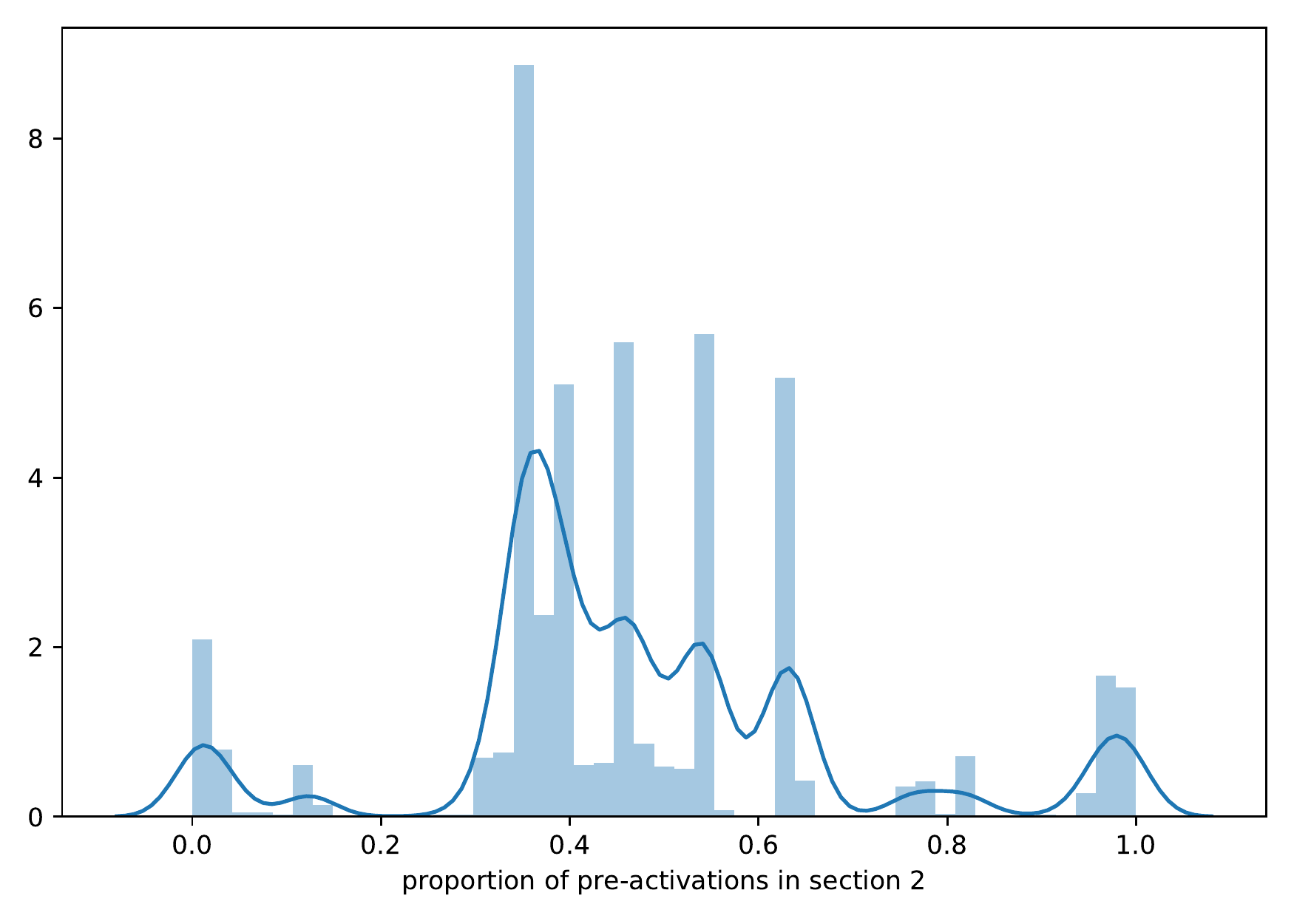}
            \caption{LeNet, i.i.d. normal data.} 
            \label{fig:probe_agg_data_random_weights_lenet_iid}
        \end{subfigure}
        \begin{subfigure}[b]{0.236\textwidth}
            \centering
            \includegraphics[width=\textwidth]{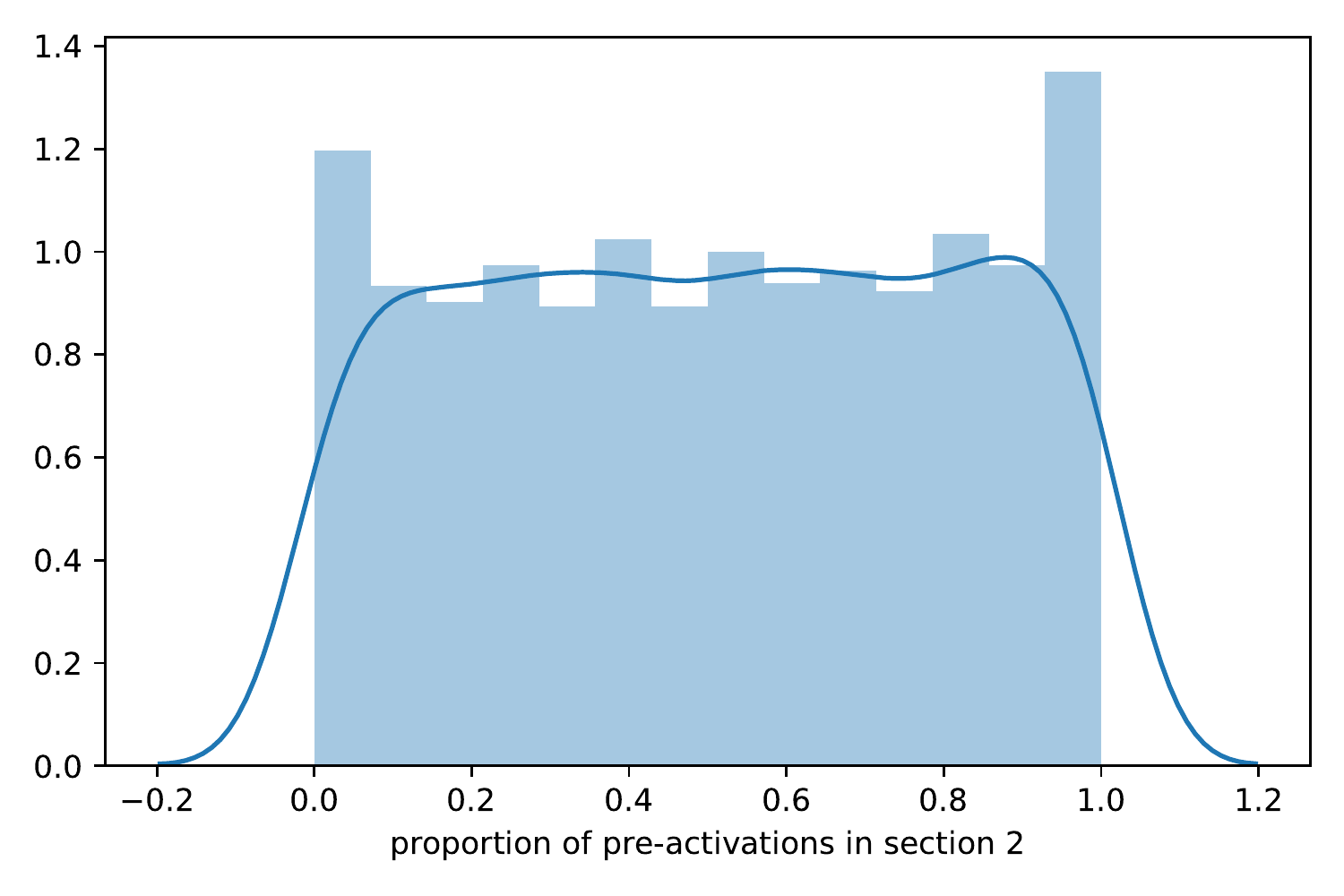}
            \caption{MLP, MNIST data.} 
            \label{fig:probe_agg_data_random_weights_mlp_mnist}
        \end{subfigure}
        \begin{subfigure}[b]{0.236\textwidth}
            \centering
            \includegraphics[width=\textwidth]{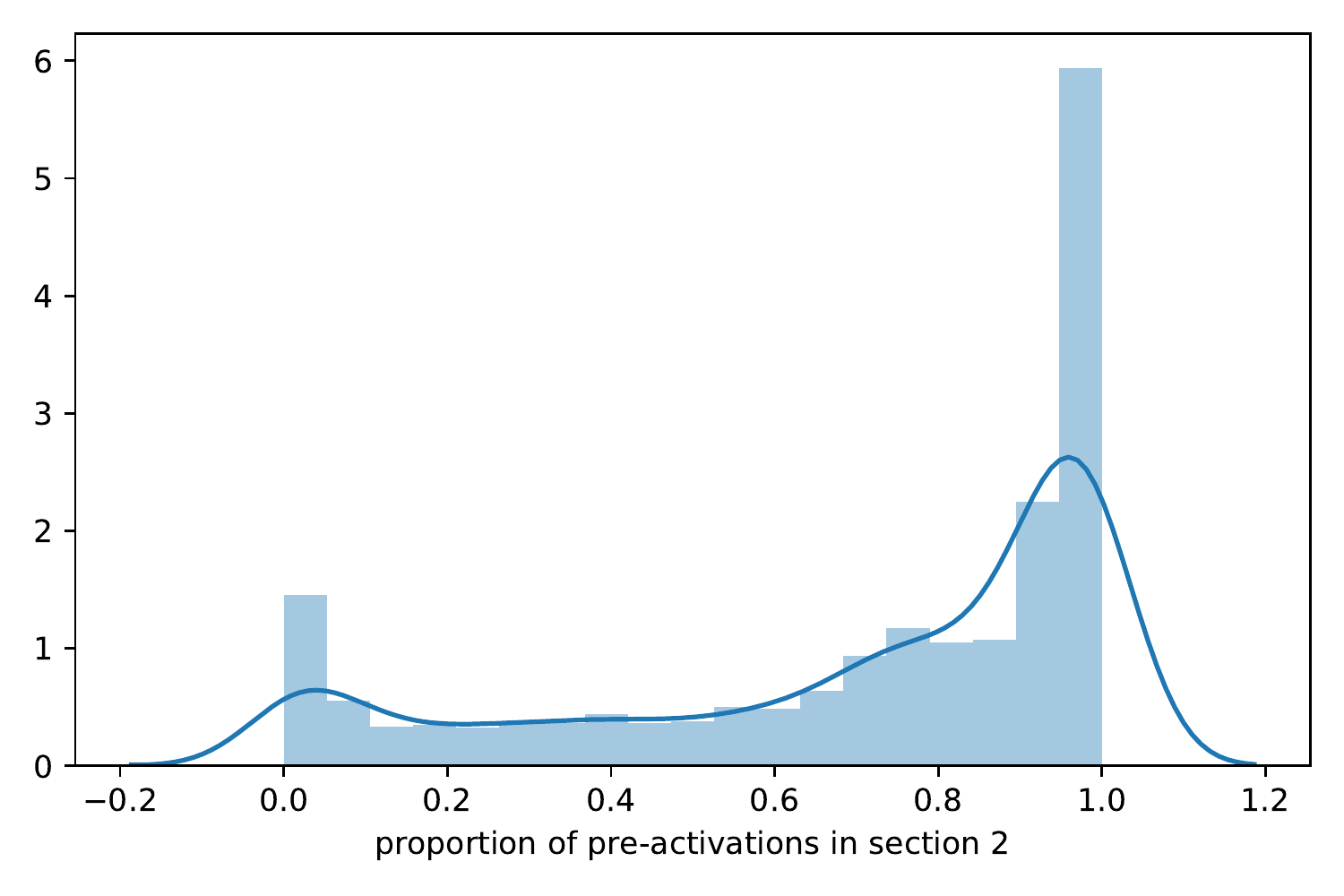}
            \caption{LeNet, MNIST data.} 
            \label{fig:probe_agg_data_random_weights_lenet_mnist}
        \end{subfigure}
        \caption{Experimental distribution of $R_2$ (data averaging; each sample is a single neuron) for random MLP and LeNet \texttt{ReLU} networks, and i.i.d. normal and MNIST data. The blue line is a kernel density estimation fit.} 
        \label{fig:probe_agg_data_random_weights}
    \end{figure*}
  \begin{figure*}[p]
        \centering
        \begin{subfigure}[b]{0.236\textwidth}
            \centering
            \includegraphics[width=\textwidth]{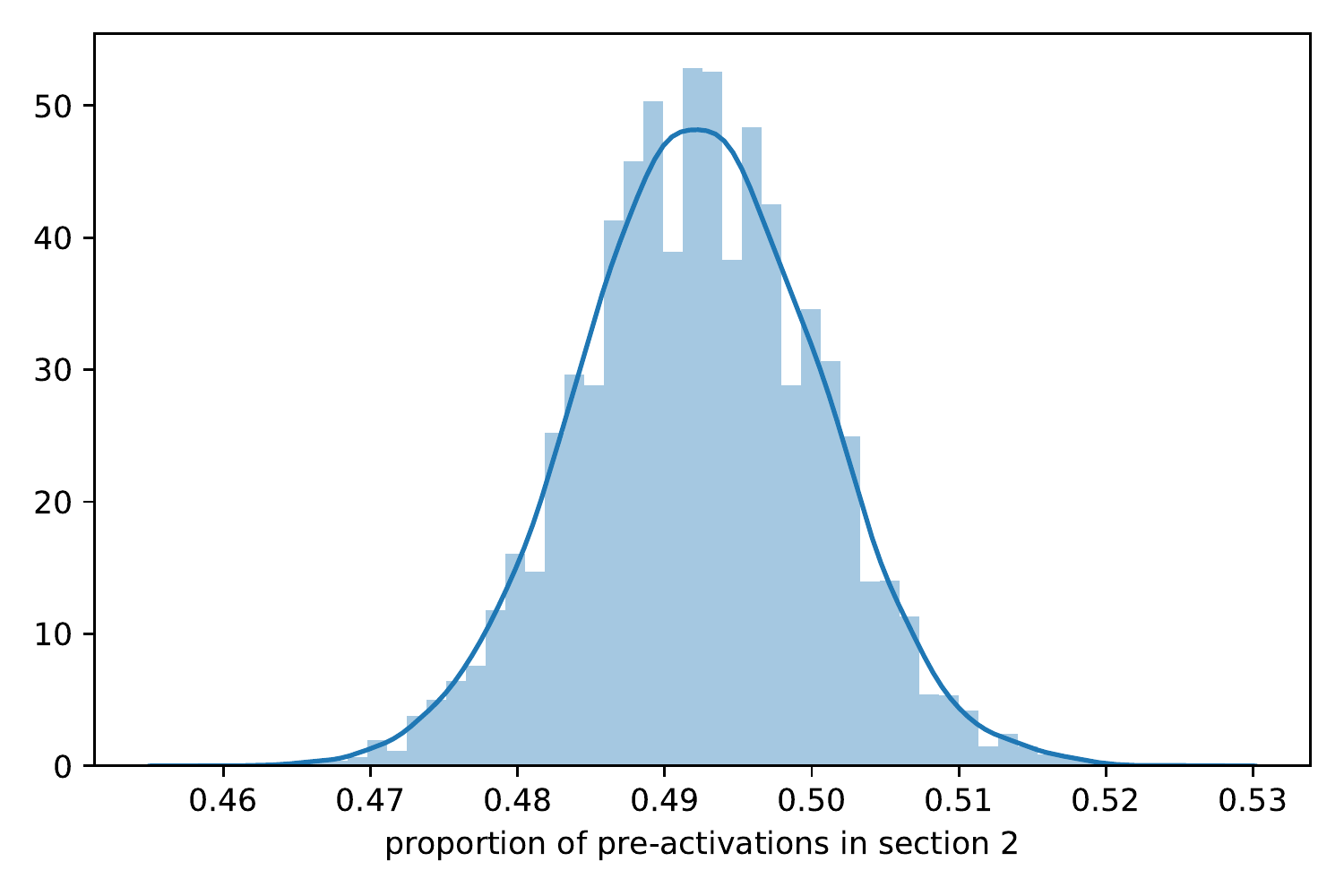}
            \caption{MLP, i.i.d. normal data.} 
            \label{fig:probe_agg_neuron_random_weights_mlp_iid}
        \end{subfigure}
        \begin{subfigure}[b]{0.236\textwidth}
            \centering
            \includegraphics[width=\textwidth]{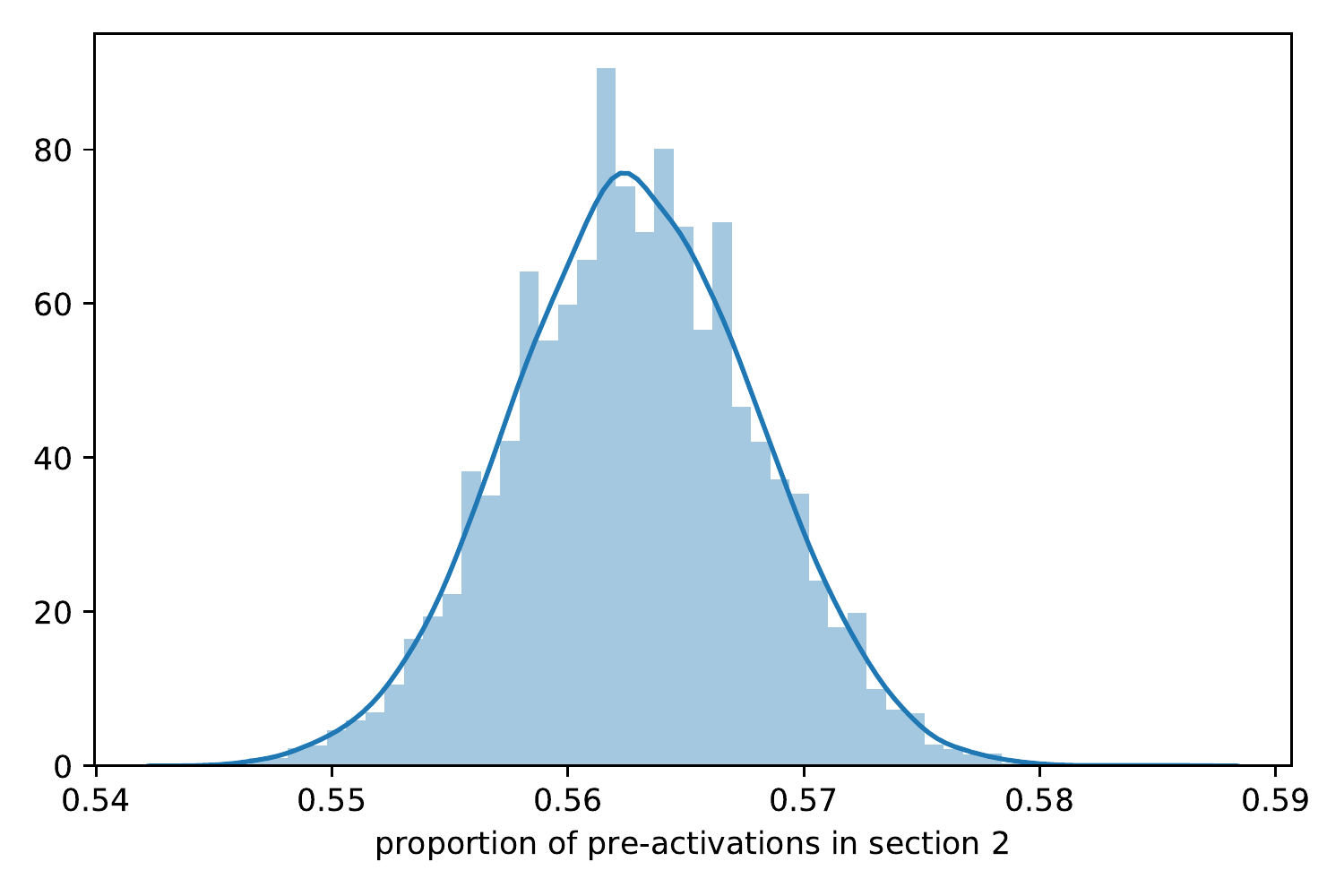}
            \caption{LeNet, i.i.d. normal data.} 
            \label{fig:probe_agg_neuron_random_weights_lenet_iid}
        \end{subfigure}
        \begin{subfigure}[b]{0.236\textwidth}
            \centering
            \includegraphics[width=\textwidth]{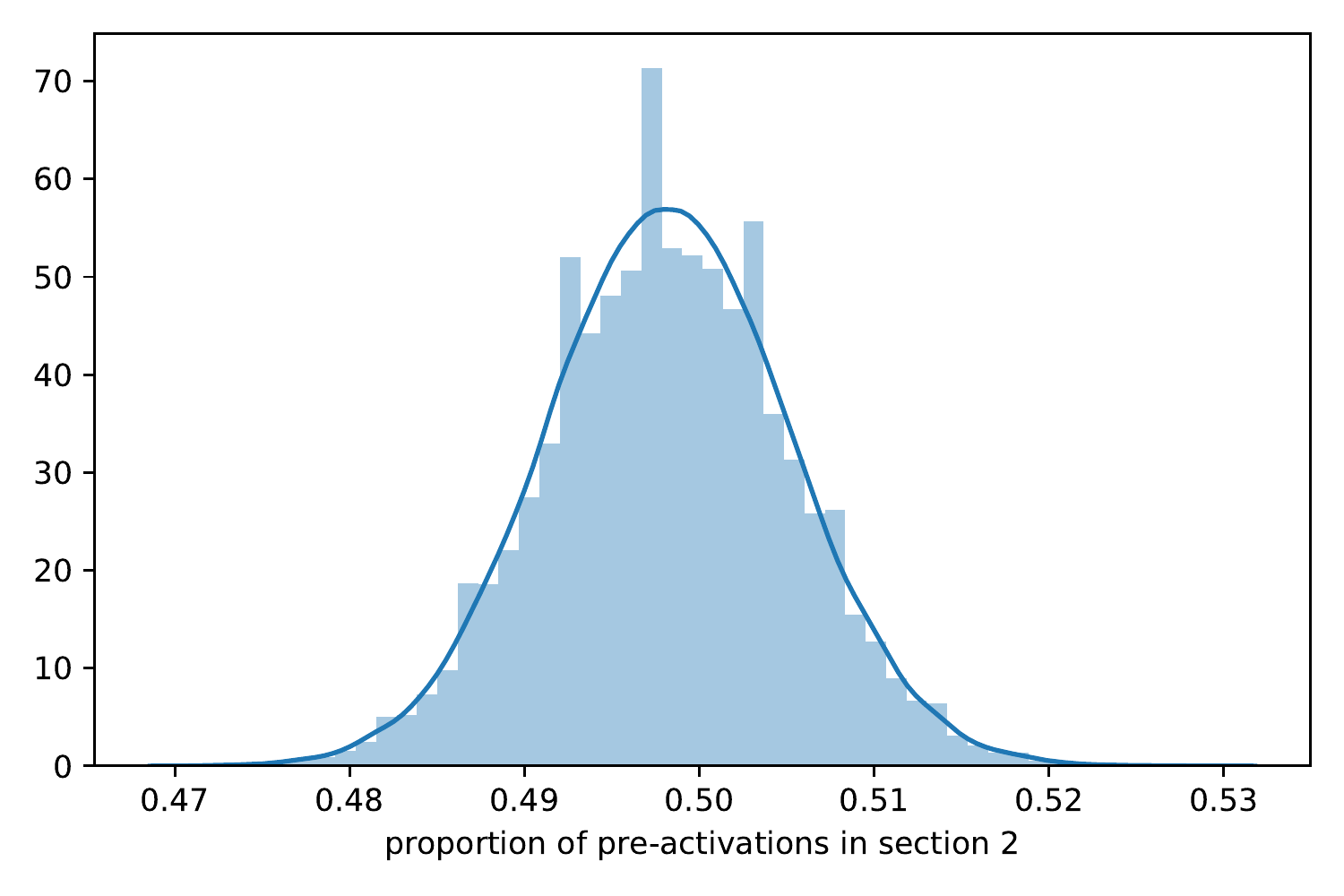}
            \caption{MLP, MNIST data.} 
            \label{fig:probe_agg_neuron_random_weights_mlp_mnist}
        \end{subfigure}
        \begin{subfigure}[b]{0.236\textwidth}
            \centering
            \includegraphics[width=\textwidth]{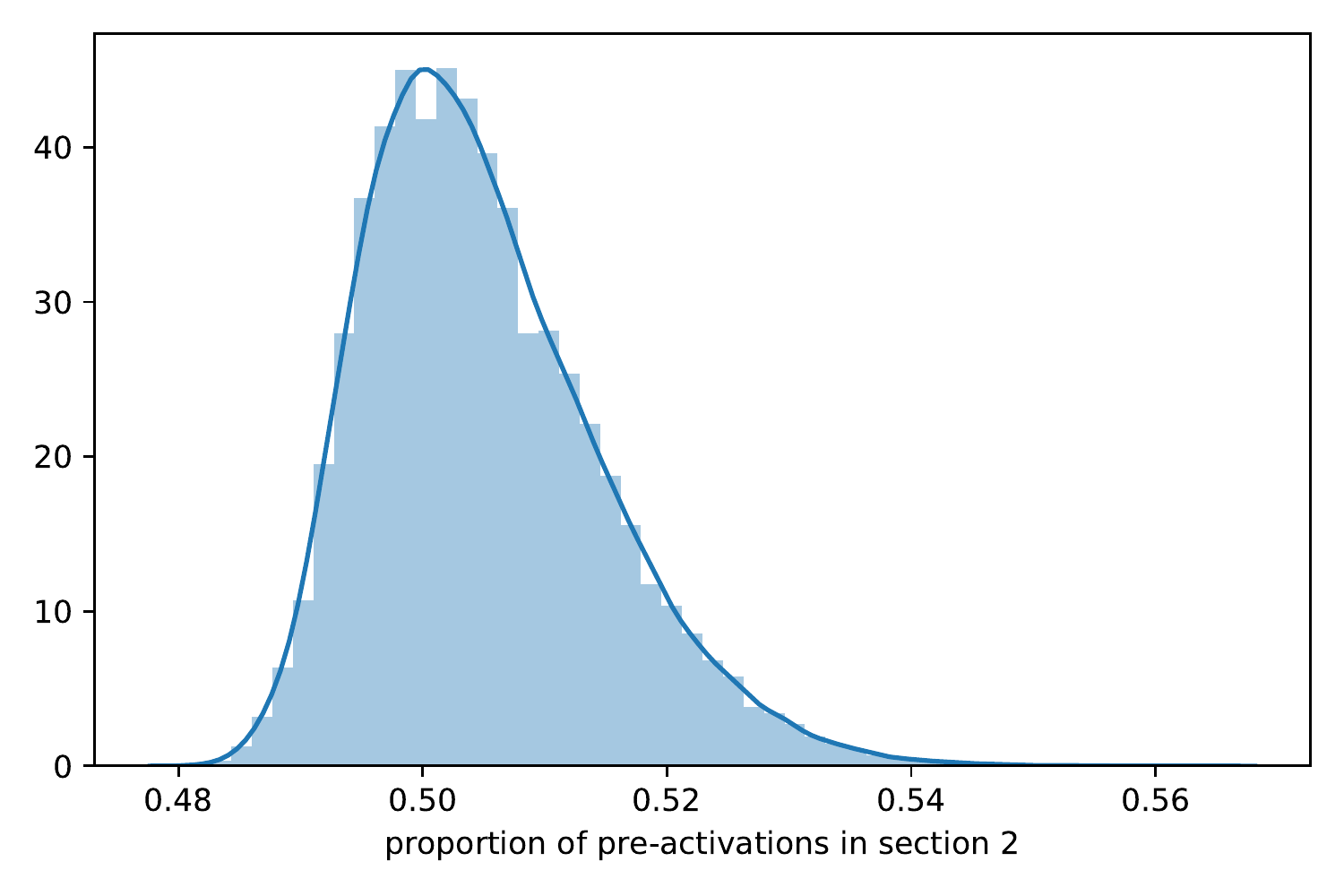}
            \caption{LeNet, MNIST data.} 
            \label{fig:probe_agg_neuron_random_weights_lenet_mnist}
        \end{subfigure}
        \caption{Experimental distribution of $\bar{R}_2$ (neuron averaging; each sample is a single datum) for random MLP and LeNet \texttt{ReLU} networks, and i.i.d. normal and MNIST data. The blue line is a kernel density estimation fit.} 
        \label{fig:probe_agg_neuron_random_weights}
    \end{figure*}
  \begin{figure*}[p]
        \centering
        \begin{subfigure}[b]{0.236\textwidth}
            \centering
            \includegraphics[width=\textwidth]{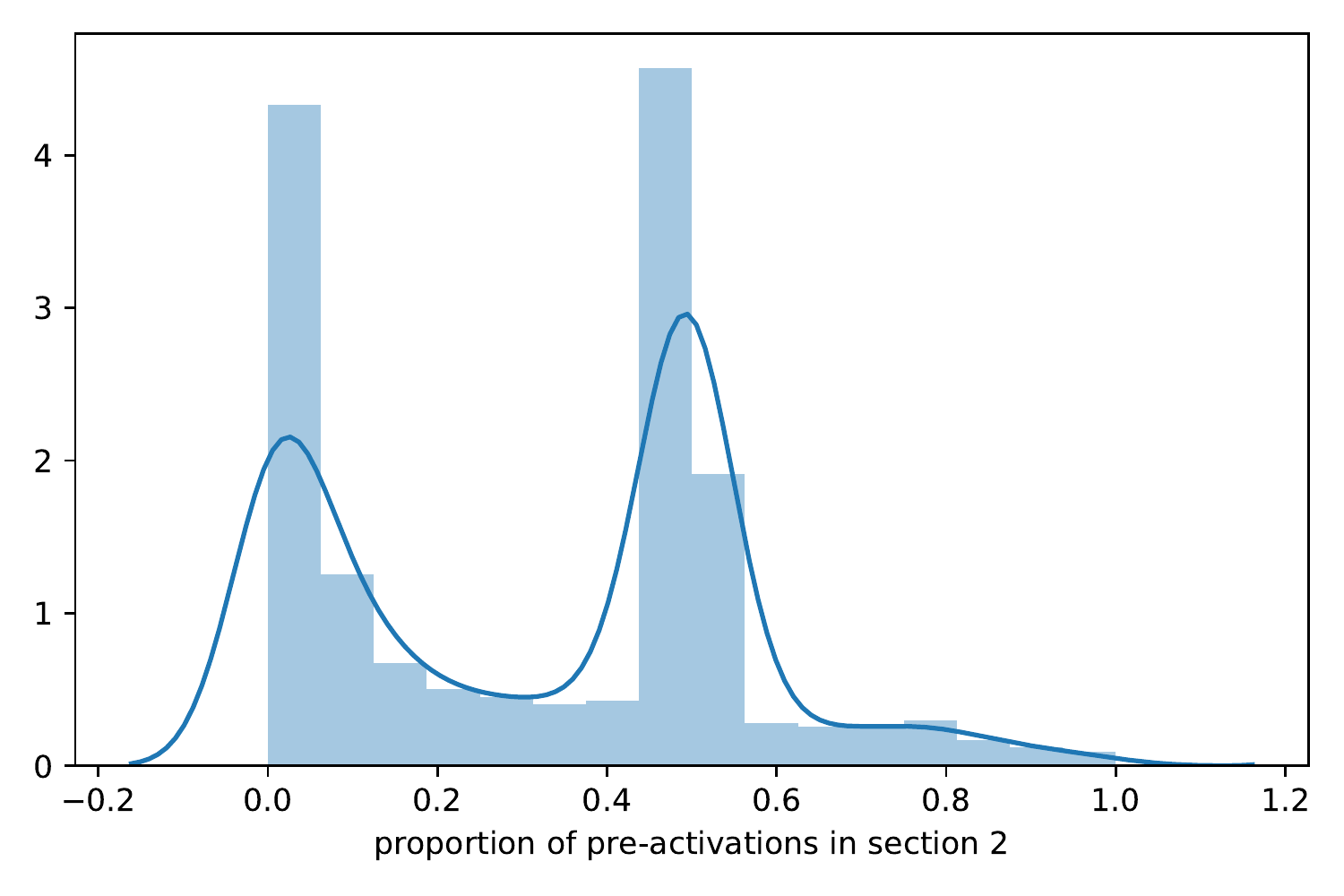}
            \caption{MLP, i.i.d. normal data.} 
            \label{fig:probe_agg_data_trained_weights_mlp_iid}
        \end{subfigure}
        \begin{subfigure}[b]{0.236\textwidth}
            \centering
            \includegraphics[width=\textwidth]{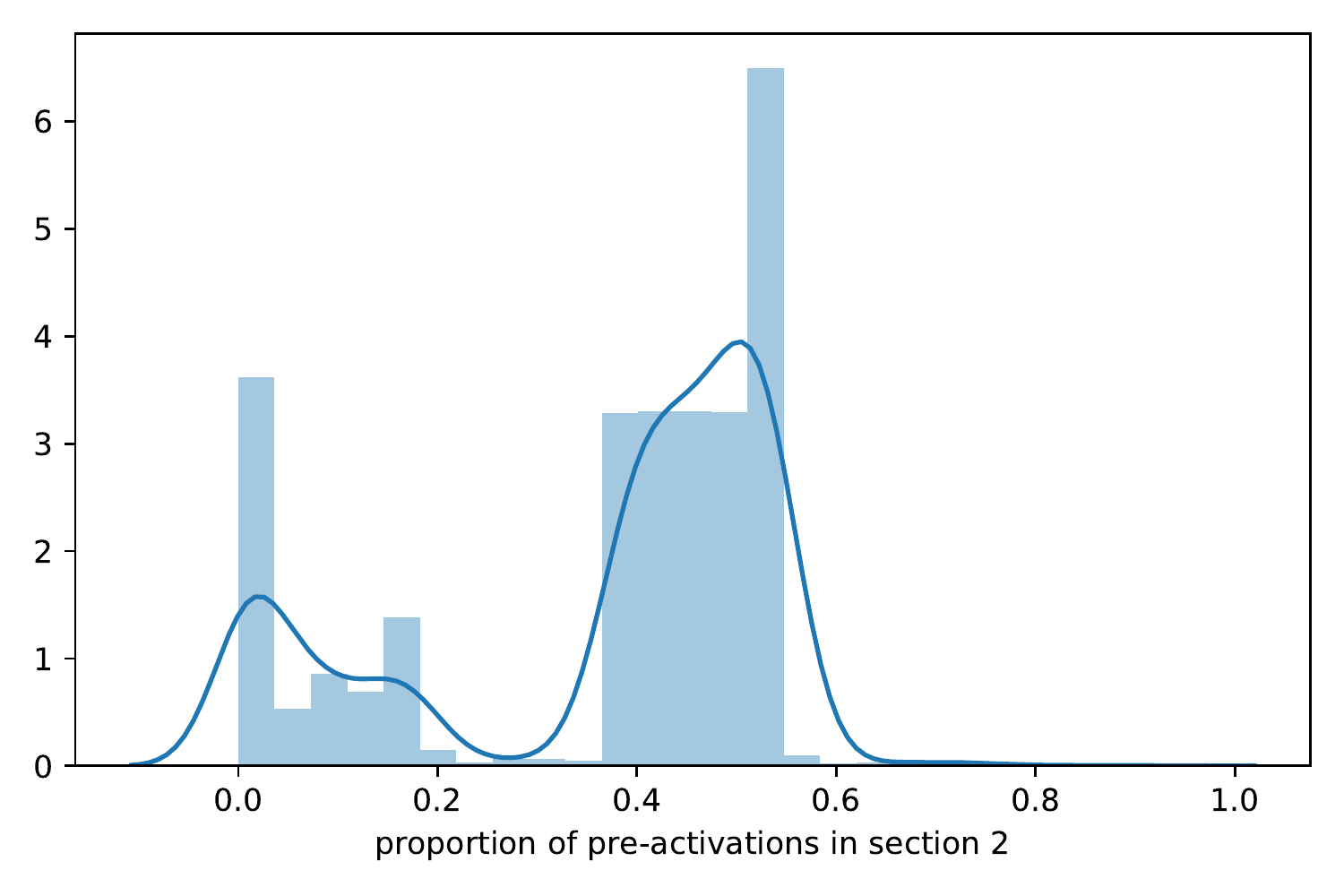}
            \caption{LeNet, i.i.d. normal data.} 
            \label{fig:probe_agg_data_trained_weights_lenet_iid}
        \end{subfigure}
        \begin{subfigure}[b]{0.236\textwidth}
            \centering
            \includegraphics[width=\textwidth]{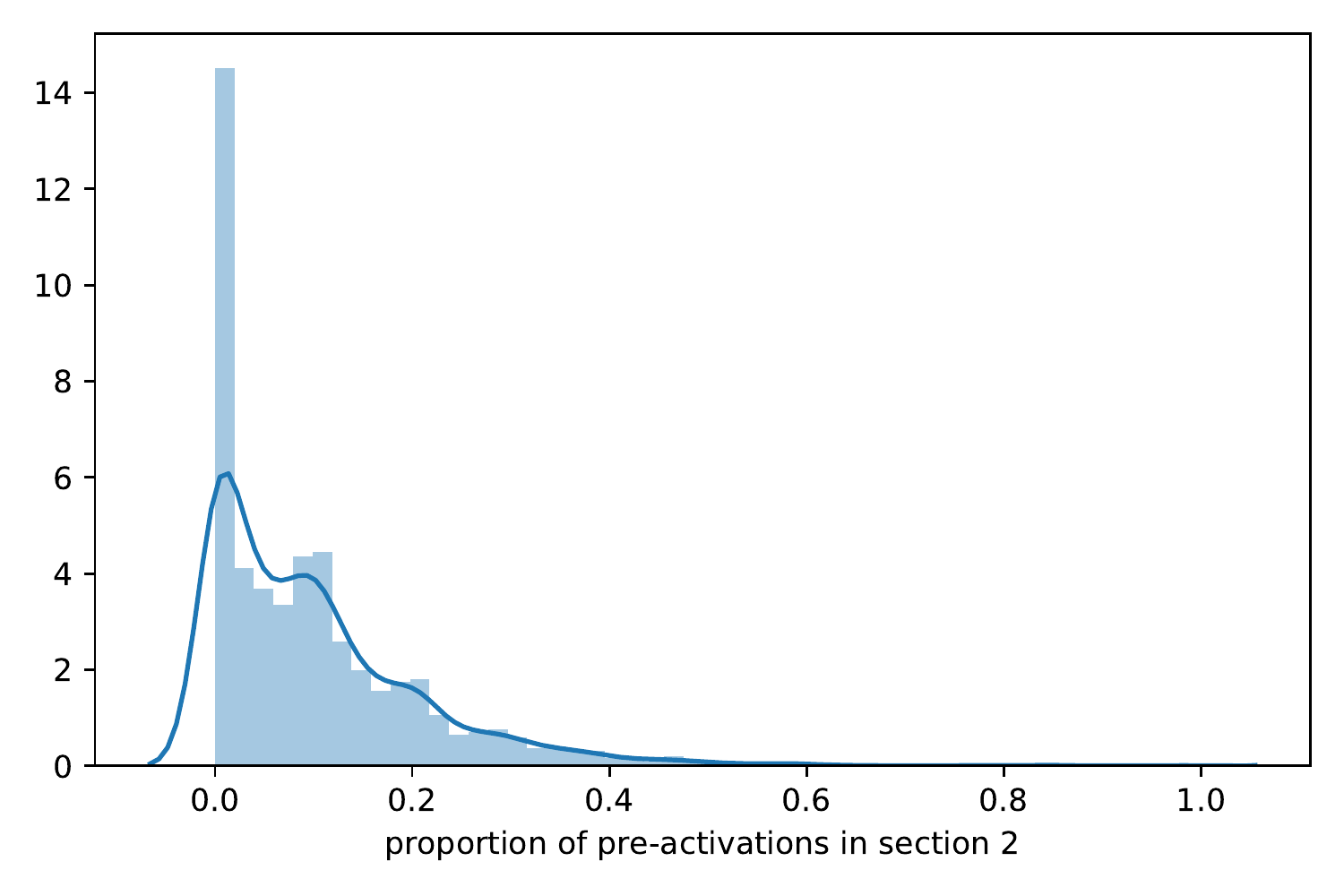}
            \caption{MLP, MNIST data.} 
            \label{fig:probe_agg_data_trained_weights_mlp_mnist}
        \end{subfigure}
        \begin{subfigure}[b]{0.236\textwidth}
            \centering
            \includegraphics[width=\textwidth]{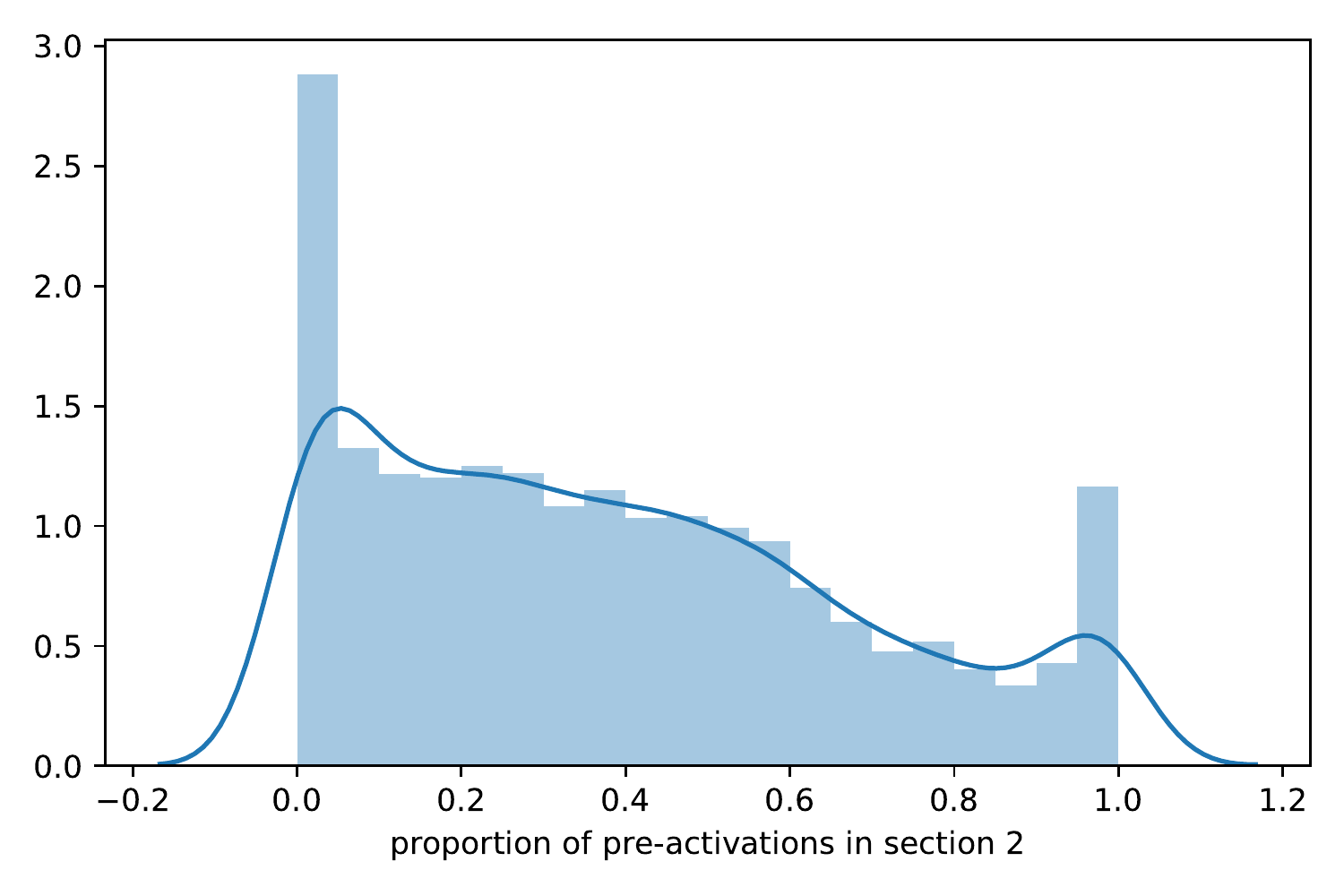}
            \caption{LeNet, MNIST data.} 
            \label{fig:probe_agg_data_trained_weights_lenet_mnist}
        \end{subfigure}
        \caption{Experimental distribution of $R_2$ (data averaging; each sample is a single neuron) for MLP and LeNet \texttt{ReLU} networks trained to high validation accuracy on MNIST, and evaluated on i.i.d. normal and MNIST data. The blue line is a kernel density estimation fit.} 
        \label{fig:probe_agg_data_trained_weights}
    \end{figure*}
  \begin{figure*}[p]
        \centering
        \begin{subfigure}[b]{0.236\textwidth}
            \centering
            \includegraphics[width=\textwidth]{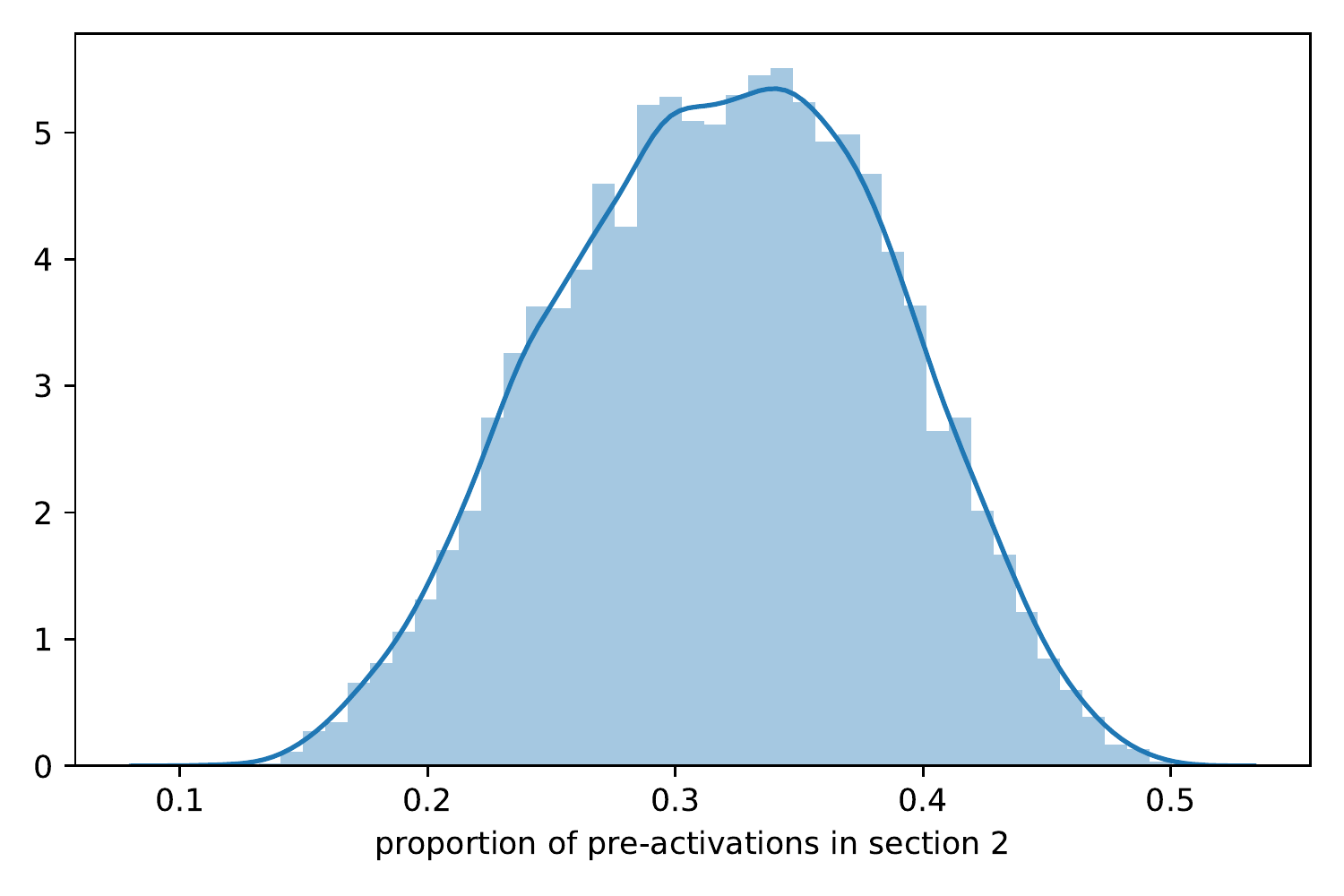}
            \caption{MLP, i.i.d. normal data.} 
            \label{fig:probe_agg_neuron_trained_weights_mlp_iid}
        \end{subfigure}
        \begin{subfigure}[b]{0.236\textwidth}
            \centering
            \includegraphics[width=\textwidth]{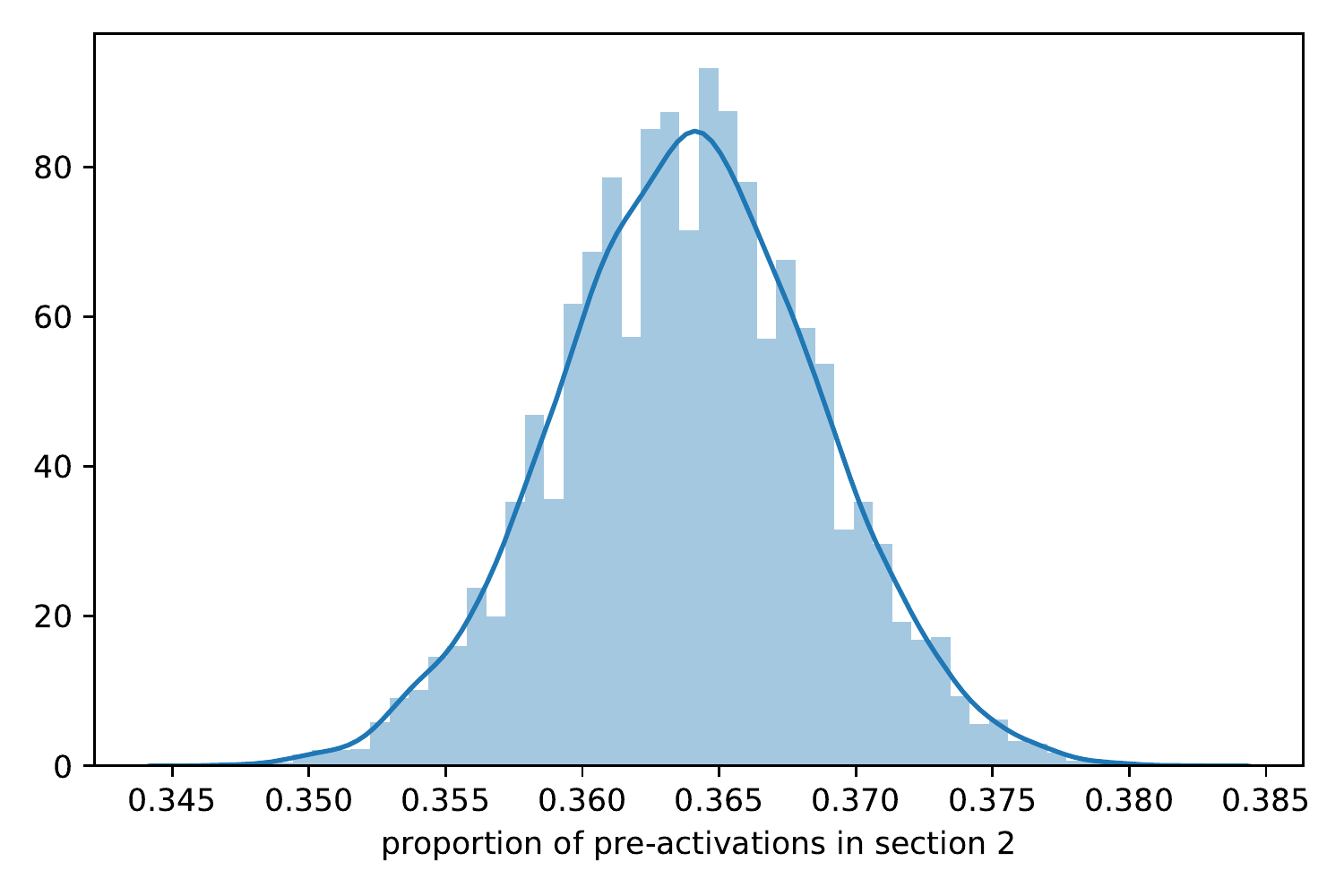}
            \caption{LeNet, i.i.d. normal data.} 
            \label{fig:probe_agg_neuron_trained_weights_lenet_iid}
        \end{subfigure}
        \begin{subfigure}[b]{0.236\textwidth}
            \centering
            \includegraphics[width=\textwidth]{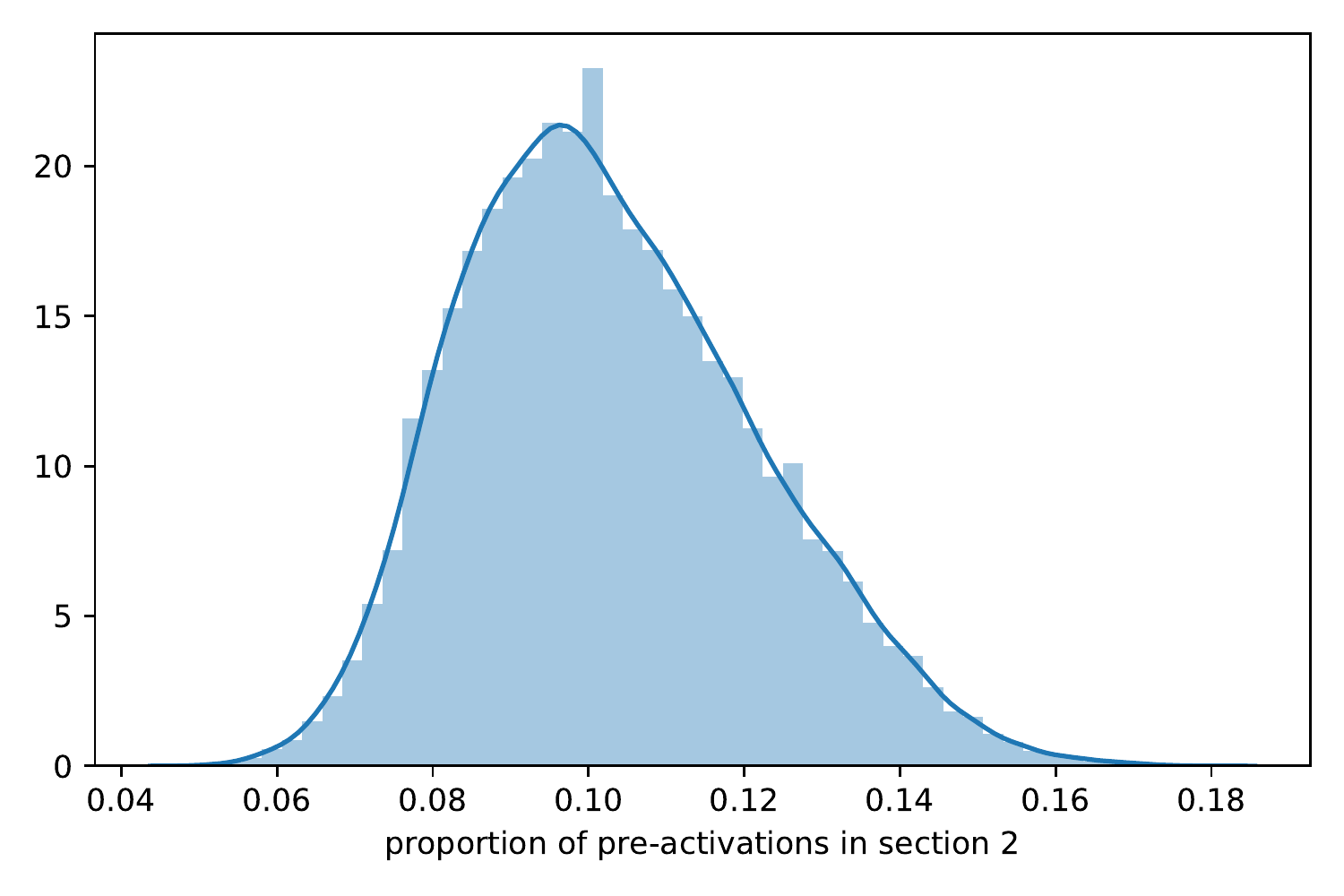}
            \caption{MLP, MNIST data.} 
            \label{fig:probe_agg_neuron_trained_weights_mlp_mnist}
        \end{subfigure}
        \begin{subfigure}[b]{0.236\textwidth}
            \centering
            \includegraphics[width=\textwidth]{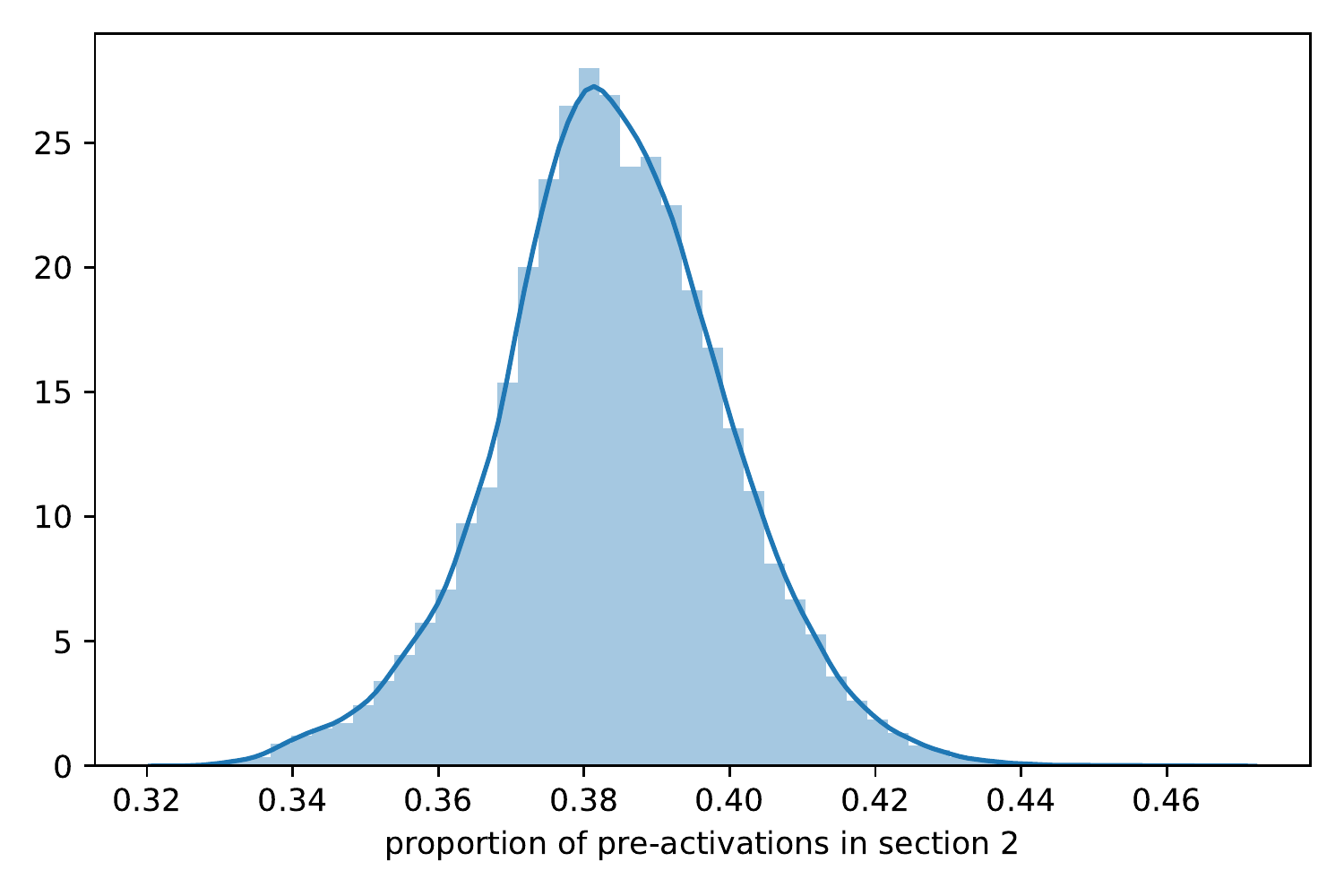}
            \caption{LeNet, MNIST data.} 
            \label{fig:probe_agg_neuron_trained_weights_lenet_mnist}
        \end{subfigure}
        \caption{Experimental distribution of $\bar{R}_2$ (neuron averaging; each sample is a single datum) for MLP and LeNet \texttt{ReLU} networks trained to high validation accuracy on MNIST, and evaluated on i.i.d. normal and MNIST data. The  blue line is a kernel density estimation fit.} 
        \label{fig:probe_agg_neuron_trained_weights}
    \end{figure*}

  \begin{figure*}
        \centering
        \begin{subfigure}[b]{0.236\textwidth}
            \centering
            \includegraphics[width=\textwidth]{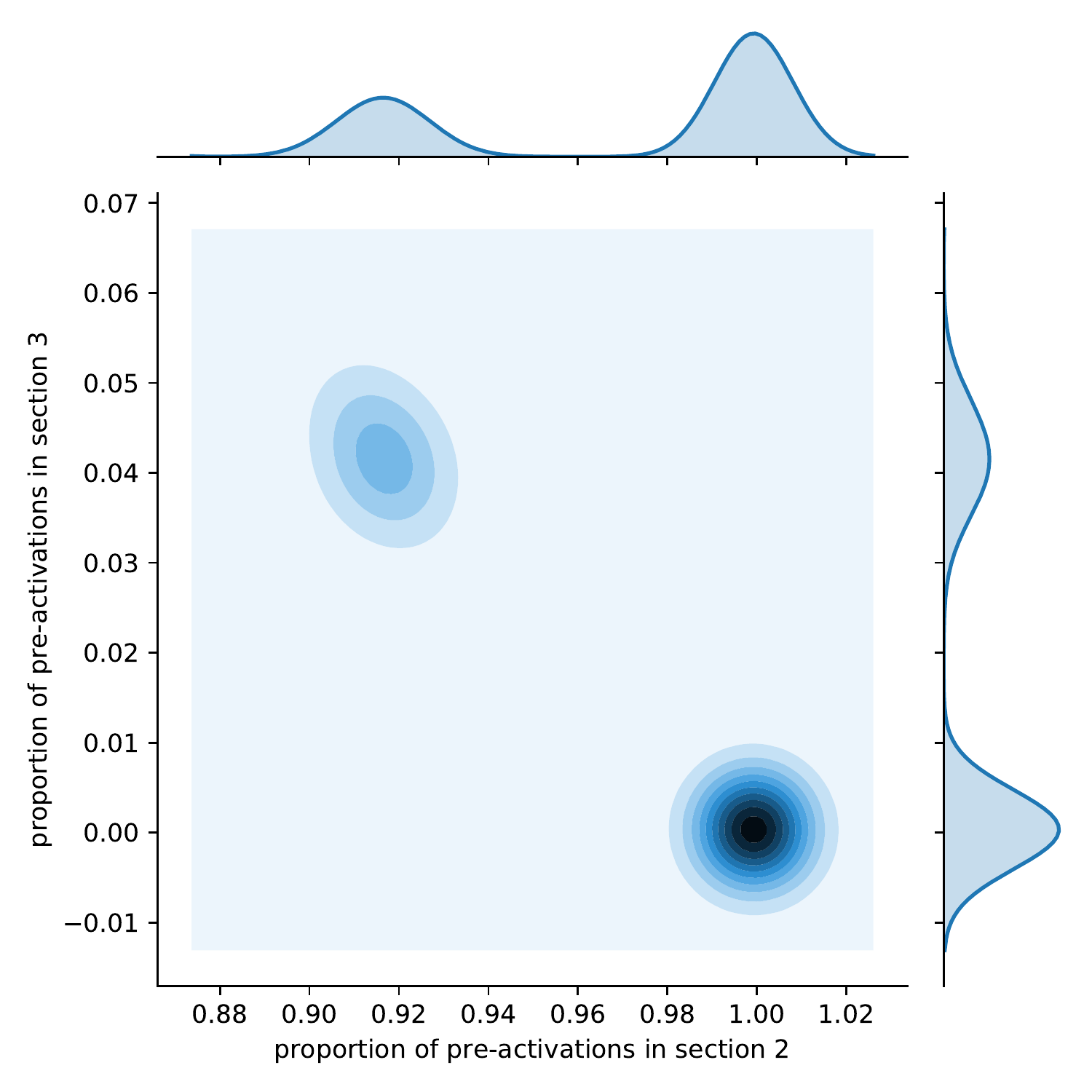}
            \caption{MLP, i.i.d. normal data.} 
        \end{subfigure}
        \begin{subfigure}[b]{0.236\textwidth}
            \centering
            \includegraphics[width=\textwidth]{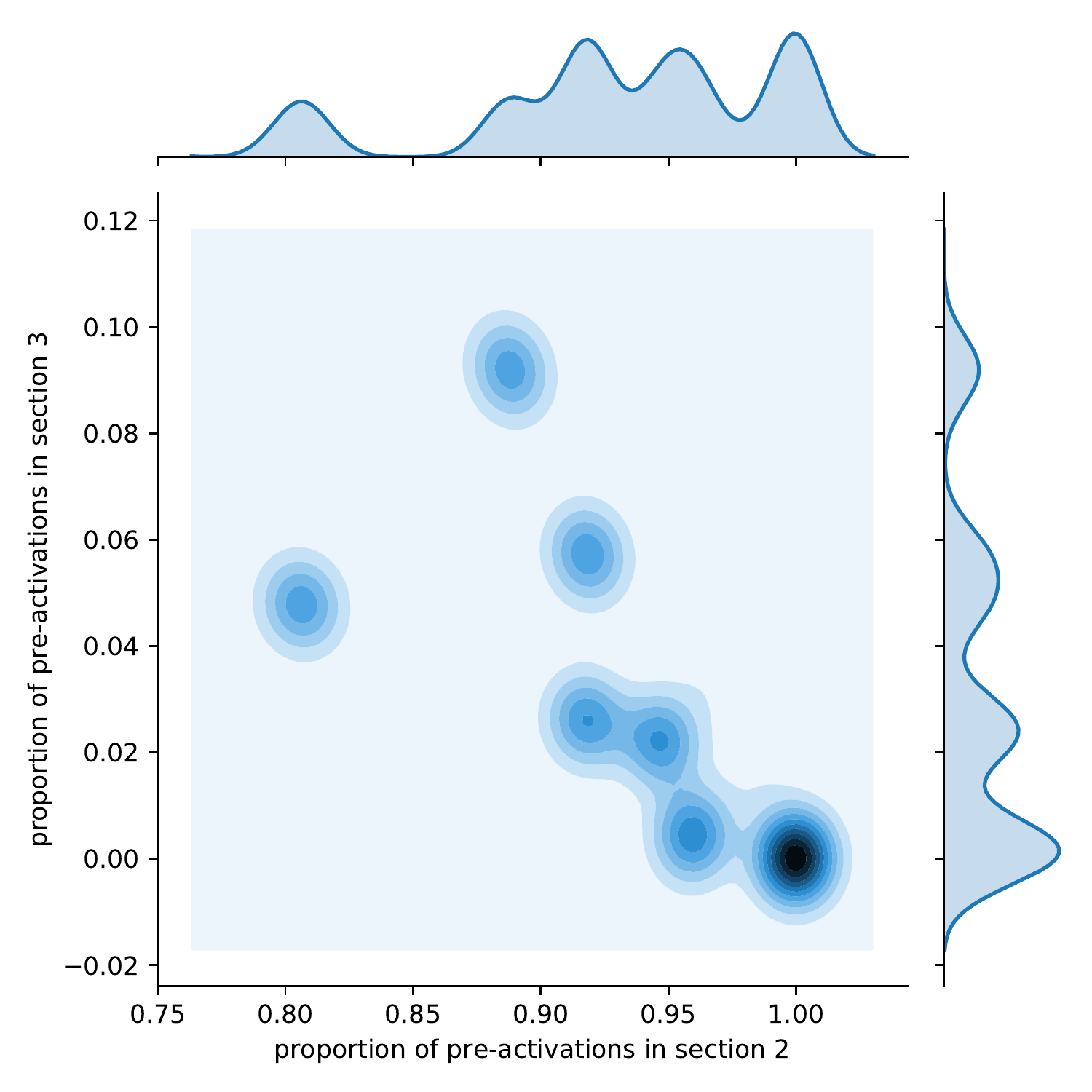}
            \caption{LeNet, i.i.d. normal data.} 
        \end{subfigure}
        \begin{subfigure}[b]{0.236\textwidth}
            \centering
            \includegraphics[width=\textwidth]{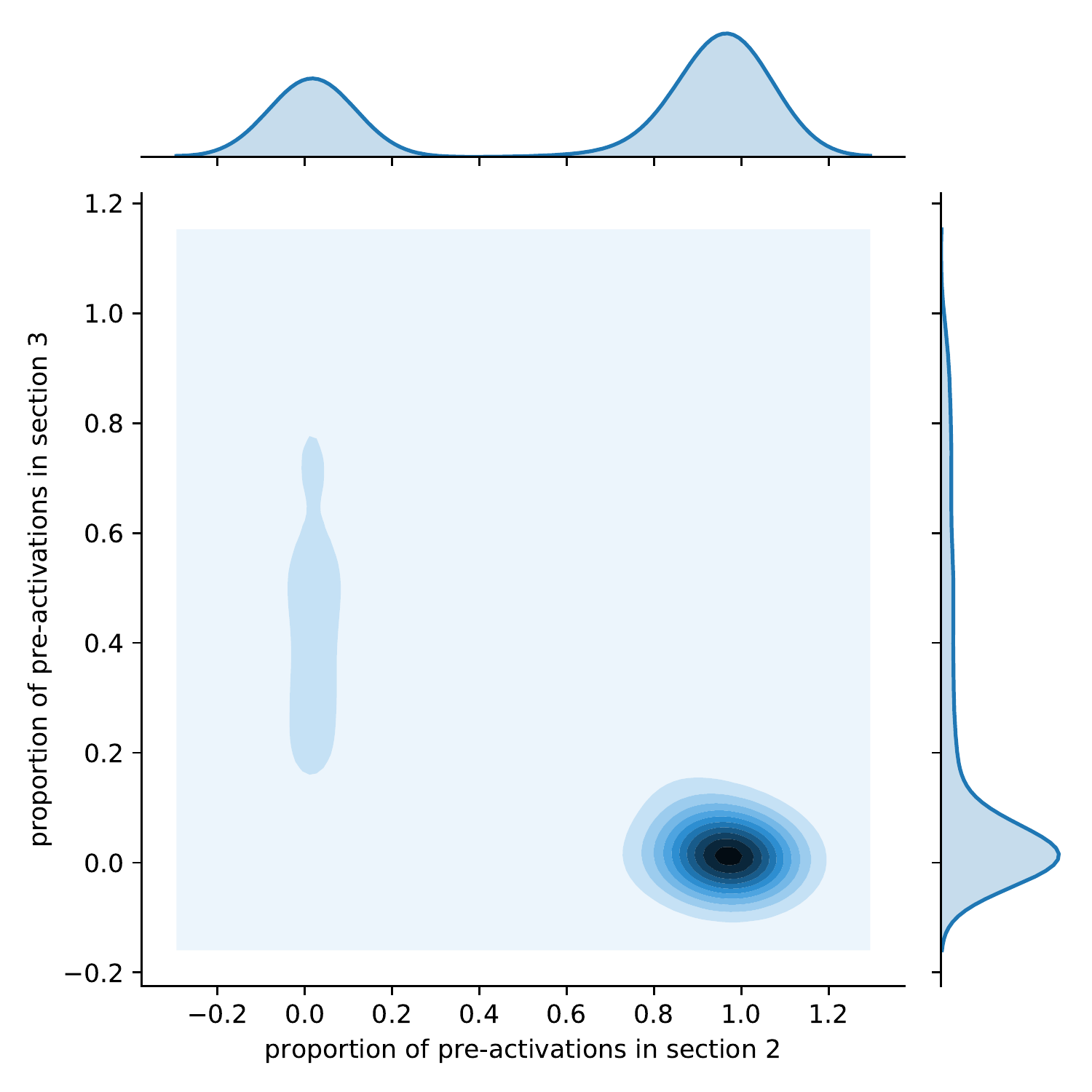}
            \caption{MLP, MNIST data.} 
        \end{subfigure}
        \begin{subfigure}[b]{0.236\textwidth}
            \centering
            \includegraphics[width=\textwidth]{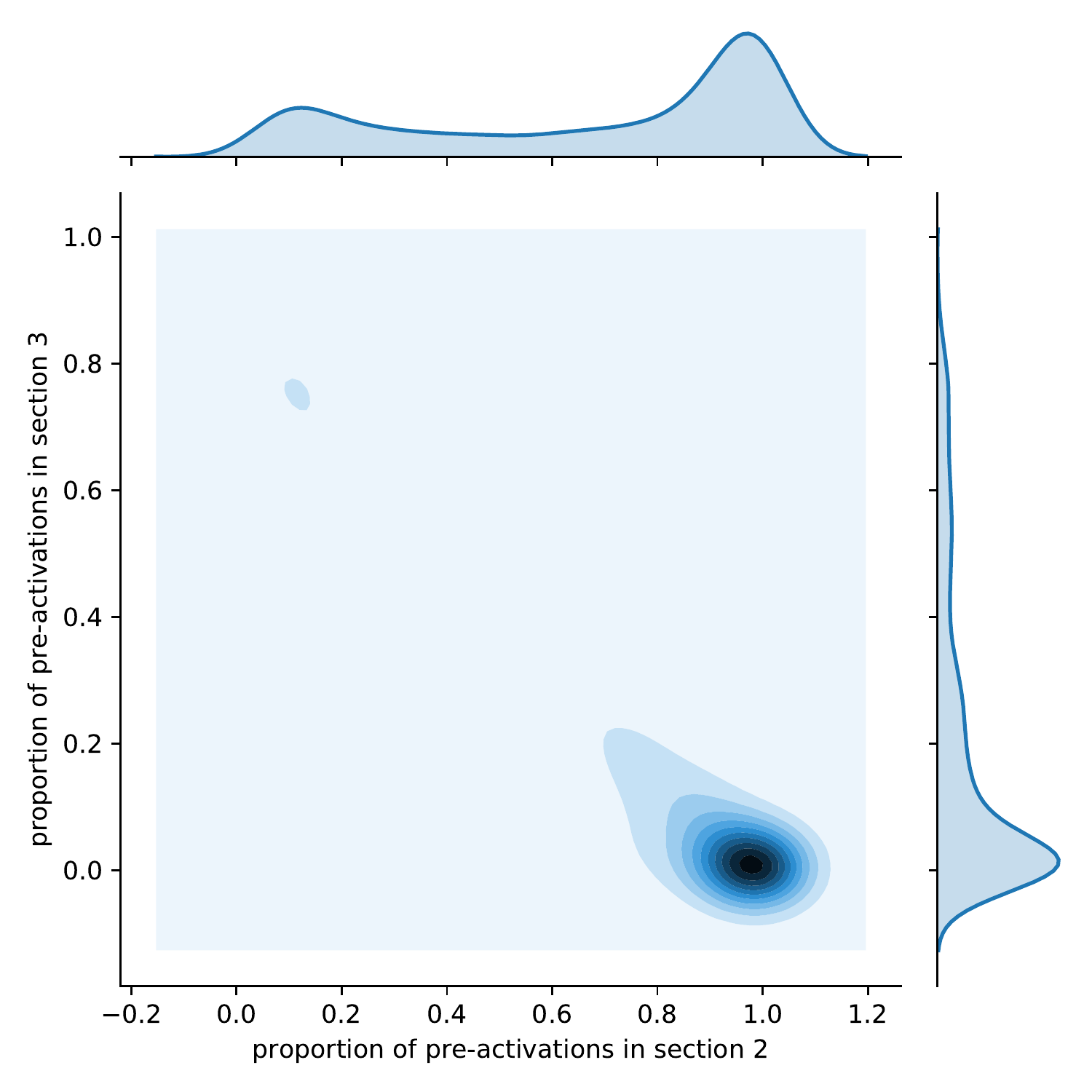}
            \caption{LeNet, MNIST data.} 
        \end{subfigure}
        \caption{Experimental distribution of $(R_2,R_3)$ (data averaging; each sample is a single neuron) for random MLP and LeNet \texttt{HardTanh} networks, and i.i.d. normal and MNIST data. The plots show 2d kernel density estimation fits of the joint and 1d fits of the marginals.} 
                \label{fig:probe_agg_data_random_weights_tanh}
    \end{figure*}
  \begin{figure*}
        \centering
        \begin{subfigure}[b]{0.236\textwidth}
            \centering
            \includegraphics[width=\textwidth]{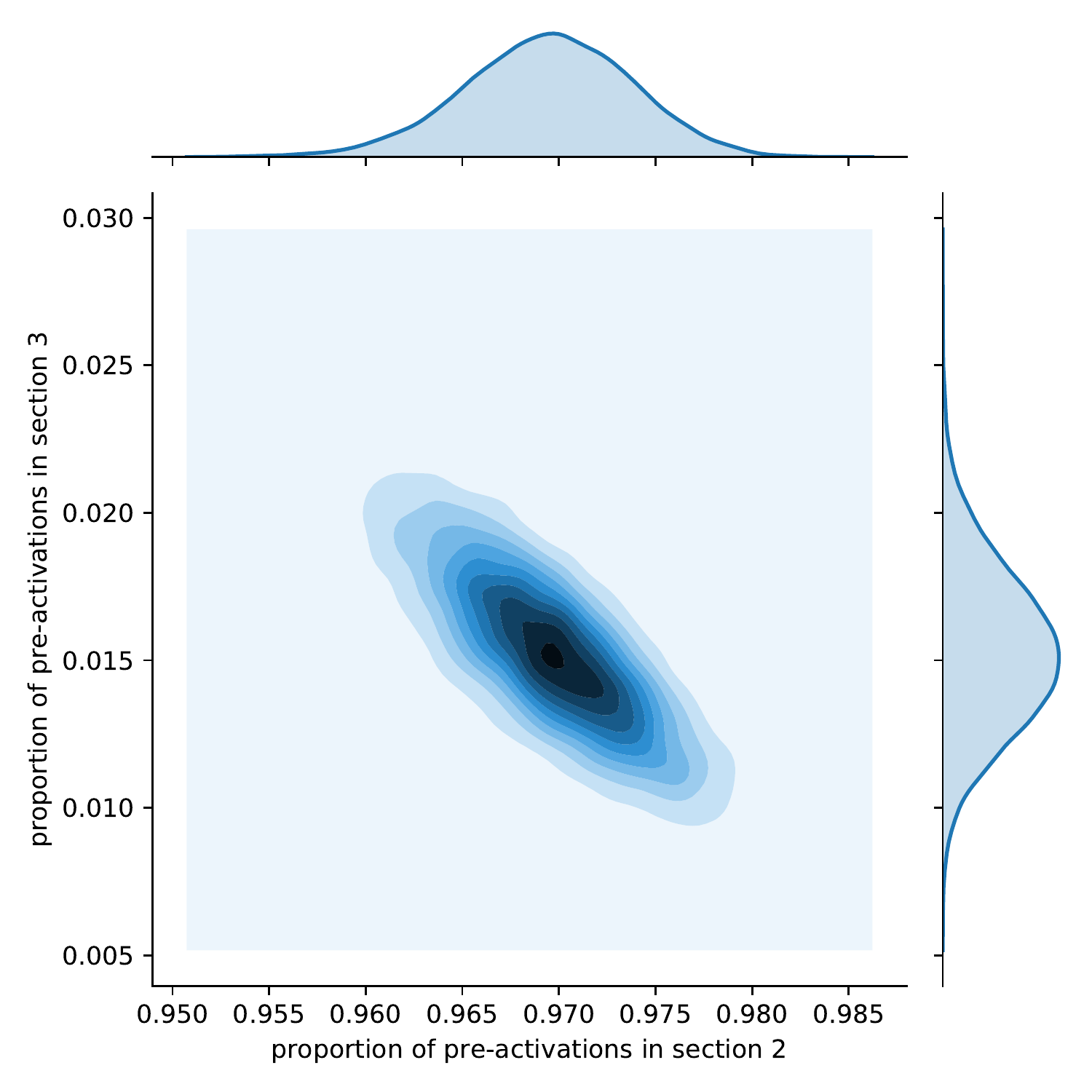}
            \caption{MLP, i.i.d. normal data.} 
                    \label{fig:probe_agg_neuron_random_weights_tanh_mlp_iid}
        \end{subfigure}
        \begin{subfigure}[b]{0.236\textwidth}
            \centering
            \includegraphics[width=\textwidth]{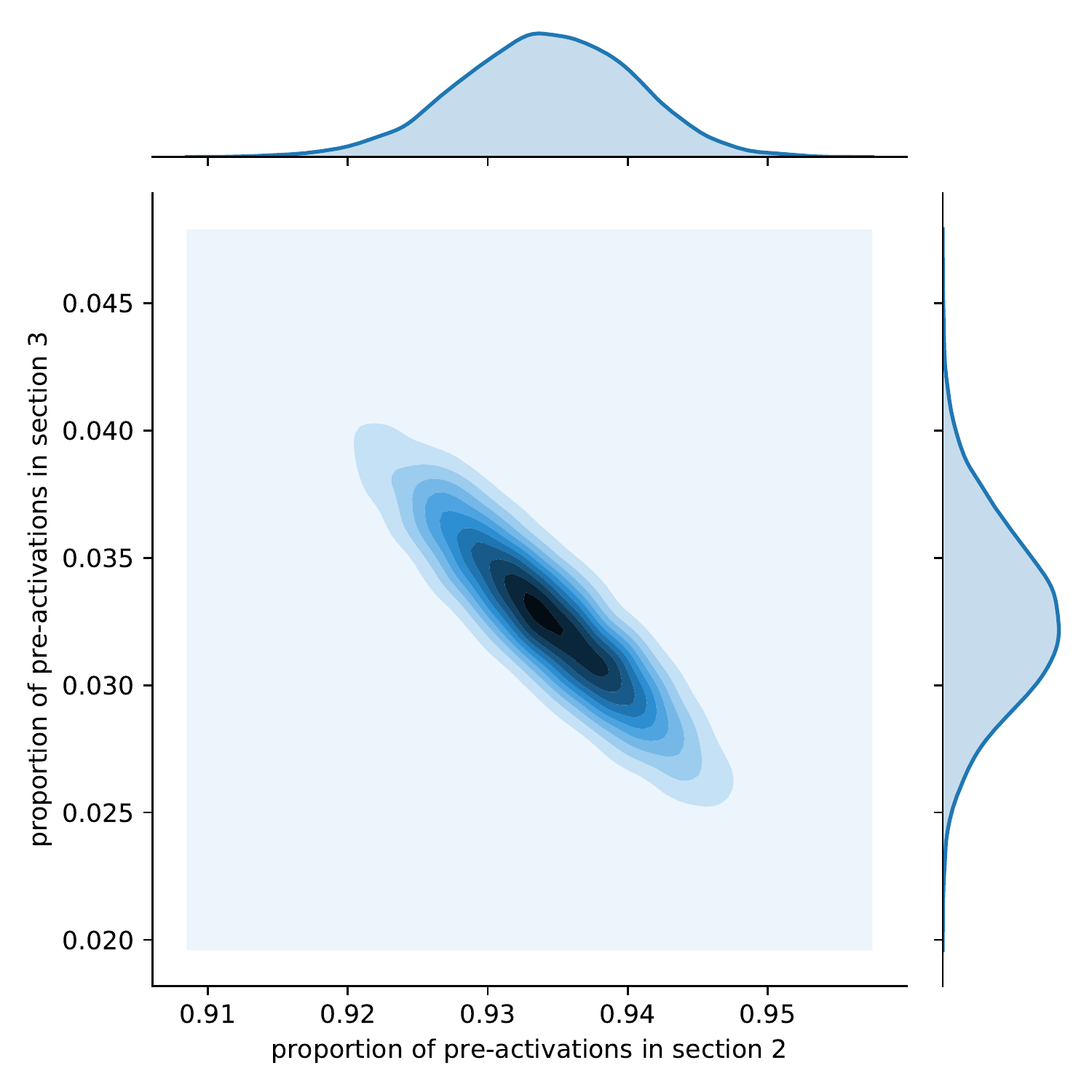}
            \caption{LeNet, i.i.d. normal data.} 
        \label{fig:probe_agg_neuron_random_weights_tanh_lenet_iid}
        \end{subfigure}
        \begin{subfigure}[b]{0.236\textwidth}
            \centering
            \includegraphics[width=\textwidth]{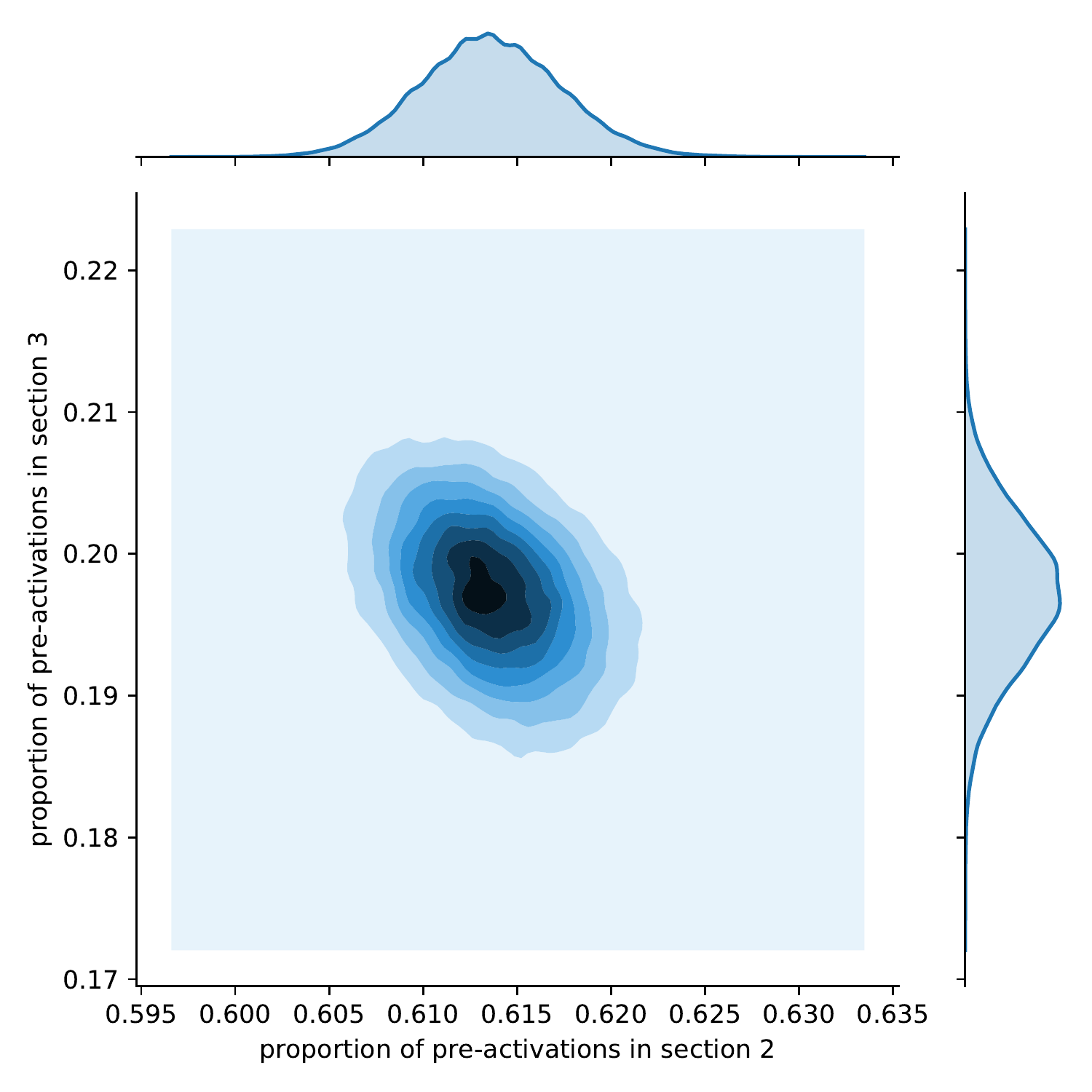}
            \caption{MLP, MNIST data.} 
        \label{fig:probe_agg_neuron_random_weights_tanh_nlp_mnist}
        \end{subfigure}
        \begin{subfigure}[b]{0.236\textwidth}
            \centering
            \includegraphics[width=\textwidth]{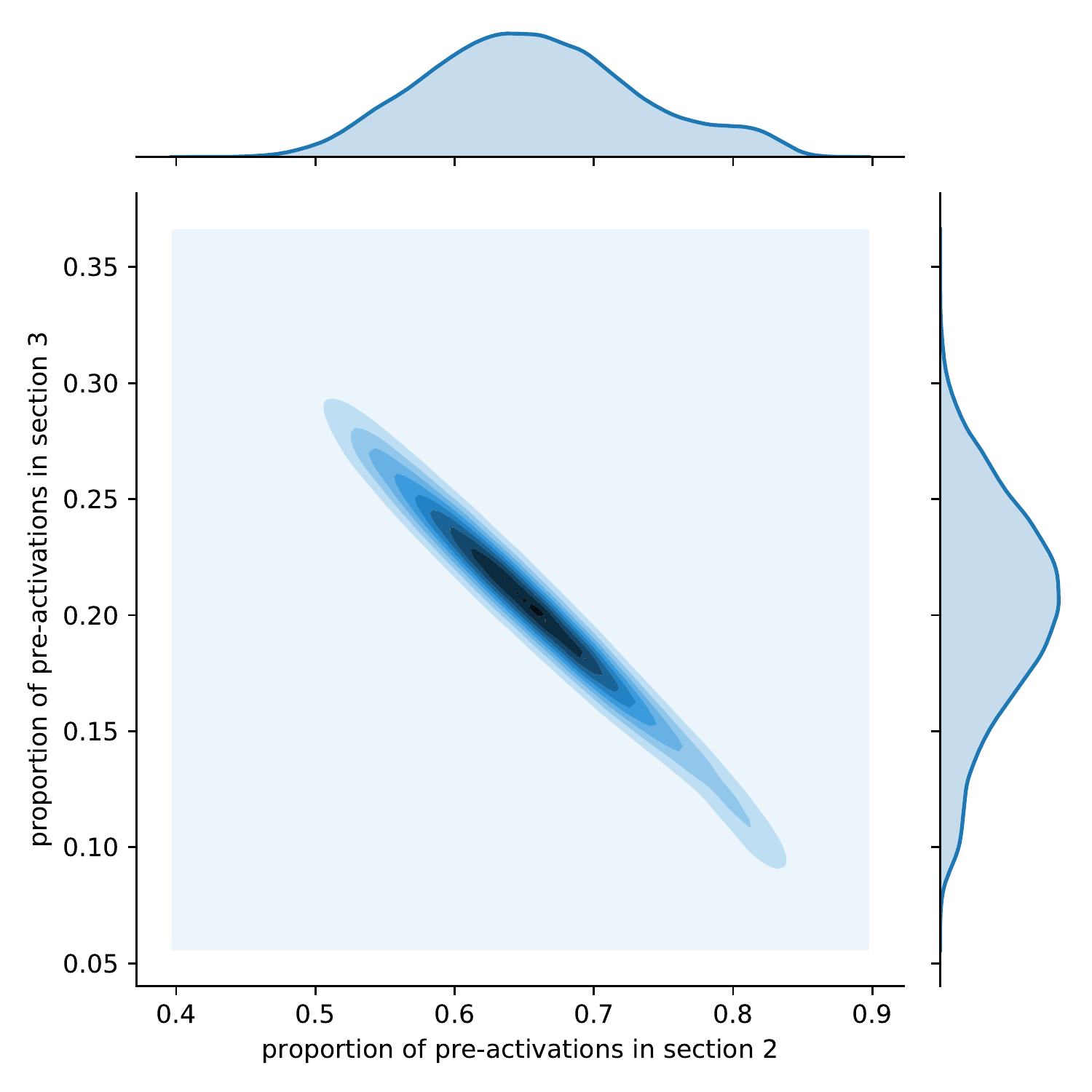}
            \caption{LeNet, MNIST data.} 
        \label{fig:probe_agg_neuron_random_weights_tanh_lenet_mnist}
        \end{subfigure}
        \caption{Experimental distribution of $(\bar{R}_2, \bar{R}_3)$ (neuron averaging; each sample is a single datum) for random \texttt{HardTanh} MLP and LeNet networks, and i.i.d. normal and MNIST data. The plots show 2d kernel density estimation fits of the joint and 1d fits of the marginals.} 
        \label{fig:probe_agg_neuron_random_weights_tanh}
    \end{figure*}
  \begin{figure*}
        \centering
        \begin{subfigure}[b]{0.236\textwidth}
            \centering
            \includegraphics[width=\textwidth]{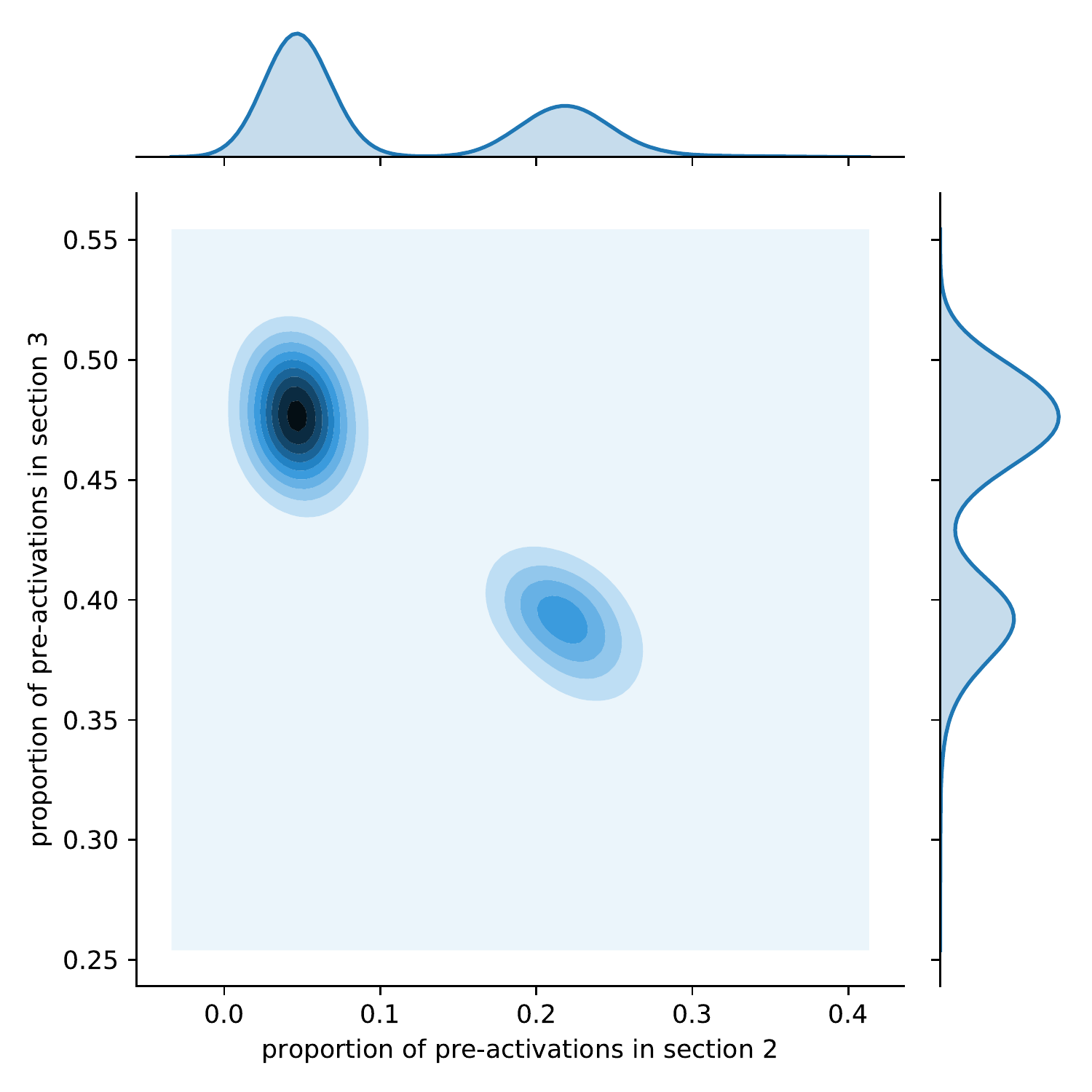}
            \caption{MLP, i.i.d. normal data.} 
        \end{subfigure}
        \begin{subfigure}[b]{0.236\textwidth}
            \centering
            \includegraphics[width=\textwidth]{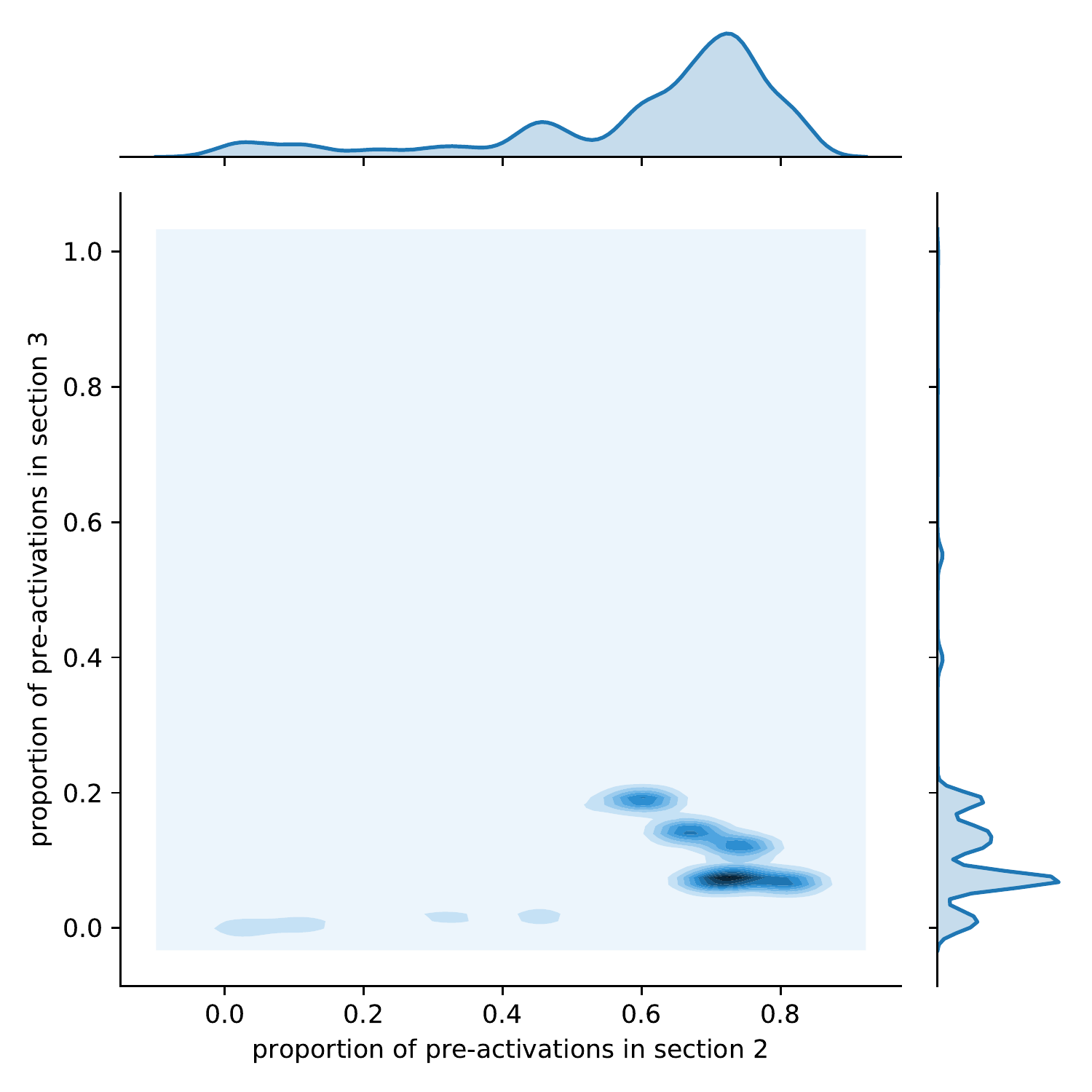}
            \caption{LeNet, i.i.d. normal data.} 
        \end{subfigure}
        \begin{subfigure}[b]{0.236\textwidth}
            \centering
            \includegraphics[width=\textwidth]{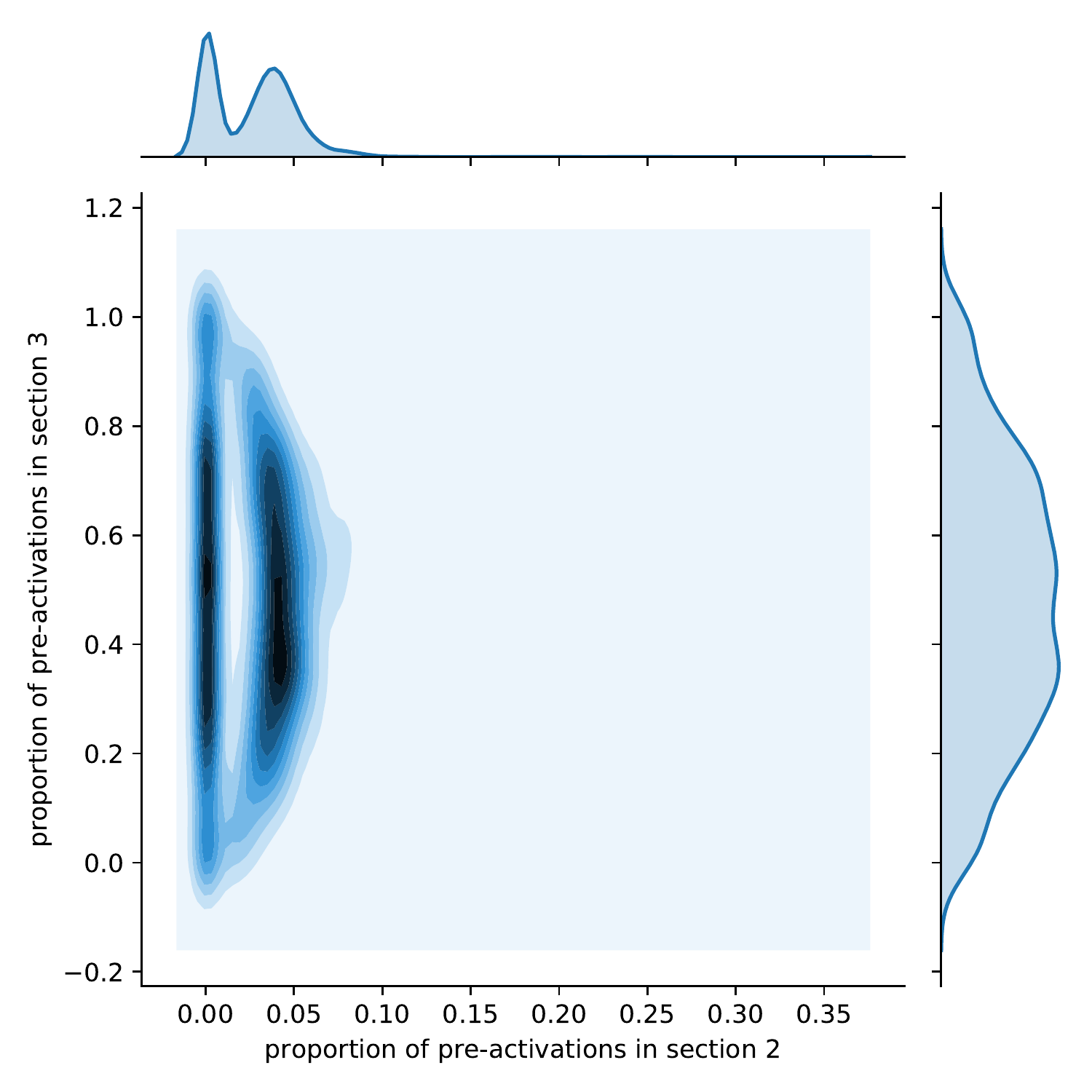}
            \caption{MLP, MNIST data.} 
        \end{subfigure}
        \begin{subfigure}[b]{0.236\textwidth}
            \centering
            \includegraphics[width=\textwidth]{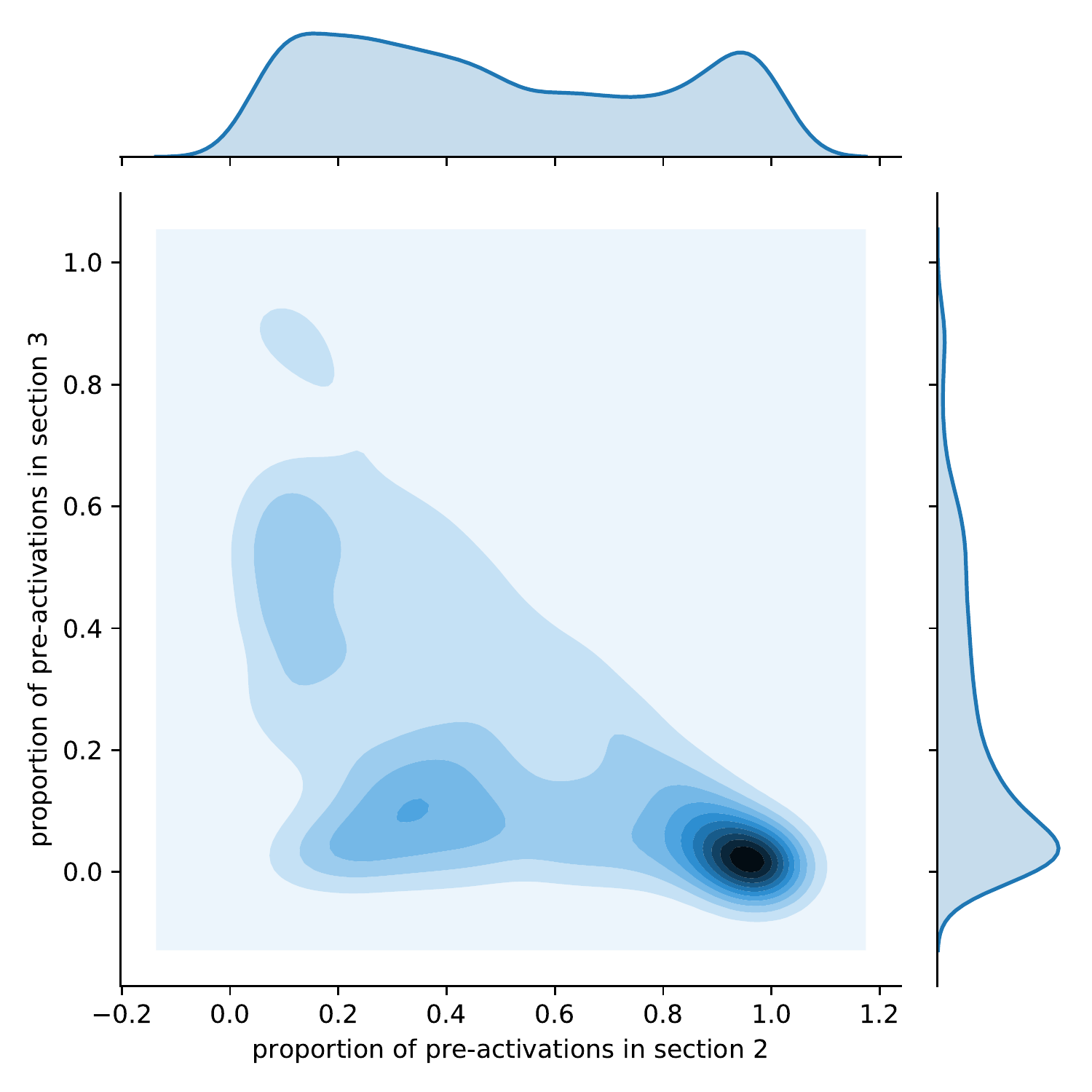}
            \caption{LeNet, MNIST data.} 
        \end{subfigure}
        \caption{Experimental distribution of $(R_2, R_3)$ (data averaging; each sample is a single neuron) for MLP and LeNet \texttt{HardTanh} networks trained to high validation accuracy on MNIST, and evaluated on i.i.d. normal and MNIST data. The plots show 2d kernel density estimation fits of the joint and 1d fits of the marginals.} 
                \label{fig:probe_agg_data_trained_weights_tanh}
    \end{figure*}
  \begin{figure*}
        \centering
        \begin{subfigure}[b]{0.236\textwidth}
            \centering
            \includegraphics[width=\textwidth]{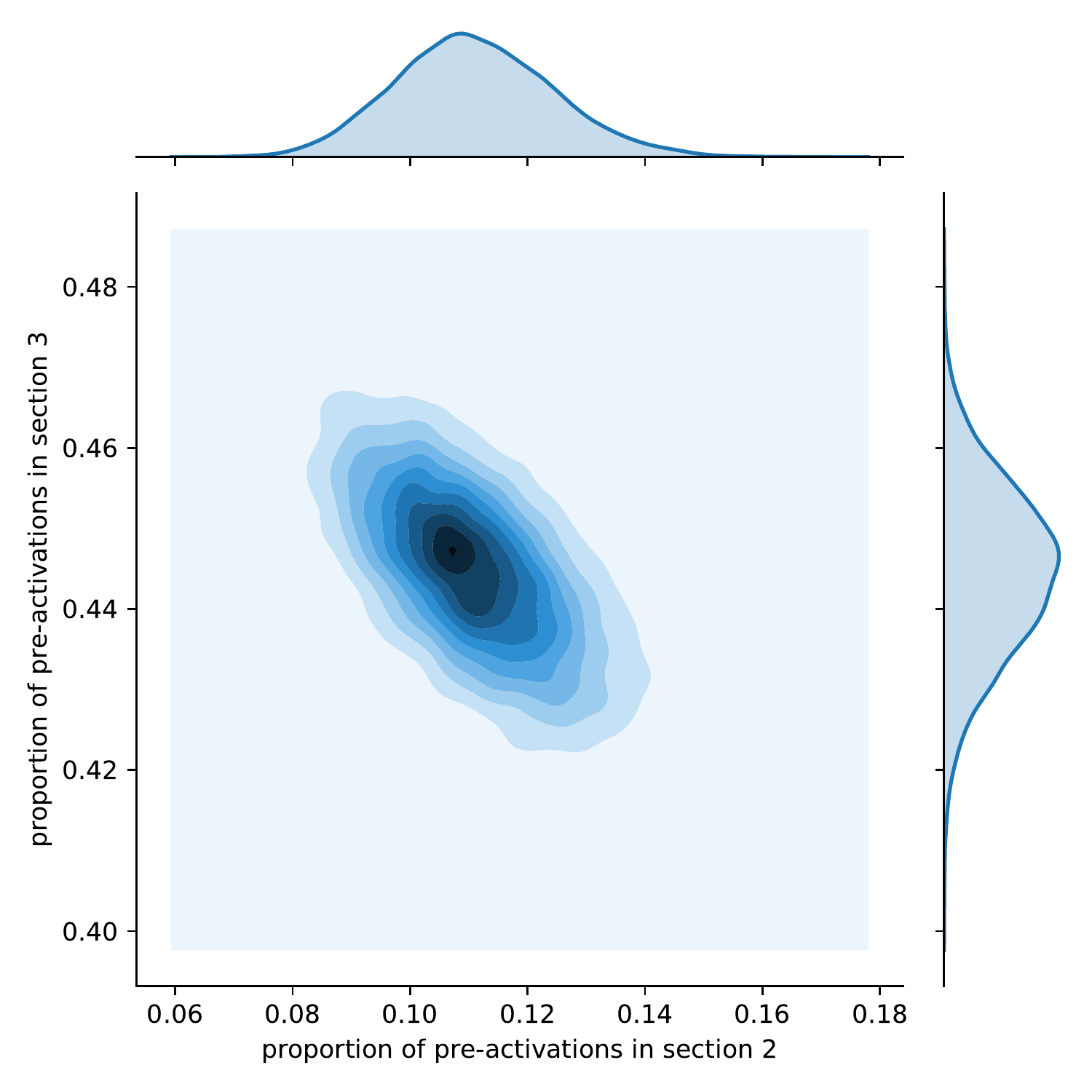}
            \caption{MLP, i.i.d. normal data.} 
        \end{subfigure}
        \begin{subfigure}[b]{0.236\textwidth}
            \centering
            \includegraphics[width=\textwidth]{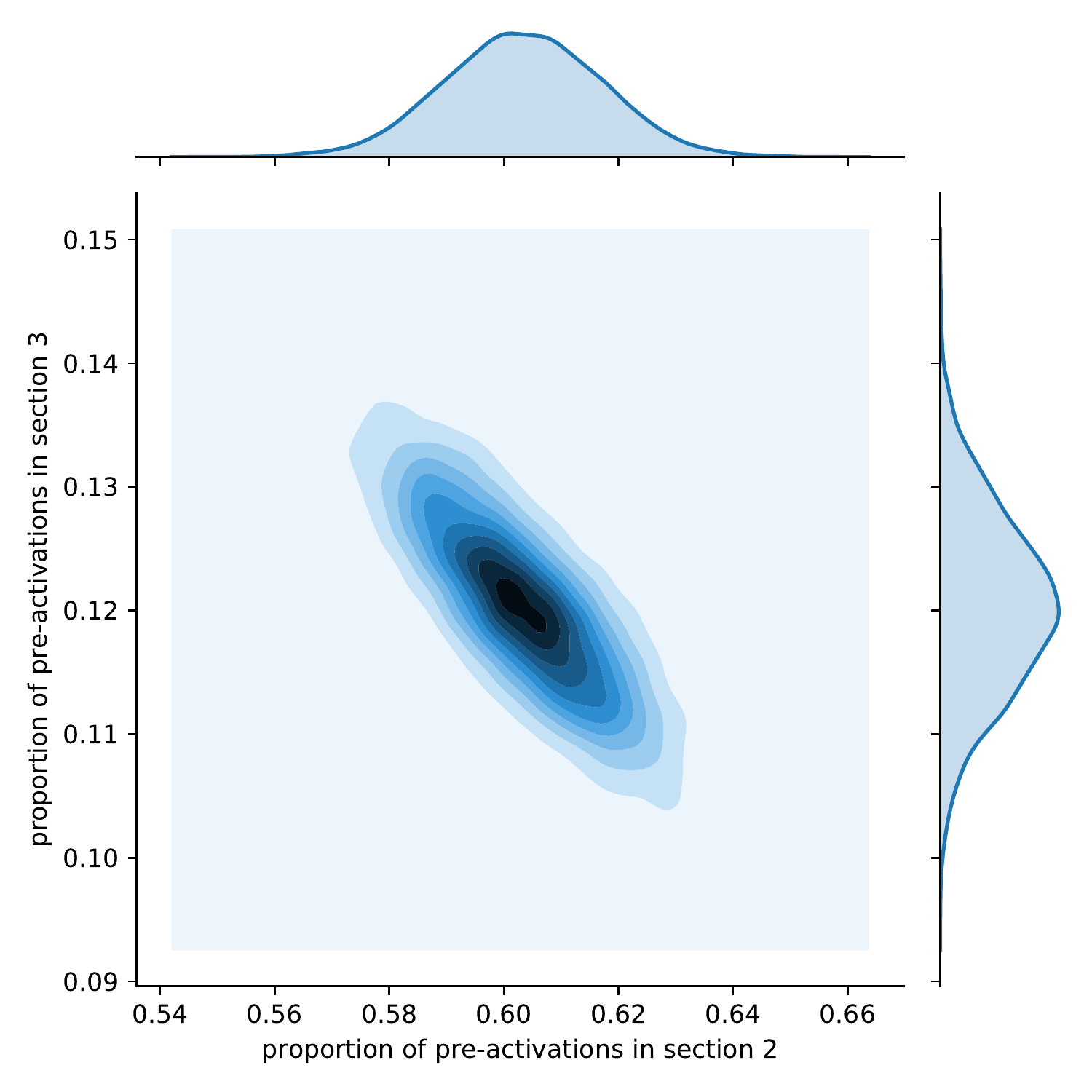}
            \caption{LeNet, i.i.d. normal data.} 
        \end{subfigure}
        \begin{subfigure}[b]{0.236\textwidth}
            \centering
            \includegraphics[width=\textwidth]{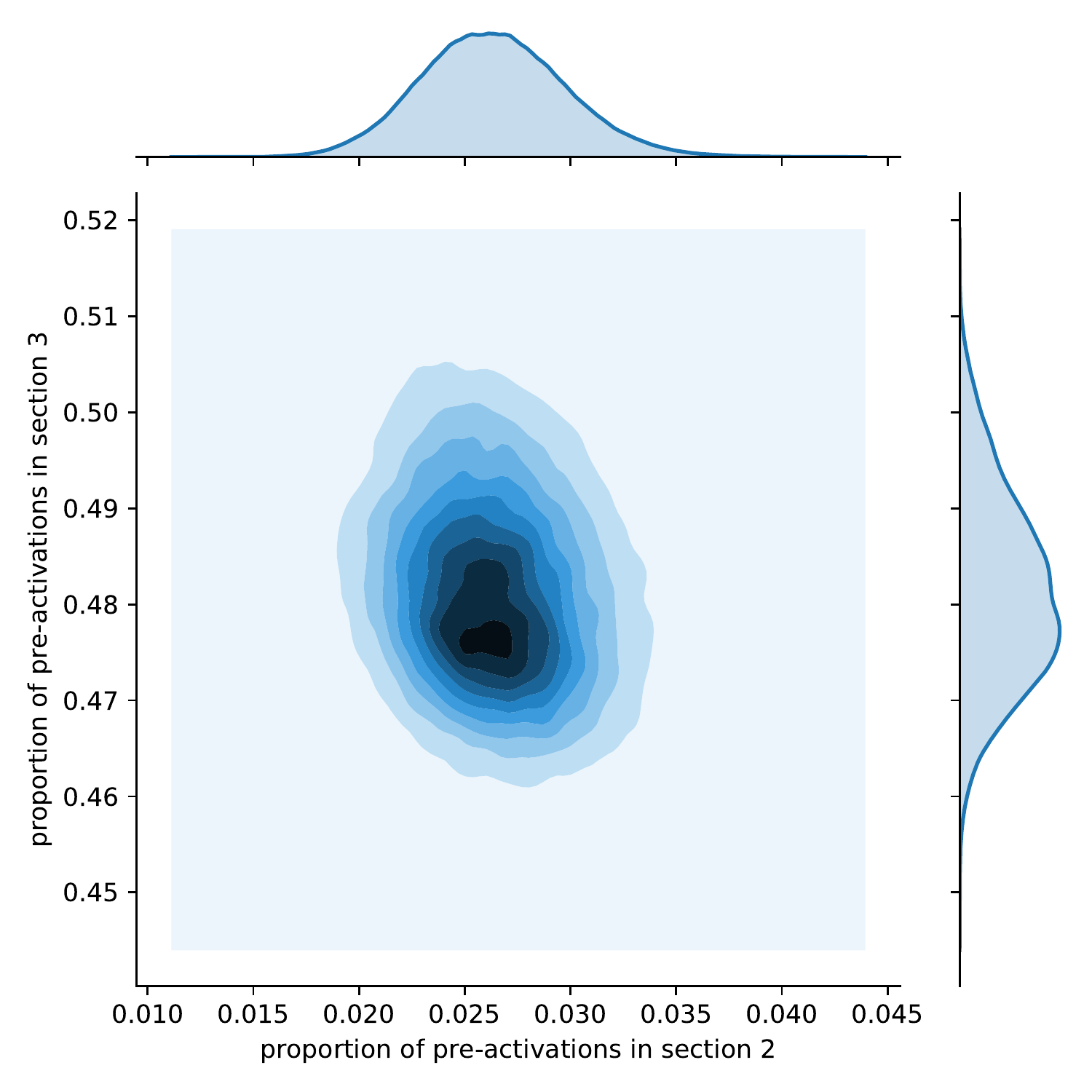}
            \caption{MLP, MNIST data.} 
        \end{subfigure}
        \begin{subfigure}[b]{0.236\textwidth}
            \centering
            \includegraphics[width=\textwidth]{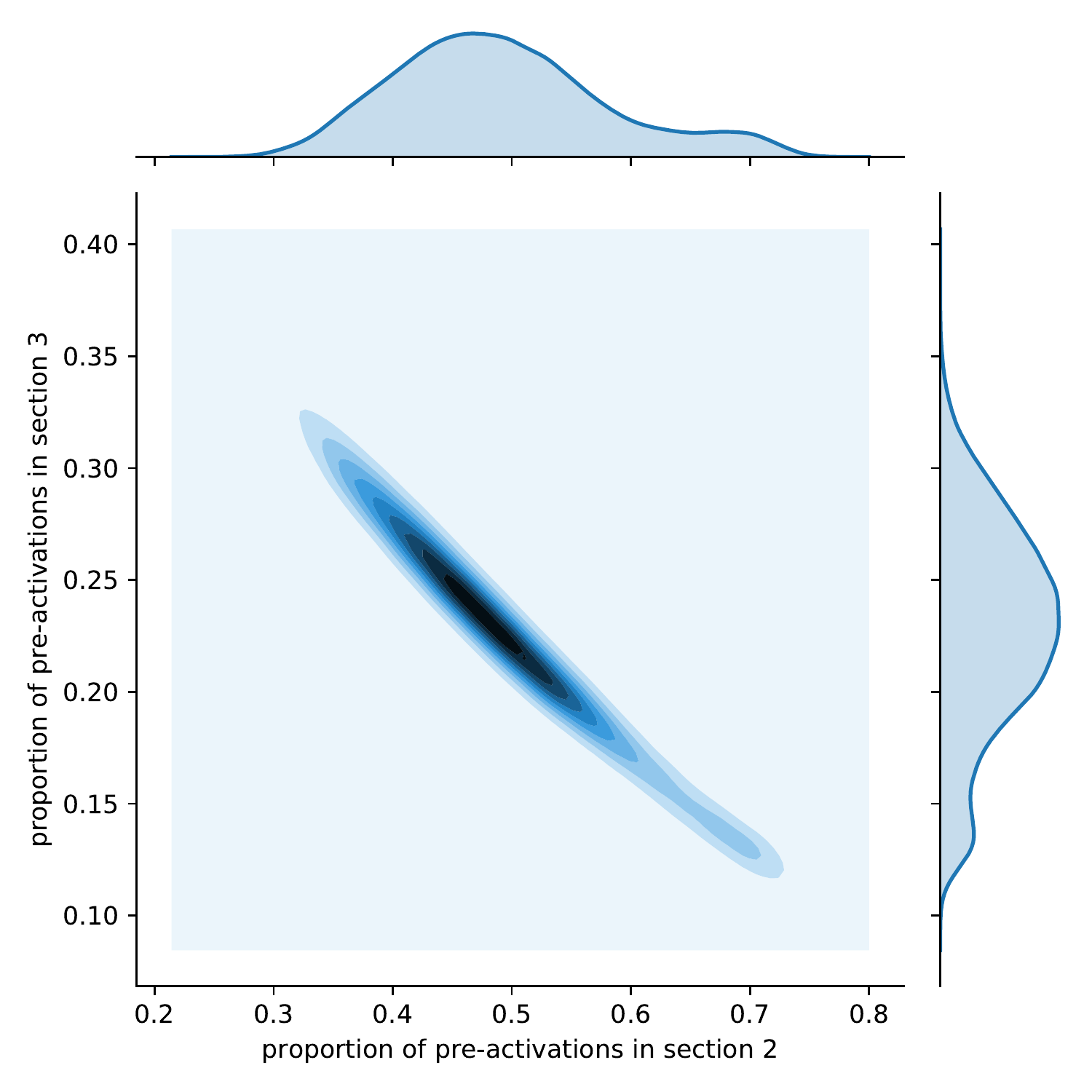}
            \caption{LeNet, MNIST data.} 
        \end{subfigure}
        \caption{Experimental distribution of  $(\bar{R}_2, \bar{R}_3)$ (neuron averaging; each sample is a single datum) for MLP and LeNet \texttt{HardTanh} networks trained to high validation accuracy on MNIST, and evaluated on i.i.d. normal and MNIST data. The plots show 2d kernel density estimation fits of the joint and 1d fits of the marginals.} 
        \label{fig:probe_agg_neuron_trained_weights_tanh}

    \end{figure*}

\begin{enumerate}
    \item The variance of $\bar{R}_2$ is `small' in all cases for \texttt{ReLU} networks except when evaluating MNIST-trained MLP networks on i.i.d. random normal data. This is the least relevant case practically.
    \item For $R_2$, the results are much less convincing, though we do note that, with random weights and i.i.d. data, the MLP network does have quite a strongly peaked distribution. In other cases the variance is undeniably large.
    \item The variance of $\bar{R}_{2,3}$ is `small' in all cases for \texttt{HardTanh} except when evaluating LeNet architectures on MNIST data.
    \item For $R_{3}$ in \texttt{HardTanh} networks, the variance seems to be low when the weights are random, but not when trained.
\end{enumerate}

Overall, we see that in some circumstances, particularly with un-trained weights, the assumption \ref{item: assumption_bernoulli} is not as unreasonable as it first sounds. More importantly for the present work, comparing the three examined activation functions supports the hypothesis that, insofar as modeling the action of the \texttt{ReLU} activation function by independent Bernoulli random variables was valid in \cite{choromanska2015loss}, our analogous modelling of the action of general piece-wise linear functions by independent discrete random variables is also valid. \jstat{Put another way, it does not appear that the assumptions we make here are any stronger than those made in \cite{choromanska2015loss}. We finally note an interesting comparison between, for example, Figures \ref{fig:probe_agg_neuron_random_weights_mlp_iid} and \ref{fig:probe_agg_neuron_random_weights_mlp_mnist}, or equally Figures \ref{fig:probe_agg_neuron_random_weights_tanh_mlp_iid} and \ref{fig:probe_agg_neuron_random_weights_tanh_nlp_mnist}. In both cases, the variance is low for both distributions, and the only difference between the two experiments is the evaluation data, being i.i.d. Gaussian in the one case, and MNIST in the other. These results seem to demonstrate that the assumption of i.i.d. Gaussian data distribution is not trivialising the problem as one might expect a priori.}

\jstat{Taking all of the results of this section together, we see that the case for our extension of \cite{choromanska2015loss} is quite strong, but there are clearly realistic cases where the modelling assumptions applied to activation functions in \cite{choromanska2015loss} are convincingly violated.}

\section{Statement of results}\label{sec:statement_results}
 We shall use \emph{complexity} to refer to any of the following defined quantities which we define precisely as they appear in \cite{auffinger2013random}.

\begin{defn}\label{def:cnk}
For a Borel set $B\subset\mathbb{R}$ and non-negative integer $k$, let \begin{equation}
    \CNk^g(B) = \left|\left\{\vec{w} \in \sqrtsign{N}S^{N-1} ~:~ \grad g(\vec{w})=0, g(\vec{w})\in B, ~ i(\grad^2g)=k\right\}\right|
\end{equation}
where $i(M)$ for a square matrix $M$ is the \emph{index} of $M$, i.e. the number of negative eigenvalues of $M$. We also define the useful generalisation $\ind{x}(M)$ to be the number of eigenvalues of $M$ less than $x$, so $\ind{0}(M) = i(M).$
\end{defn}
\begin{defn}
For a Borel set $B\subset\mathbb{R}$, let \begin{equation}
    \CN^g(B) = \left|\left\{\vec{w} \in \sqrtsign{N}S^{N-1} ~:~ \grad g(\vec{w})=0, g(\vec{w})\in B\right\}\right|.
\end{equation}
\end{defn}

We now state our main identities, which we find simpler to prove by scaling $\vec{w}$ to lie on the hyper-sphere of unit radius: $h(\vec{w}) \defeq N^{-H/2}g(\sqrtsign{N}\vec{w})$.  For convenience, we define \begin{equation}\label{eq:rho_N_redef}
\rho_{\ell}^{(N)} =  \rho_{\ell} N^{-\ell/2}
\end{equation} so that, recalling the form of $g$ in (\ref{eq:g_def}), we obtain \begin{equation}\label{eq:h_def}
    h(\vec{w}) = \sum_{i_1, \ldots, i_H = 1}^{\Lambda} X_{i_1, \ldots, i_H} \prod_{k=1}^H w_{i_k} + \sum_{\ell=1}^H\rho_{\ell}^{(N)} \sum_{i_{\ell + 1}, \ldots, i_H=1}^{\Lambda} \prod_{k=\ell +1}^H w_{i_k}.
\end{equation}Though the complexities have been defined using general Borel sets, as in \cite{auffinger2013random}, we focus on half-infinite intervals $(-\infty, u)$, acknowledging that everything that follows could be repeated instead with general Borel sets \emph{mutatis mutandis}.  We will henceforth be studying the following central quantities (note the minor abuse of notation):

\begin{equation}\label{eq:cnkh_def}
     \CNk^h(\sqrtsign{N}u) = \left|\left\{\vec{w} \in S^{N-1} ~:~ \grad h(\vec{w})=0, h(\vec{w})\in \sqrtsign{N}u, ~ i(\grad^2h)=k\right\}\right|,
\end{equation}
\begin{equation}\label{eq:cnh_def}
     \CN^h(\sqrtsign{N}u) = \left|\left\{\vec{w} \in S^{N-1} ~:~ \grad h(\vec{w})=0, h(\vec{w})\in \sqrtsign{N}u\right\}\right|
\end{equation}
and it will be useful to define a relaxed version of (\ref{eq:cnkh_def}) for $\mathcal{K}\subset \{0,1,\ldots, N\}$:
\begin{equation}\label{eq:cnkh_set_def}
      C_{N, \mathcal{K}}^h(\sqrtsign{N}u) = \left|\left\{\vec{w} \in S^{N-1} ~:~ \grad h(\vec{w})=0, h(\vec{w})\in \sqrtsign{N}u, ~ i(\grad^2h)\in\mathcal{K}\right\}\right|.
\end{equation}

Our main results take the form of two theorems that extend Theorems 2.5 and 2.8 from \cite{auffinger2013random} to our more general spin glass like object $g$, and a third theorem with partially extends Theorem 2.17 of \cite{auffinger2013random}. In the case of Theorem 2.8, we are able to obtain exactly the same result in this generalised setting. For Theorem 2.5, we have been unable to avoid slackening the result slightly, hence the introduction of the quantity $C^h_{N, \mathcal{K}}$ above. In the case of Theorem 2.17, we are only able to perform the calculations of the exact leading order term in one case and obtain a term very similar to that in \cite{auffinger2013random} but with an extra factor dependent on the piece-wise linear approximation to the generalised activation function. This exact term correctly falls-back to the term found in \cite{auffinger2013random} when we take $f=\texttt{ReLU}$.\\

\begin{restatable}{theorem}{auffindk}
\label{thm:auff2.8}%
Recall the definition of $\CN^h$ in (\ref{eq:cnh_def}) and let $\Theta_H$ be defined as in \cite{auffinger2013random}: \begin{align}
    \Theta_H(u) = \begin{cases} \frac{1}{2}\log(H-1) - \frac{H-2}{4(H-1)}u^2 - I_1(u; E_{\infty}) ~~ &\text{if } u\leq -E_{\infty},\\
    \frac{1}{2}\log(H-1) - \frac{H-2}{4(H-1)}u^2 &\text{if } -E_{\infty} \leq u \leq 0,\\
    \frac{1}{2}\log(H-1) &\text{if } 0\geq u,\end{cases}
    \end{align}
    where $E_{\infty} = 2\sqrtsign{\frac{H-1}{H}}$, and $I_1(\cdot; E)$ is defined on $(-\infty, -E]$ as in \cite{auffinger2013random} by \begin{equation}
        I_1(u; E) = \frac{2}{E^2}\int_{u}^{-E} (z^2 - E^2)^{1/2} dz = -\frac{u}{E^2}\sqrtsign{u^2 - E^2} - \log\left(-u + \sqrtsign{u^2 - E^2}\right) + \log E,\label{eq:auffinger_I1_def}
    \end{equation} then \begin{equation}
        \lim_{N\rightarrow\infty} \frac{1}{N}\log\expect C_{N}^{h}(\sqrtsign{N}u) = \Theta_H(u).
    \end{equation}
\end{restatable}

\begin{restatable}{theorem}{auffdepk}
\label{thm:auff2.5}%
Recall the definition of $C_{N, \mathcal{K}}^h$ in (\ref{eq:cnkh_set_def}) and let $\Theta_{H,k}$ be defined as in \cite{auffinger2013random}: \begin{align}
    \Theta_{H,k}(u) = \begin{cases} \frac{1}{2}\log(H-1) - \frac{H-2}{4(H-1)}u^2 - (k+1)I_1(u; E_{\infty}) ~~ &\text{if } u\leq -E_{\infty},\\
     \frac{1}{2}\log(H-1) - \frac{H-2}{H}
     &\text{if } u > -E_{\infty}, \end{cases}
    \end{align}
    then, with $\mathcal{K} = \{k-1, k, k+1\}$ for $k>0$, \begin{equation}
       \Theta_{H,k+1}(u) \leq \lim_{N\rightarrow\infty} \frac{1}{N}\log\expect C_{N,\mathcal{K}}^{h}(\sqrtsign{N}u) \leq \Theta_{H,k-1}(u)
    \end{equation}
        and similarly with $\mathcal{K} = \{0, 1\}$
        \begin{equation}
       \Theta_{H,1}(u) \leq \lim_{N\rightarrow\infty} \frac{1}{N}\log\expect C_{N,\mathcal{K}}^{h}(\sqrtsign{N}u) \leq \Theta_{H,0}(u).
    \end{equation}
    \end{restatable}

\begin{remark}
Note that Theorem \ref{thm:auff2.5} holds for \texttt{ReLU} networks (equivalently, pure multi-spin glass models), as indeed it must. It can be seen as an immediate (weaker) consequence of the Theorem 2.5 in \cite{auffinger2013random} of which it is an analogue in our more general setting.
\end{remark}

\begin{restatable}{theorem}{auffexact}
\label{thm:exact_term}%
Let $u<-E_{\infty}$ and define $v = -\frac{\sqrtsign{2}u}{E_{\infty}}$. Define the function $h$ by (c.f. (7.10) in \cite{auffinger2013random}) \begin{equation}
    h(v) = \left(\frac{|v - \sqrtsign{2}|}{|v + \sqrtsign{2}|}\right)^{1/4} + \left(\frac{|v + \sqrtsign{2}|}{|v - \sqrtsign{2}|}\right)^{1/4},
\end{equation}
\jstat{and the functions
\begin{align}
    q(\theta') =   \frac{1}{2} \sin^2 2\theta' + \frac{1}{4}\left(3+4\cos 4\theta'\right),
\end{align}
 \begin{align}
 j(x, s_1, \theta') = 1 + \frac{1}{4}s_1\sqrtsign{x^2 - 2}h(x)^2 - \frac{1}{4}s_1^2 q(\theta')|x^2 - 2|h(x)^2,
\end{align}
 \begin{align}
    T(v, s_1) = \frac{2}{\pi}\int_{0}^{\pi/2}j(-v, s_1, \theta')d\theta'.
\end{align}}
The $N-1 \times N-1$ deterministic matrix $S$ is defined subsequently around (\ref{eq:S_def}). $S$ has fixed rank $r=2$ and non-zero eigenvalues \jstat{$\{s_1, N^{-1/2}s_2\}$} where $s_j = \mathcal{O}(1)$. The specific form of $S$ is rather cumbersome and uninformative and so is relegated to Appendix \ref{ap:S_specific}, \jstat{and the vector $\vec{v}$ is defined in Lemma \ref{lemma:conditional_dist}}. Then we have 
\jstat{\begin{align}
     \expect C_{N}^{h}(\sqrtsign{N}u) &\sim \frac{N^{-\frac{\jstat{1}}{2}}}{\sqrtsign{2\pi H}} \jstat{e^{-\frac{\vec{v}^2}{2H}}}\jstat{T(v, s_1)} h(v) e^{N\Theta_H(u)} \frac{e^{I_1(u; E_{\infty}) - \frac{1}{2}u I_1'(u; E_{\infty})}}{\frac{H-2}{2(H-1)}u + I_1'(u; E_{\infty})}.
\end{align}}
\end{restatable}

We include in Figures \ref{fig:theta_h} and \ref{fig:theta_hk} plots of the functions $\Theta_H$ and $\Theta_{H,k}$ for completeness, though these figures are precisely the same as those appearing in \cite{choromanska2015loss, auffinger2013random}. The critical observation from these plots is that each of the $\Theta_{H,k}$ and $\Theta_H$ are monotonically increasing and that there exist unique $E_0 > E_1 > \ldots > E_{\infty}$ such that $\Theta_{H,k}(-E_k) = 0$ and so the critical values $-E_k$ are the boundaries between regions of exponentially many and `exponentially few' critical points of each respective index.

\begin{figure}
\begin{subfigure}{0.45\textwidth}
    \centering
    \includegraphics[width=\textwidth]{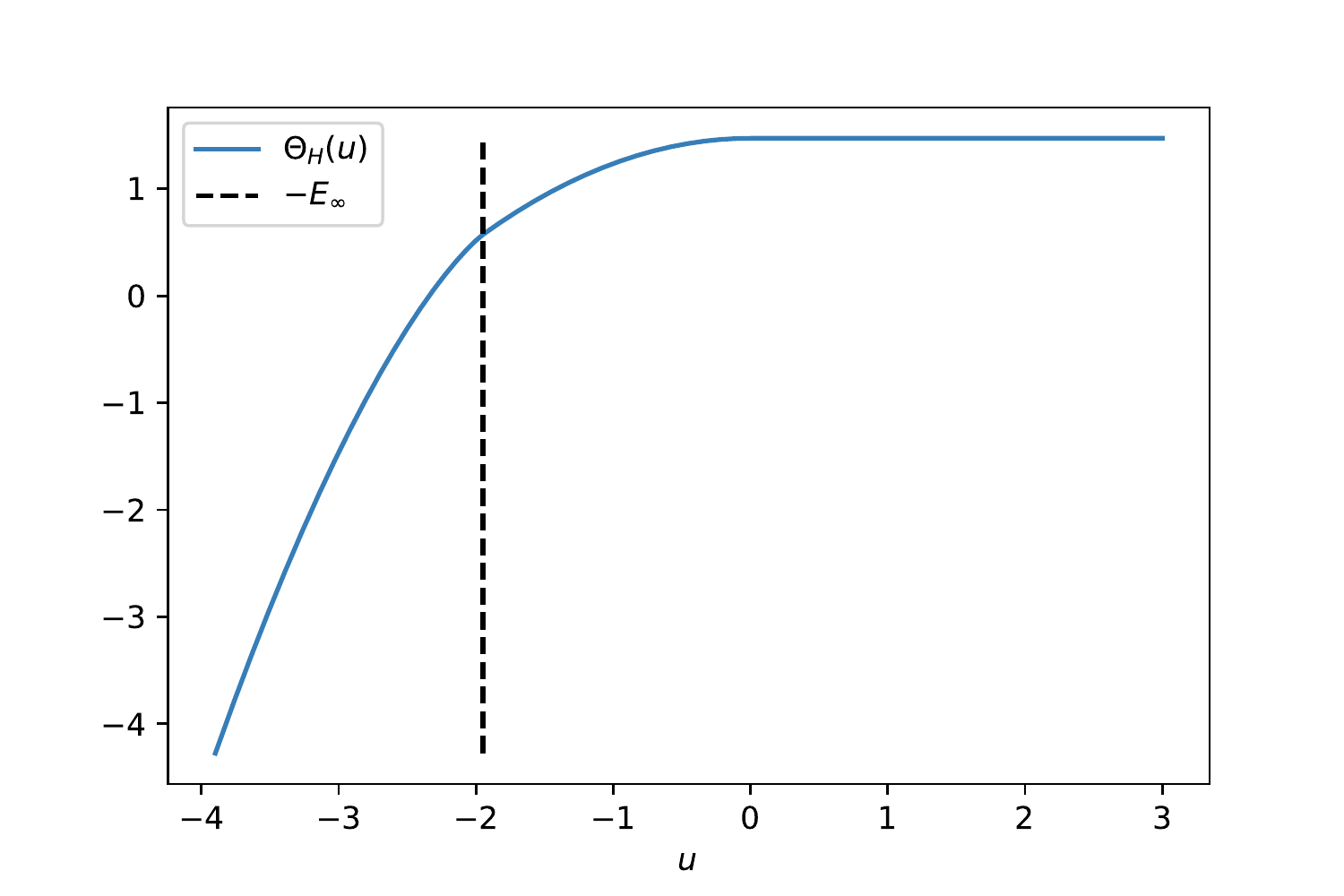}
    \caption{Plot of  $\Theta_H$.}
    \label{fig:theta_h}
\end{subfigure}
\begin{subfigure}{0.45\textwidth}
    \centering
    \includegraphics[width=\textwidth]{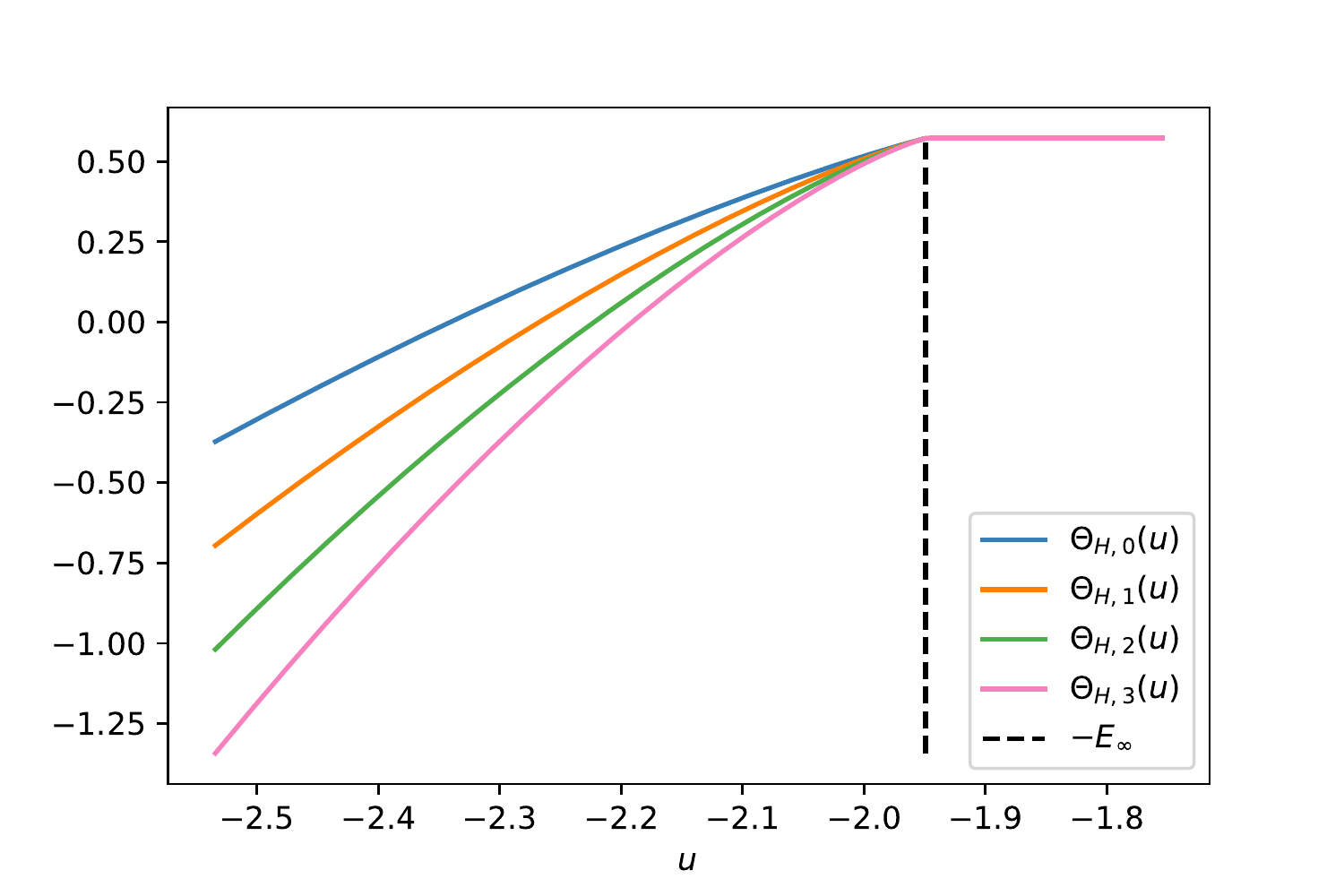}
    \caption{Plot of $\Theta_{H,k}$  $k=0,1,2,3$.}
    \label{fig:theta_hk}
\end{subfigure}
\caption{Plots of the functions $\Theta_H$ and $\Theta_{H,k}$ for $H=20$.}
\end{figure}

\jstat{\begin{remark}
It is interesting to compare the expression (\ref{eq:h_def}) to the analogous expression for the model of \cite{ros2019complex}. In that work, when scaled to the unit hypersphere and scaled so that the spin glass term is composed of $\mathcal{O}(1)$ terms, the scale of the deterministic term is $\mathcal{O}(N^{1/2})$, while the corresponding scale in (\ref{eq:h_def}) is $\mathcal{O}(N^{-1/2})$. Based on this, one might well \emph{conjecture} Theorem \ref{thm:auff2.8} and Theorem \ref{thm:auff2.5}, however one would have no means by which to conjecture Theorem \ref{thm:exact_term}, and as far we can see no means to \emph{prove} Theorem \ref{thm:auff2.8} and Theorem \ref{thm:auff2.5}. As mentioned in the introduction, the single fixed distinguished direction in \cite{ros2019complex} is quite a special feature and is not present in (\ref{eq:h_def}). \end{remark}}

\section{GOE expressions for the complexity from Kac-Rice formulae}\label{sec:goe_expressions}
In this section we conduct analysis similar to that in \cite{auffinger2013random, fyodorov2007replica, fyodorov2004complexity} to obtain expressions for the the expected number of critical points of the function $h$ as defined in (\ref{eq:h_def}). We start with an elementary lemma deriving the 2-point covariance function for $h$.

\begin{lemma}\label{lemma:covariance_dot_product}
\jstat{For $\vec{w}\in S^{N-1}$, $h$ is defined as in (\ref{eq:h_def}):  \begin{equation*}
    h(\vec{w}) = \sum_{i_1, \ldots, i_H = 1}^{\Lambda} X_{i_1, \ldots, i_H} \prod_{k=1}^H w_{i_k} + \sum_{\ell=1}^H\rho_{\ell}^{(N)} \sum_{i_{\ell + 1}, \ldots, i_H=1}^{\Lambda} \prod_{k=\ell +1}^H w_{i_k}, ~~~~ X_{i_1, \ldots,, i_H}\overset{\text{i.i.d.}}{\sim}\mathcal{N}(0,1).
\end{equation*}}
For any $\vec{w}, \vec{w}'\in S^{N-1}$ the following holds \begin{equation}
    Cov(h(\vec{w}), h(\vec{w}')) = (\vec{w}\cdot\vec{w}')^H.
\end{equation}
\end{lemma}
\begin{proof}
    Let us begin by writing \begin{equation}
        h(\vec{w}) =  \sum_{i_1, \ldots, i_H = 1}^{N} X_{i_1, \ldots, i_H} \prod_{k=1}^H w_{i_k} + h^{(2)}(\vec{w}) \equiv h^{(1)}(\vec{w}) + h^{(2)}(\vec{w})
    \end{equation}
    where $h^{(2)}$ is deterministic. Then we have \begin{align}
        Cov(h(\vec{w}), h(\vec{w}')) &\equiv \expect\left[h(\vec{w})h(\vec{w}')\right] - \expect h(\vec{w})\expect h(\vec{w}') \notag \\
        &= \expect\left[ h^{(1)}(\vec{w}) h^{(1)}(\vec{w}') - h^{(1)}(\vec{w})h^{(2)}(\vec{w}') - h^{(2)}(\vec{w})h^{(1)}(\vec{w}') + h^{(2)}(\vec{w})h^{(2)}(\vec{w}')\right] - h^{(2)}(\vec{w})h^{(2)}(\vec{w}')\notag \\
        &= \expect \left[ h^{(1)}(\vec{w})h^{(1)}(\vec{w}')\right]\notag \\
        &= \sum_{i_1,\ldots i_H=1}^N \prod_{k=1}^H w_{i_k}w_{i_k}'\notag \\
        &= \prod_{k=1}^H \sum_{i_k=1}^N w_{i_k}w_{i_k}'\notag \\
        &= (\vec{w}\cdot\vec{w}')^H
    \end{align}
    where we have used $\expect h^{(1)} = 0$ in going from the first to the second and the second to the third lines.
\end{proof}

The following lemma calculates the full joint and thence conditional distribution of $h$ and its first and second derivatives. The calculations follow closely those of \cite{auffinger2013random} and the results are required for later use in a Kac-Rice formula.

\begin{lemma}\label{lemma:conditional_dist}
Pick some Cartesian coordinates on $S^{N-1}$ and let $\vec{w}$ be the north-pole of the sphere $\vec{w} = (1,0,0,\ldots)$.  Let $h_i = \partial_i h(\vec{w})$ and $h_{ij} = \partial_i\partial_j h(\vec{w})$ where $\{\partial_i\}_{i=1}^{N-1}$ are the coordinate basis around $\vec{w}$ on the sphere. Then the following results hold.
\begin{enumerate}[label=(\alph*)]
    \item For all $1 \leq i,j,k < N$, $h(\vec{w}), h_i(\vec{w}), h_{jk}(\vec{w})$ are Gaussian random variables whose  distributions are given by 
    \begin{align}
                \expect [h(\vec{w})] &= \sum_{\ell=1}^H \rho_{\ell}^{(N)} \label{eq:derivs_exp_h}\\ 
                Var [h(\vec{w})] &= 1\label{eq:derivs_var_h} \\
               \expect h_i(\vec{w}) &= \jstat{\sum_{\ell=1}^{H-1} \rho_{\ell}^{(N)} \left[(H-\ell) + (H - \ell - 1)\delta_{i1}\right]\equiv v_i\label{eq:derivs_exp_hi}}\\
                \expect [h_{ij}(\vec{w})] &= 
                \jstat{\sum_{\ell=1}^{H-2} \rho_{\ell}^{(N)}\Bigg\{\left[(H-\ell)(H-\ell-1) +1\right] \delta_{i1}\delta_{j1} + (H-\ell - 2)(\delta_{i1} + \delta_{j1}) +1
                \Bigg\} \label{eq:derivs_exp_hij} } \\
                Cov(h(\vec{w}), h_i(\vec{w})) &= 0 \label{eq:derivs_cov_hhi}\\
                Cov(h_i(\vec{w}), h_{jk}(\vec{w})) &= 0 \label{eq:derivs_cov_hihjk}\\
                Cov(h_i(\vec{w}), h_j(\vec{w})) &= H\delta_{ij} \label{eq:derivs_cov_hihj}\\
                 Cov(h(\vec{w}), h_{ij}(\vec{w})) &= -H\delta_{ij} \label{eq:derivs_cov_hhij}\\
                 Cov(h_{ij}(\vec{w}), h_{kl}(\vec{w})) &= H(H-1)(\delta_{ik}\delta{jl} + \delta_{il}\delta_{kl}) + H^2 \delta_{ij}\delta_{kl}.\label{eq:derivs_cov_hijhkl}
     \end{align}
      \item Make the following definitions: \begin{align}
          \xi_0 &= \sum_{\ell=1}^H\rho_{\ell}^{(N)}\\
                    \xi_1 &=  \jstat{\sum_{\ell=1}^{H-2}\rho_{\ell}^{(N)} \left[(H-\ell)(H-\ell -1) +1 \right]}\\
          \xi_2 & =\jstat{\sum_{\ell=1}^{H-2}\rho_{\ell}^{(N)}(H-\ell - 2)} \\
          \xi_3 &= \sum_{\ell=1}^{H-2}\rho_{\ell}^{(N)}
      \end{align}
       Then, conditional on $h(\vec{w}) = x$, for $x\in\mathbb{R}$, the random variables $h_{ij}(\vec{w})$ are independent Gaussians satisfying \begin{align}
           \expect[h_{ij}(\vec{w}) ~|~ h(\vec{w})=x] &=\jstat{ \xi_3 + \xi_2(\delta_{i1} + \delta_{j1}) + \xi_1\delta_{i1}\delta_{j1}} - (x-\xi_0)\delta_{ij} \label{eq:cond_exp_hij} \\
            Var[h_{ij}(\vec{w}) ~|~ h(\vec{w})=x] &= H(H-1)(1+\delta_{ij})\label{eq:cond_var_hij}.
       \end{align}
       Or, equivalently, \begin{align}\label{eq:hij_goe}
           \left(h_{ij}(\vec{w}) ~|~ h(\vec{w})=x\right) 
           \sim  \sqrtsign{2(N-1)H(H-1)} \left( M^{N-1}- \frac{1}{ \sqrtsign{2(N-1)H(H-1)} } H\left(x- \xi_0\right)I + S\right)
       \end{align}
       where $M^{N-1} \sim GOE^{N-1}$ and the matrix $S$ is given by \begin{align}\label{eq:S_def}
         \jstat{  S_{ij} = \frac{1}{ \sqrtsign{2(N-1)H(H-1)} }\left(\xi_3 + \xi_2(\delta_{i1} + \delta_{j1}) + \xi_1\delta_{i1}\delta_{j1}\right).}
       \end{align}
      \jstat{Clearly all entries of $S$ are of order $N^{-1}$, recalling the scale of $\rho_{\ell}^{(N)}$ given in (\ref{eq:rho_N_redef}). Moreover, $S$ is of rank 2 and has eigenvalues $\{s_1, N^{-1/2}s_2\}$ for real $s_i=\mathcal{O}(1)$.}
       \end{enumerate}
       \end{lemma}
    
       
\begin{proof}
\begin{enumerate}[label=(\alph*)]
    \item Becuase the $X_{i_1, \ldots, i_H}$ are centred Gaussians and $\vec{w} = (1,0,0,\ldots, 0)$, we immediately obtain (\ref{eq:derivs_exp_h}). (\ref{eq:derivs_exp_hi})-(\ref{eq:derivs_exp_hij}) can be seen to be true similarly, e.g. (\ref{eq:derivs_exp_hij}) by observing that the stochastic term is again zeroed-out by taking the expectation and the only terms that survive in the non-stochastic part are of the form \begin{equation}
        \frac{\partial^2}{\partial w_i \partial  w_j}w_iw_j w_1^{H-\ell-2} ~ (i,j\neq 1), ~~~  \frac{\partial^2}{\partial w_i \partial w_1 }w_i w_1^{H-\ell-1} ~ (i\neq 1), ~~~  \frac{\partial^2}{\partial w_1^2}  w_1^{H-\ell}.  
    \end{equation}
    
    The remaining results (\ref{eq:derivs_var_h}), (\ref{eq:derivs_cov_hhi})-(\ref{eq:derivs_cov_hijhkl}) all match those in Lemma 3.2 of \cite{auffinger2013random} and follow similarly from Lemma \ref{lemma:covariance_dot_product} and the following (\cite{adler2009random}):\begin{equation}
        Cov\left(\frac{\partial^k\bar{h}(x)}{\partial x_{i_1}\ldots \partial x_{i_k}}, \frac{\partial^l\bar{h}(y)}{\partial y_{j_1}\ldots \partial y_{j_l}}\right) = \frac{\partial^{k+l}Cov(\bar{h}(x),\bar{h}(y))}{\partial x_{i_1}\ldots \partial x_{i_k}\partial y_{j_1}\ldots \partial y_{j_l}}
    \end{equation}
    where $\bar{h}\defeq h\circ \Phi^{-1}$ and $\Phi$ is a coordinate chart around $\vec{w}$.
    
    \item
        (\ref{eq:cond_exp_hij}), (\ref{eq:cond_var_hij}) and the conditional independence result follow from (\ref{eq:derivs_exp_h}), (\ref{eq:derivs_var_h}), (\ref{eq:derivs_exp_hij}), (\ref{eq:derivs_cov_hijhkl}) and the standard result for the conditional distribution of one Gaussian under another (see e.g. \cite{anderson1962introduction} Section 2.5), just as in the proof of Lemma 3.2 in \cite{auffinger2013random}.
        
        To show (\ref{eq:hij_goe}), recall that a $GOE^N$ matrix is a real symmetric random matrix $M$ and whose entries are independent centred Gaussians with with \begin{equation}
            \expect M_{ij}^2 = \frac{1+\delta_{ij}}{2N}.
        \end{equation}
        \jstat{Finally we have to determine the eigenvalues of $S$. With $a = \xi_1 + 2\xi_2 + \xi_3, b=\xi_2 + \xi_3$ and $c=\xi_3$, $S$ has entries \begin{align}
            S = \frac{1}{\sqrtsign{2(N-1)H(H-1)}}\left(\begin{array}{ccccc}
            a & b & b & \ldots & b \\
            b & c & c & \ldots & c \\
            b & c & c & \ldots & c \\
            \vdots & \vdots & \vdots & \vdots & \vdots \\
            b & c & c & \ldots & c \\
            \end{array}\right),
        \end{align}
and so has non-null eigenvectors $(1, u, u, \ldots, u)^T$ with eigenvalues $\left(2(N-1)H(H-1)\right)^{-1/2}\lambda$, where (after some simple manipulation) \begin{align}
    \lambda^2 - (a - c(N-1))\lambda + ca(N-1) - b^2(N-1) = 0, ~~~~~ u = \frac{\lambda - a}{(N-1)b}.
\end{align}
Recalling the scale of $\rho_{\ell}^{(N)} = \mathcal{O}(N^{-\ell/2})$ in (\ref{eq:rho_N_redef}) and the definitions $\xi_1,\xi_2, \xi_3$, we see that $a, b, c=\mathcal{O}(N^{-1/2})$ and so one easily obtains two solutions for $\lambda$, one of order $N^{1/2}$ and another of order $N^{-1/2}$, hence $S$ has two non-zero eigenvalues of order $1$ and $N^{-1/2}$.}
\end{enumerate}
\end{proof}

Our next lemma establishes for use in this context a Kac-Rice fomula that will provide the first step in the computation of $C^h_{N}$ and $C^h_{N, \mathcal{K}}$.

\begin{lemma}\label{lemma:kac_rice}
    Let $\hat{F}$ be a real-valued centred Gaussian field on $S^{N-1}$ that is almost surely (a.s.) $C^2$, $\tilde{F}$ be some non-random, real-valued $C^2$ function on $S^{N-1}$ and let $F \defeq\hat{F} + \tilde{F}$. Let  $\mathcal{A} = \{U_{\alpha}, \Phi_{\alpha}\}_{\alpha\in I}$ be a finite atlas on $S^{N-1}$. Let $h^{\alpha} = h\circ \Phi_{\alpha}^{-1}$, and let $h^{\alpha}_i, h^{\alpha}_{ij}$ denote derivatives of $h$ in the coordinate basis of the chart $(U_{\alpha}, \Phi_{\alpha}).$ Assume that the joint distribution $(F^{\alpha}_i(\vec{x}), F^{\alpha}_{ij}(\vec{x}))$ is non-degenerate for all $\alpha$ and for all $\vec{x}\in S^{N-1}$ and that there exist constants $K_{\alpha}, \beta >0$ such that \begin{equation}\label{eq:var_log_assumption}
        \max_{i,j}\left|Var(\hat{F}_{ij}^{\alpha}(\vec{x})) + Var(\hat{F}_{ij}^{\alpha}(\vec{y})) - 2Cov(\hat{F}_{ij}^{\alpha}(\vec{x}), \hat{F}_{ij}^{\alpha}(\vec{y}))\right| \leq K_{\alpha}\left|\log|x-y|\right|^{-1-\beta}
    \end{equation}Then the following holds \begin{equation}\label{eq:kac_rice}
      \CNk^F(B) = \int_{S^{N-1}} p_{\vec{x}}(0) \mathcal{S}_{N-1}(d\vec{x}) \expect\left[|\det\grad^2 F(\vec{x})|\indic\left\{F(\vec{x})\in B,~ i(\grad^2F(\vec{x})) = k\right\} ~|~ \grad F(\vec{x})=0\right]     \end{equation}

      where $p_{\vec{x}}$ is the density of $\grad F$ at $\vec{x}$ and $\mathcal{S}_{N-1}$ is the usual surface measure on $S^{N-1}$. Similarly,
      \begin{equation}\label{eq:kac_rice_no_k}
      \CN^F(B) = \int_{S^{N-1}} p_{\vec{x}}(0) \mathcal{S}_{N-1}(d\vec{x}) \expect\left[|\det\grad^2 F(\vec{x})|\indic\left\{F(\vec{x})\in B\right\} ~|~ \grad F(\vec{x})=0\right]     \end{equation}
\end{lemma}

The proof of Lemma \ref{lemma:kac_rice} shall rely heavily on a central result from \cite{adler2009random} which we now state as a Theorem.
\begin{theorem}[\cite{adler2009random} Theorem 12.1.1]\label{thm:adler_kac_rice}
Let $\mathcal{M}$ be a compact , oriented, N-dimensional $C^1$ manifold with a $C^1$ Riemannian metric $g$. Let $\phi:\mathcal{M}\rightarrow\mathbb{R}^N$ and $\psi:\mathcal{M}\rightarrow \mathbb{R}^K$ be random fields on $\mathcal{M}$. For an open set $A\subset\mathbb{R}^K$ for which $\partial A$ has dimension $K-1$ and a point $\vec{u}\in\mathbb{R}^{N}$ let \begin{equation}
    N_{\vec{u}} \defeq \left|\{x\in\mathcal{M} ~|~ \phi(x) = \vec{u}, ~ \psi(x)\in A\}\right|.
\end{equation}

Assume that the following conditions are satisfied for some orthonormal frame field E:
\begin{enumerate}[label=(\alph*)]
\item
All components of $\phi$, $\grad_E \phi$, and $\psi$ are a.s. continuous and have finite variances (over $\mathcal{M}$).
\item
 For all $x\in\mathcal{M}$, the marginal densities $p_{x}$  of $\phi(x)$ (implicitly assumed to exist) are continuous at  $\vec{u}$.
 \item
 The conditional densities $p_{x}(\cdot|\grad_E\phi(x),\psi(x))$ of $\phi(x)$ given $\psi(x)$ and $\grad_E\phi(x)$ (implicitly assumed to exist) are bounded above and continuous at $\vec{u}$, uniformly in $\mathcal{M}$.
 \item
 The conditional densities $p_x (\cdot|\phi(x) = \vec{z})$ of $\det(\grad_{E_j}\phi^i  (x))$ given are continuous in a neighbourhood of $0$ for $\vec{z}$ in a neighbourhood of $\vec{u}$  uniformly in $\mathcal{M}$.
 \item
 The conditional densities $p_x (\cdot|\phi (x) = \vec{z})$ are continuous for $\vec{z}$ in a neighbourhood of $\vec{u}$ uniformly in $\mathcal{M}$.
 \item
 The following moment condition holds \begin{equation}
     \sup_{x\in\mathcal{M}}\max_{1\leq i,j\leq N}\expect\left\{\left|\grad_{E_j}f^i(x)\right|^N\right\}< \infty
 \end{equation}
 \item
 The moduli of continuity with respect to the (canonical) metric induced  by $g$ of each component of $\psi$, each component of $\phi$ and each $\grad_{E_j}f^i$ all satisfy, for any $\epsilon > 0$ \begin{equation}\label{eq:moduli_condition}
    \mathbb{P}( \omega(\eta) >\epsilon) = o(\eta^N), ~~ \text{as } \eta\downarrow 0
 \end{equation}
 where the \emph{modulus of continuity} of a real-valued function $G$ on a metric space $(T, \tau)$ is defined as (c.f. \cite{adler2009random} around (1.3.6)) \begin{equation}
     \omega(\eta) \defeq \sup_{s,t : \tau(s,t)\leq\eta}\left|G(s) - G(t)\right|
 \end{equation} 
\end{enumerate}
Then \begin{equation}\label{eq:adler_taylor_kac_rice}
    \expect N_{\vec{u}} = \int_{\mathcal{M}}\expect \left\{|\det \grad_E\phi(x)|\indic\{\psi(x)\in A\} ~| ~ \phi(x) = \vec{u}\right\}p_x(\vec{u}) \text{Vol}_g(x)
\end{equation}
where $p_x$ is the density of $\phi$ and $\text{Vol}_g$ is the volume element induced by $g$ on $\mathcal{M}$.
\end{theorem}

\begin{proof}[Proof of Lemma \ref{lemma:kac_rice}]
Following the proofs of Theorem 12.4.1 in \cite{adler2009random} and Lemma 3.1 in \cite{auffinger2013random}, we will apply Theorem \ref{thm:adler_kac_rice} to the choices \begin{align}
    \phi &\defeq \grad F\notag\\
    \psi &\defeq (F, \grad_{i}\grad_{j}F)\notag\\
    A &\defeq B \times A_k \equiv B \times \{H\in\text{Sym}_{N-1\times N-1} ~|~ i(H)=k\} \subset \mathbb{R}\times \text{Sym}_{N-1\times N-1},\notag\\
    \vec{u} &= 0
\end{align}
Then, if the conditions of Theorem \ref{thm:adler_kac_rice} hold for these choices, we immediately obtain the result. It remains therefore to check the conditions of Theorem \ref{thm:adler_kac_rice}. Firstly, $A$ is indeed an open subset of of $\mathbb{R}\times \text{Sym}_{N-1\times N-1}$ (in turn, isomorphic to some $\mathbb{R}^K$) as can be easily deduced from the continuity of a matrix's eigenvalues in its entries. Condition (a) follows from the assumption of $\hat{F}$ being a.s. $C^2$ and $\tilde{F}$ being $C^2$. Conditions (b)-(f) all follow immediately from the Gaussianity of $\hat{F}$. To establish condition (g), we define $\hat{\omega}(\eta)$ and $\tilde{\omega}(\eta)$ in the obvious way and note that $\tilde{\omega}$ is non-random. Then, because $\tilde{F}$ is continuous, given $\epsilon > 0$ there exists some $\eta_0 >0$ such that for all $\eta < \eta_0$, $\tilde{\omega}(\eta) \leq \epsilon$. Let $\tilde{\omega}_0 \defeq \tilde{\omega}(\eta_0)$ and choose some $\eta_1$ such that for all $\eta < \eta_1$, $\tilde{\omega}(\eta) < \tilde{\omega}_0$. We have $\omega(\eta) \leq \hat{\omega}(\eta) + \tilde{\omega}(\eta)$ and so for $\eta < \eta_1$ \begin{align}
    \mathbb{P}(\omega(\eta) > \epsilon) & \leq \mathbb{P}(\hat{\omega}(\eta) + \tilde{\omega}(\eta) > \epsilon) \notag\\
    &=\mathbb{P}(\hat{\omega}(\eta) > \epsilon - \tilde{\omega}(\eta)) \notag\\
    &\leq \mathbb{P}(\hat{\omega}(\eta) > \epsilon - \tilde{\omega}_0)\label{eq:condition_g_ineq}
\end{align}
and we note that $\epsilon - \tilde{\omega}_0 \geq 0$ by construction. $\hat{\omega}$ is the modulus of continuity for a centred Gaussian field and so the condition (g) follows from (\ref{eq:condition_g_ineq}) and the assumption (\ref{eq:var_log_assumption}) by the Borell-TIS inequality \cite{adler2009random}, just as in the proof of Corollary 11.2.2 in \cite{adler2009random}. (\ref{eq:kac_rice_no_k}) is obtained in precisely the same way but simply dropping the $i(H) = k$ condition.
\end{proof}

\section{Asymptotic evaluation of complexity}\label{sec:asymptotic_evaluation}
In this section we conduct an asymptotic analysis of the GOE expressions for the complexity found in the preceding section. We first consider the case of counting critical points without any condition of the signature of the Hessian, which turns out to be easier. We then introduce the exact signature condition on the Hessian and proceed by presenting the necessary modifications to certain parts of our arguments.

\subsection{Complexity results with no Hessian signature prescription}
\jstat{We need to establish a central lemma, which is a key step towards a generalisation of the results presented in \cite{auffinger2013random} but established by entirely different means, following the supersymmetric calculations of \cite{nock}. Before this main lemma, we state a generalisation of a result from \cite{fyodorov2002characteristic} which is proved in Appendix \ref{ap:low_rank_fyod}.
\begin{restatable}{lemma}{fyodgeneral}
\label{lem:fyod_general}
Given $m$ vectors in $\mathbb{R}^N$ $\vec{x}_1, \ldots, \vec{x}_m$, denote by $Q(\vec{x}_1, \ldots, \vec{x}_m)$ the $m\times m$ matrix whose entries are given by $Q_{ij} = \vec{x}_i^T\vec{x}_j$. Let $F$ be any function of an $m\times m$ matrix such that the integral \begin{equation}
     \int_{\mathbb{R}^N}\ldots\int_{\mathbb{R}^N}d\vec{x}_1\ldots d\vec{x}_m |F(Q)|
\end{equation}
exists, and let $S$ be a real symmetric $N\times N$ matrix of fixed rank $r$ and with non-zero eigenvalues $\{N^{\alpha}s_i\}_{i=1}^r$ for some $\alpha < 1/2$. Define the integral \begin{equation}
     \mathcal{J}_{N, m}(F; S) \defeq \int_{\mathbb{R}^N}\ldots\int_{\mathbb{R}^N}d\vec{x}_1\ldots d\vec{x}_m F(Q) e^{-iN\sum_{i=1}^N \vec{x}_i^T S\vec{x}_i}.
\end{equation} Then as $N\rightarrow\infty$ we have \begin{equation}
    \mathcal{J}_{N,m}(F; S) =\left( 1 + o(1))\right) \frac{\pi^{\frac{m}{2}\left(N - \frac{m-1}{2}\right)}}{\prod_{k=0}^{m-1}\Gamma\left(\frac{N-k}{2}\right)}\int_{\text{Sym}_{\geq 0}(m)}d\hat{Q} \left(\det \hat{Q}\right)^{\frac{N-m-1}{2}}F(\hat{Q})\prod_{i=1}^N\prod_{j=1}^r\left( 1+ iN^{\alpha}\hat{Q}_{ii}s_j\right).
\end{equation}
\end{restatable}}

Now we state and prove the main lemma.
\begin{lemma}\label{lemma:nock_deformed}
    Let $S$ be a rank $r$ $N\times N$ symmetric matrix with non-zero eigenvalues $\{s_j\}_{j=1}^r$, where $r=\mathcal{O}(1)$ and $s_j = \mathcal{O}(1)$, and suppose $S$ has all entries of order $\mathcal{O}(N^{-1})$ in a fixed basis. Let $x<0$ and let $M$ denote an $N\times N$ GOE matrix with respect to whose law expectations are understood to be taken. Then
    \begin{align}
                  \expectGOE |\det(M - xI + S)| = K_N\lim_{\epsilon\searrow 0} e^{2N(x^2 - \epsilon^2)}\left(1 + o(1)\right)&\iiint_0^{\pi/2} d\theta d\theta'd\hat{\theta}\iint_0^{\infty}dp_1dp_2 \iint_{\Gamma} dr_1dr_2\notag\\
                  &J_1(p_1, p_2, \theta'; S, N)J_2(r_1,r_2, p_1, p_2)\cos^22\theta \sin2\theta \sin2\hat{\theta}\notag\\
& \exp\Bigg\{-N\Bigg(2\psi^{(+)}_L(r_1; x; \epsilon\cos2\theta\cos2\hat{\theta}) +2\psi^{(+)}_U(r_2; x; \epsilon\cos2\theta\cos2\hat{\theta})\notag\\& +\psi^{(-)}_L(p_1; x; \epsilon\cos2\theta')+\psi^{(-)}_U(p_2; x; \epsilon\cos2\theta')\Bigg)\Bigg\}\end{align}
    where \begin{align}
       J_1(p_1, p_2, \theta'; \{s_j\}_{j=1}^r, N) &=\prod_{j=1}^r\left(1 + iN^{1/2}s_j(p_1 + p_2) - Ns_j^2\left[\frac{1}{4}\sin^22\theta' (p_1^2 + p_2^2) + \frac{1}{4}\left(3 + 4\cos4\theta'\right)p_1p_2 \right]\right)^{-1/2},\\
        J_2(r_1, r_2, p_1, p_2; \epsilon) &= (r_1 + p_1)(r_2 + p_1)(r_1 + p_2)(r_2 + p_2)|r_1 - r_4|^4 |p_1-p_2| (r_1r_2)^{-2} (p_1p_2)^{-3/2}
    \end{align} 
    and \begin{equation}
        K_N =   \frac{N^{N+3}(-i)^N }{\Gamma\left(\frac{N}{2}\right)\Gamma\left(\frac{N-1}{2}\right) \pi^{3/2}}
    \end{equation}
and the functions $\psi^{\pm}_L, \psi^{(\pm)}_U$ are given by \begin{align}
 \psi^{(\pm)}_L(z; x,\epsilon) &= \frac{1}{2}z^2 \pm i(x+i\epsilon)z - \frac{1}{2}\log z,\\
              \psi^{(\pm)}_U(z;x,\epsilon) &= \frac{1}{2}z^2 \pm i(x-i\epsilon)z - \frac{1}{2}\log z,
\end{align}
and $\Gamma$ is a contour bounded away from zero in $\mathbb{C}$, e.g. that shown in Figure \ref{fig:r1r2_contour}.
\end{lemma}
\begin{proof}
We begin with the useful expression for real symmetric matrices $A$ \cite{fyodorov2005counting, fyodorov2004complexity} \begin{equation}
    |\det A| = \lim_{\epsilon\rightarrow 0} \frac{\det A \det A}{\sqrtsign{\det (A - i\epsilon)}\sqrtsign{\det (A +i\epsilon)}}
\end{equation} where the limit is taken over real $\epsilon$, and wlog $\epsilon > 0$. We're free to deform the matrices in the numerator for the sake of symmetry in the ensuing calculations, so \begin{align}\label{eq:det_ratio_epsilon}
     |\det A| = \lim_{\epsilon\searrow 0} \frac{\det (A - i\epsilon) \det (A + i\epsilon)}{\sqrtsign{\det (A - i\epsilon)}\sqrtsign{\det (A +i\epsilon)}}.
\end{align}
For notational convenience we put \begin{equation}
    \Delta_{\epsilon}(M; x, S) = \frac{\det (M - xI + S - i\epsilon) \det (M - xI + S + i\epsilon)}{\sqrtsign{\det (M - xI + S - i\epsilon)}\sqrtsign{\det (M - xI + S +i\epsilon)}}.
\end{equation}

Then we express the determinants and half-integer powers of determinants as Gaussian integrals over anti-commuting and commuting variables respectively as in \cite{nock} and \cite{fyodorov2015random}: \begin{align}
     \Delta_{\epsilon}(M; x, S)  = K^{(1)}_N\int d\vec{x}_1 d\vec{x}_2 d\zeta_1 d\zeta_1^{\dagger} d\zeta_2 d\zeta_2^{\dagger} &\exp\left\{-i\vec{x}_1^T(M-(x + i\epsilon)I+S)\vec{x}_1 - i\vec{x}_2^T(M-(x-i\epsilon)I + S)\vec{x}_2\right\}\notag \\
      + &\exp\left\{ i \zeta_1^{\dagger}(M-(x+i\epsilon) I+S)\zeta_1
      + i \zeta_2^{\dagger}(M-(x - i\epsilon)I+S)\zeta_2\right\}\label{eq:nock_initial}\end{align}
     where $K^{(1)}_N = (-i)^N \pi^{-N}$, which follows from standard facts about commuting Gaussian integrals and Berezin integration. 
The remainder of the calculation is very similar to that presented in \cite{nock,fyodorov2015random} but we present it in full to keep track of the slight differences. Let \begin{equation}
    A = \vec{x}_1\vec{x}_1^T + \vec{x}_2\vec{x}_2^T + \zeta_1\zeta_1^{\dagger} + \zeta_2\zeta_2^{\dagger}
\end{equation}
and note that, by the cyclicity of the trace, \begin{align}
    \vec{x}_j^T(M-(x \pm i\epsilon)I+S)\vec{x}_j &= \Tr\left((M-(x\pm i\epsilon)I+S)\vec{x}_j\vec{x}_j^T\right)\\
    \zeta_j^{\dagger}(M-(x \pm i \epsilon)I+S)\zeta_j &= -\Tr\left((M-(x\pm i \epsilon)I+S)\zeta_j\zeta_j^{\dagger}\right)
\end{align}
and so we can rewrite (\ref{eq:nock_initial}) as \begin{align}
      \Delta_{\epsilon}(M; x, S)  = K^{(1)}_N\int d\vec{x}_1 d\vec{x}_2 d\zeta_1 d\zeta_1^{\dagger} d\zeta_2 d\zeta_2^{\dagger} &\exp\left\{-i\Tr MA - i\Tr SA + i(x + i\epsilon)\vec{x}_1^T\vec{x}_1 + i(x - i\epsilon)\vec{x}_2^T\vec{x}_2  \right\} \notag\\ &\exp\left\{-i(x + i\epsilon)\zeta_1^{\dagger}\zeta_1 -i (x - i\epsilon)\zeta_2^{\dagger}\zeta_2 \right\}. \label{eq:nock2}
\end{align}
We then define the Bosonic and Fermionic matrices \begin{align}
    Q_B = \left(\begin{array}{cc} \vec{x}_1^T\vec{x}_1 & \vec{x}_1^T\vec{x}_2 \\ \vec{x}_2^T\vec{x}_1 & \vec{x}_2^T\vec{x}_2\end{array}\right), ~~  Q_F =  \left(\begin{array}{cc} \zeta_1^{\dagger}\zeta_1 & \zeta_1^{\dagger}\zeta_2 \\ \zeta_2^{\dagger}\zeta_1 & \zeta_2^{\dagger}\zeta_2\end{array}\right)
\end{align}
and also $B = \vec{x}_1\vec{x}_1^T + \vec{x}_2\vec{x}_2^T$.
Note that (\ref{eq:det_ratio_epsilon}) is true for all real symmetric matrices $A$ and so for \emph{all} real symmetric $M,S$ and real values $x$ we have \begin{equation}
    \lim_{\epsilon\searrow 0} \Delta_{\epsilon}(M; x, S) = |\det\left(M - xI + S\right)|
\end{equation}
and so with respect to the GOE law for $M$ we certainly have \begin{equation}
  \Delta_{\epsilon}(M; x, S) \overset{\text{a.s.}}{\rightarrow} |\det\left(M - xI + S\right)| ~~~ \text{as } \epsilon\searrow 0
\end{equation}
thus meaning that the $\epsilon\searrow 0$ limit can be exchanged with a GOE expectation over $M$. We therefore proceed with fixed $\epsilon>0$ to compute the GOE expectation of $\Delta_{\epsilon}.$

We have the standard Gaussian Fourier transform result for matrices: \begin{align}\label{eq:goe_fourier}
    \expectGOE e^{-i\Tr MA} = \exp\left\{-\frac{1}{8N}\Tr(A + A^T)^2\right\}
\end{align} and from \cite{nock}\footnote{Note that (4.100) in \cite{nock} contains a trivial factor of 4 error that has non-trivial consequences in our calculations.} \begin{equation}
    \Tr(A+A^T)^2 = 4\Tr Q_B^2 - 2\Tr Q_F^2 + 4\zeta_1^T\zeta_2\zeta_2^{\dagger}\zeta_1^* - 8\zeta_1^{\dagger}B\zeta_1 - 8 \zeta_2^{\dagger}B\zeta_2
\end{equation} so we can take the GOE average in (\ref{eq:nock2}) and obtain \begin{align}
    \expectGOE  \Delta_{\epsilon}(M; x, S) = K^{(1)}_N\int d\vec{x}_1 d\vec{x}_2 d\zeta_1 &d\zeta_1^{\dagger} d\zeta_2 d\zeta_2^{\dagger} 
    \exp\left\{ -\frac{1}{2N}\Tr Q_B^2 - i\Tr SB + ix\Tr Q_B + \epsilon \Tr Q_B\sigma \right\}\notag\\
    &\exp\left\{\frac{1}{4N}\Tr Q_F^2 - \frac{1}{2N}\zeta_1^T\zeta_2\zeta_2^{\dagger}\zeta_1^* + \sum_{j=1}^2 \zeta_j^{\dagger}\left(\frac{B}{N} + iS - i(x + i(-1)^{j-1} \epsilon)\right)\zeta_j\right\}.\label{eq:nock3}
\end{align}
where we have defined $$\sigma = \left(\begin{array}{cc} -1 & 0 \\ 0 & 1 \end{array}\right).$$

We can then use the transformation \begin{equation}
    \exp\left\{\frac{1}{4N}\Tr Q_F^2\right\} = \frac{N^2}{\pi Vol(U(2))}\int d\hat{Q}_F \exp\left\{-N\Tr \hat{Q}_F^2 + \Tr Q_F\hat{Q}_F\right\}
\end{equation}
to obtain \begin{align}
        \expectGOE \Delta_{\epsilon}(M; x, S) = K^{(2)}_N\int &d\vec{x}_1 d\vec{x}_2 d\zeta_1 d\zeta_1^{\dagger} d\zeta_2 d\zeta_2^{\dagger} d\hat{Q}_F
    \exp\left\{ -\frac{1}{2N}\Tr Q_B^2 - i\Tr SB + ix\Tr Q_B + \epsilon \Tr Q_B\sigma\right\}\notag\\
    &\exp\left\{-N\Tr \hat{Q}_F^2 + \Tr \hat{Q}_FQ_F -  \frac{1}{2N}\zeta_1^T\zeta_2\zeta_2^{\dagger}\zeta_1^* + \sum_{j=1}^2 \zeta_j^{\dagger}\left(\frac{B}{N} + iS - i(x + i(-1)^{j-1} \epsilon\right)\zeta_j\right\}
\label{eq:nock4}\end{align}
where $K^{(2)}_N = K^{(1)}_N \frac{N^2}{\pi Vol(U(2))}.$
The Fermionic cross-term in (\ref{eq:nock4}) can be dealt with using (see \cite{nock} (4.104)) \begin{equation}
    \exp\left(-\frac{1}{2N}\zeta_1^T\zeta_2\zeta_2^{\dagger}\zeta_1^{*}\right) = \frac{2N}{\pi} \int d^2u \exp\left(-2N\bar{u}u - i\left(u\zeta_1^{\dagger}\zeta_2^{*} + \bar{u}\zeta_2^{\dagger}\zeta_1\right)\right)
\end{equation} where $d^2u = d\Re{u} ~ d\Im{u}$, and so we obtain \begin{align}
     \expectGOE  \Delta_{\epsilon}(M; x, S) = K^{(3)}_N\int &d\vec{x}_1 d\vec{x}_2 d\zeta_1 d\zeta_1^{\dagger} d\zeta_2 d\zeta_2^{\dagger} d\hat{Q}_F d^2u
    \exp\left\{ -\frac{1}{2N}\Tr Q_B^2 - i\Tr SB + ix\Tr Q_B + \epsilon\Tr Q_B \sigma\right\}\notag\\
    &\exp\left\{-N\Tr\hat{Q}_F^2 - 2N u \bar{u}\right\}\notag\\
    &\exp\left\{ \Tr \hat{Q}_FQ_F - i(u\zeta_1^{\dagger}\zeta_2^{*} + \bar{u}\zeta_2^T\zeta_1) + \sum_{j=1}^2 \zeta_j^{\dagger}\left(\frac{B}{N} + iS - i(x + i(-1)^{j-1} \epsilon\right)\zeta_j\right\}\label{eq:nock5}
\end{align}
where $K_N^{(3)} = K_N^{(2)}\frac{2N}{\pi}$.
To simplify the Fermionic component of (\ref{eq:nock5}) and make apparent its form, we introduce $\zeta^T=(\zeta_1^{\dagger}, \zeta_1^T, \zeta_2^{\dagger}, \zeta_2^T)$ and then (\ref{eq:nock5}) reads \begin{align}
        \expectGOE  \Delta_{\epsilon}(M; x, S) = K^{(3)}_N\int &d\vec{x}_1 d\vec{x}_2 d\zeta d\hat{Q}_F d^2u
    \exp\left\{ -\frac{1}{2N}\Tr Q_B^2 - i\Tr SB + ix\Tr Q_B + \epsilon\Tr Q_B \sigma\right\}\notag\\
    &\exp\left\{-N\Tr\hat{Q}_F^2 - 2N u \bar{u}\right\}\notag\\
    &\exp\left\{\frac{1}{2}\zeta^T\mathcal{M}\zeta\right\}\notag\\
     = K^{(3)}_N\int &d\vec{x}_1 d\vec{x}_2 d\hat{Q}_F d^2u
    \exp\left\{ -\frac{1}{2N}\Tr Q_B^2 - i\Tr SB + ix\Tr Q_B + \epsilon\Tr Q_B\sigma\right\}\notag\\
    &\exp\left\{-N\Tr\hat{Q}_F^2 - 2N u \bar{u}\right\}\notag\\
    &\sqrtsign{\det\mathcal{M}}\label{eq:nock6} 
\end{align}
where the matrix $\mathcal{M}$ is given by \begin{equation}
    \mathcal{M} = \left(\begin{array}{cccc}
        0 & A_1 & - iu & q_{12}^* \\
        -A_1 & 0 & - q_{12} & i\bar{u} \\
        iu & q_{12} & 0 & A_2 \\
        -q_{12}^* & -i\bar{u} & -A_2 & 0
    \end{array}\right)\end{equation}
    and, by analogy with (4.107) in \cite{nock}, \begin{equation}
        A_j = q_{jj} - i(x + i(-1)^{j-1}\epsilon)+ \frac{1}{N}B + iS,
    \end{equation}
    where $q_{ij}$ are the entries of $\hat{Q}_F$. To evaluate $\det\mathcal{M}$, we make repeated applications of the well-known result for block $2\times 2$ matrices consisting of  $N\times N$ blocks: $$\det\left(\begin{array}{cc} A & B \\ C & D \end{array}\right) = \det(A - BD^{-1}C)\det(D).$$ This process quickly results in \begin{align}
        \sqrtsign{\det\mathcal{M}} &= \det(A_1A_2 - (u\bar{u} + q_{12}\bar{q}_{12}))\notag\\
        & =  \det\left(\left[\det(\hat{Q}_F - ix- \epsilon\sigma) - \bar{u}u\right]I + \Tr(\hat{Q}_F - ix- \epsilon\sigma)\left(\frac{1}{N}B + iS\right) + \left(\frac{1}{N}B + iS\right)^2\right)\notag\\
        & = \det\left( G_1 + N^{-1}B + iS\right)\det\left( G_2 + N^{-1}B + iS\right)\label{eq:detM_G_factor}
        \end{align}
    where we have chosen $G_1$, $G_2$ to be solutions to \begin{align}
        G_1G_2 &= \det(\hat{Q}_F - ix- \epsilon\sigma) - \bar{u}u\label{eq:G_sim1}\\
        G_1 + G_2 &= \Tr(\hat{Q}_F - ix- \epsilon\sigma)\label{eq:G_sim2}.
    \end{align}
    Recalling the $B$ has rank $2$ we let $O_{B}$ be the $N\times 2$ matrix of the non-null eigenvectors of $B$ and $\lambda^{(B)}_{1,2}$ be its non-null eigenvalues and use the determinantal identity found in equation (3) of \cite{benaych2012large} to write\footnote{Note that we here include explicitly the identity matrix symbols to make plain the dimension of the determinants.}
    \begin{align}
        \det\left(G_j I_N + N^{-1}B + iS\right) &= \det\left(G_jI_N + iS\right)\det\left(I_2 + N^{-1}O_{B}^T\left(G_jI_N + iS\right)^{-1}O_{B}\text{diag}\left(\lambda^{(B)}_1, \lambda^{(B)}_2\right)\right).\label{eq:det_M_expansion_pre_asymp}
    \end{align}

We would now like to apply the integral formula found in Appendix D of \cite{fyodorov2002characteristic} to re-write the integrals over the $N$-dimensional vectors $\vec{x}_1, \vec{x}_2$ as a single integral over a $2\times 2$ symmetric matrix $Q_B$. However, the integrand does not only depend on $\vec{x}_1, \vec{x}_2$ through $Q_B \equiv \left(\begin{array}{cc} \vec{x}_1^T\vec{x}_1 & \vec{x}_1^T\vec{x}_2 \\ \vec{x}_2^T\vec{x}_1 & \vec{x}_2^T\vec{x}_2\end{array}\right)$ thanks to the dependence on the eigenvectors of $B$ in (\ref{eq:det_M_expansion_pre_asymp}) and also in the term $\Tr SB$ in (\ref{eq:nock6}). Before addressing this problem, we will continue to manipulate the $\hat{Q}_F$ and $u$ integrals along the lines of \cite{nock}.\\

First make the change of variables $\hat{Q}_F \leftarrow \hat{Q}_F + ix + \epsilon\sigma$ and $\vec{x}_j \leftarrow \sqrtsign{N}\vec{x}_j$ in (\ref{eq:nock6}) using (\ref{eq:detM_G_factor}) to obtain \begin{align}
        \expectGOE  \Delta_{\epsilon}(M; x, S) = K^{(4)}_N\int &d\vec{x}_1 d\vec{x}_2 d\hat{Q}_F d^2u
    \exp\left\{ -\frac{N}{2}\Tr Q_B^2 - iN\Tr SB + ixN\Tr Q_B + \epsilon N\Tr Q_B\sigma\right\}\notag\\
    &\exp\left\{-N\Tr\hat{Q}_F^2 - 2N \Tr (ix + \epsilon\sigma)\hat{Q}_F - N\Tr(ix + \epsilon\sigma)^2 - 2N u \bar{u}\right\}\prod_{j=1}^2\det\left(G_j + B + iS\right)\label{eq:nock7_prime} 
\end{align}
where $K^{(4)}_N = N^{N}K_N^{(3)}$ and now the terms $G_1, G_2$ are given by the modified versions of (\ref{eq:G_sim1})-(\ref{eq:G_sim2}):\begin{align}
        G_1G_2 &= \det\hat{Q}_F  - \bar{u}u\label{eq:G_sim1_prime}\\
        G_1 + G_2 &= \Tr\hat{Q}_F \label{eq:G_sim2_prime}.
    \end{align}
    
We now diagonalise the Hermitian matrix $\hat{Q}_F = \hat{U}\text{diag}(q_1, q_2)\hat{U}^{\dagger}$ in (\ref{eq:nock7_prime}), but the term $\Tr \sigma\hat{Q}_F$ is not unitarily invariant, so we follow \cite{nock} and introduce an explicit parametrization\footnote{\cite{nock} uses an incorrect parametrization with only two angles. The calculations are are invariant in the extra angles $\alpha,\beta$ and so this detail only matters if one is tracking the multiplicative constants, as we do here.} of the unitary matrix $\hat{U}$ $$\hat{U} =e^{i\hat{\phi}/2} \left(\begin{array}{cc} e^{i\hat{\alpha}/2} & 0 \\ 0 & e^{-i\hat{\alpha}/2}\end{array}\right)\left(\begin{array}{cc} \cos\hat{\theta} & \sin\hat{\theta} \\ -\sin\hat{\theta} & \cos\hat{\theta}\end{array}\right)\left(\begin{array}{cc} e^{i\hat{\beta}/2} & 0 \\ 0 & e^{-i\hat{\beta}/2}\end{array}\right)$$
    where $\hat{\phi},\hat{\alpha}, \hat{\beta}\in [0,2\pi)$, $\hat{\theta}\in [0,\pi/2)$ and elementary calculations give the Jacobian factor $|q_1 - q_2|^2 \sin(2\hat{\theta})$. Further brief elementary calculations give
    \begin{equation}
        \Tr \hat{Q}_F\sigma  = (q_2 - q_1)\cos(2\hat{\theta}).
    \end{equation}
    and so, integrating out $\hat{\phi}, \hat{\alpha}, \hat{\beta}$,
    \begin{align}
        \expectGOE  \Delta_{\epsilon}(M; x, S) = K^{(5)}_N e^{2N(x^2 - \epsilon^2)}\int &d\vec{x}_1 d\vec{x}_2  \iint_{-\infty}^{\infty} dq_1dq_2 \int d^2u\int_{0}^{\pi/2}d\theta \sin2\hat{\theta}\notag\\
    &\exp\left\{ -\frac{N}{2}\Tr Q_B^2 - iN\Tr SB + ixN\Tr Q_B + \epsilon N\Tr Q_B\sigma\right\}\notag\\
    &\exp\left\{-N(q_1^2 + q_2^2)  - 2Nix (q_1 + q_2) - 2N\epsilon(q_2 - q_1)\cos2\hat{\theta} - 2N u \bar{u}\right\}\notag\\
    &\prod_{j=1}^2\det\left(G_j + B + iS\right)|q_1 - q_2|^2\label{eq:nock8_prime} 
\end{align}
with $K^{(5)} = (2\pi)^3 K^{(4)}_N $ and now 
\begin{align}
        G_1G_2 &= q_1q_2 - \bar{u}u\label{eq:G_sim1_prime2}\\
        G_1 + G_2 &= q_1 + q_2 \label{eq:G_sim2_prime2}.
    \end{align}
We form an Hermitian matrix \begin{equation}
        R = \left(\begin{array}{cc} q_1 & \bar{u}\\ u & q_2\end{array}\right)
    \end{equation}
    and so (\ref{eq:nock8_prime}) is rewritten as \begin{align}
              \expectGOE  \Delta_{\epsilon}(M; x, S) = K^{(6)}_Ne^{2N(x^2 - \epsilon^2)}\int &d\vec{x}_1 d\vec{x}_2  \int dR|R_{11} - R_{22}|^2\int_{0}^{\pi/2}d\theta \sin2\hat{\theta}\notag\\
    &\exp\left\{ -\frac{N}{2}\Tr Q_B^2 - iN\Tr SB + ixN\Tr Q_B + \epsilon N\Tr Q_B\sigma\right\}\notag\\
    &\exp\left\{-N\Tr R^2 -2Nix \Tr{R} -2\epsilon N(R_{22} - R_{11})\cos2\hat{\theta} \right\}\prod_{j=1}^2\det\left(G_j + B + iS\right)\label{eq:nock9_prime}   
    \end{align}
    with $K_N^{(6)} = \frac{1}{16\pi^2}K_N^{(5)}$ and 
   \begin{align}
        G_1G_2 &= \det R\label{eq:G_sim1_prime22}\\
        G_1 + G_2 &= \Tr R \label{eq:G_sim2_prime22}.
    \end{align}
    The factor of $(16\pi^2)^{-1}$ comes from the change of variables $(q_1, q_2, u, \bar{u}) \mapsto R$. Indeed, clearly $dq_1dq_2dud\bar{u}  = Z^{-1}dR$ for some constant Jacobian factor $Z$. We can most easily determine $Z$ by integrating against a test function:
    \begin{align}
   \frac{4\pi Vol(U(2))}{Z}  =     \frac{1}{Z}\int_{\text{Herm}(2)} dR e^{-\frac{1}{2}\Tr R^2} &= \iint_{-\infty}^{\infty} dq_1 dq_2 \iint_{-\infty}^{\infty} d\Re{u} ~ d\Im{u} e^{-\frac{1}{2}(q_1^2 + q_2^2 + 2u\bar{u})}= 2\pi^2\notag\\
   \implies Z &= \frac{2Vol(U(2))}{\pi} = 16\pi^2.\notag
    \end{align}
    We diagonalise $R = U\text{diag}(r_1, r_2)U^{\dagger}$, but again the integrand in (\ref{eq:nock9_prime}) is not unitarily invariant in $R$ so we repeat the previous procedure using $$U =e^{i\phi/2} \left(\begin{array}{cc} e^{i\alpha/2} & 0 \\ 0 & e^{-i\alpha/2}\end{array}\right)\left(\begin{array}{cc} \cos\theta & \sin\theta \\ -\sin\theta & \cos\theta\end{array}\right)\left(\begin{array}{cc} e^{i\beta/2} & 0 \\ 0 & e^{-i\beta/2}\end{array}\right).$$
 Overall, integrating out $\phi, \alpha, \beta$, (\ref{eq:nock9_prime}) becomes 
   \begin{align}
              \expectGOE  \Delta_{\epsilon}(M; x, S) = K^{(7)}_Ne^{2N(x^2 - \epsilon^2)}\iint_0^{\pi/2} &d\theta d\hat{\theta} \int d\vec{x}_1 d\vec{x}_2  \iint_{-\infty}^{\infty} dr_1dr_2 |r_1 - r_2|^4 \sin2\theta \cos^22\theta\sin2\hat{\theta} \notag\\
    &\exp\left\{ -\frac{N}{2}\Tr Q_B^2 - iN\Tr SB + ixN\Tr Q_B + \epsilon N\Tr Q_B\sigma\right\}\notag\\
    &\exp\left\{-N(r_1^2 + r_2^2) - 2Ni(x-i\epsilon\cos2\theta\cos2\hat{\theta})r_1 - 2Nix(x + i\epsilon\cos2\theta\cos2\hat{\theta}) \right\}\notag\\&\prod_{j=1}^2\det\left(G_j + B + iS\right)\label{eq:nock10_prime}   
    \end{align} 
    where $K^{(7)} =(2\pi)^3 K^{(6)} $  and now \begin{align}
                G_1G_2 &= r_1r_2,\label{eq:G_sim1_prime3}\\
        G_1 + G_2 &= r_1 + r_2 \label{eq:G_sim2_prime3}\\
        \iff \{G_1, G_2\} &= \{r_1, r_2\} \label{eq:G_are_r}.
    \end{align}
    We can now clearly take $r_j = G_j$ without loss of generality.
\jstat{The terms $\det(r_j + B + iS)$ and $e^{-iN\Tr SB}$ depend on the eigenvectors of $B$ and prevent an application of the integral formula of \cite{fyodorov2002characteristic} as used by \cite{nock}. In fact, it is possible the adapt this integral formula for use in the presence of the term $e^{-iN \Tr SB}$, as seen in Lemma \ref{lem:fyod_general}.}

\jstat{Since $S$ has all entries of order $N^{-1}$, we can expand the nuisance determinants: \begin{align}
    \det(r_j + B + iS) = \prod_{i=1}^2 (r_j + \lambda^{(B)}_i) (1 + o(1)).\label{eq:simple_det_expan}
\end{align}
For this step to be legitimate in the sense of asymptotic expansions, we must have that the error term is uniformly small in the integration variables $\vec{x}_1, \vec{x}_2, r_1, r_2, \theta, \hat{\theta}$. Note  that the integrand in (\ref{eq:nock10_prime}) is analytic in $r_1, r_2$ and so we can deform the contours of integration from $(-\infty, \infty)$ to $\Gamma$, a contour that, say, runs from $-\infty$ along the real line to $-1$ and then follows the unit semi-circle in the upper half plane to $1$ before continuing to $\infty$ along the real line. We show an example contour in Figure \ref{fig:r1r2_contour}. It is now clear that $r_1, r_2$ are bounded away from $0$ and so the error terms in (\ref{eq:simple_det_expan}) are uniform, so giving
   \begin{align}
              \expectGOE  \Delta_{\epsilon}(M; x, S) = K^{(7)}_Ne^{2N(x^2 - \epsilon^2)}\iint_0^{\pi/2} &d\theta d\hat{\theta} \int d\vec{x}_1 d\vec{x}_2  \iint_{-\infty}^{\infty} dr_1dr_2 |r_1 - r_2|^4 \sin2\theta \cos^22\theta\sin2\hat{\theta} \notag\\
    &\exp\left\{ -\frac{N}{2}\Tr Q_B^2 - iN\Tr SB + ixN\Tr Q_B + \epsilon N\Tr Q_B\sigma\right\}\notag\\
    &\exp\left\{-N(r_1^2 + r_2^2) - 2Ni(x-i\epsilon\cos2\theta\cos2\hat{\theta})r_1 - 2Nix(x + i\epsilon\cos2\theta\cos2\hat{\theta}) \right\}\notag\\&\prod_{i,j=1}^2\det\left(r_j + \lambda^{(B)}_i\right)(1 + o(1))\label{eq:nock10_prime_after_exp}   
    \end{align}}
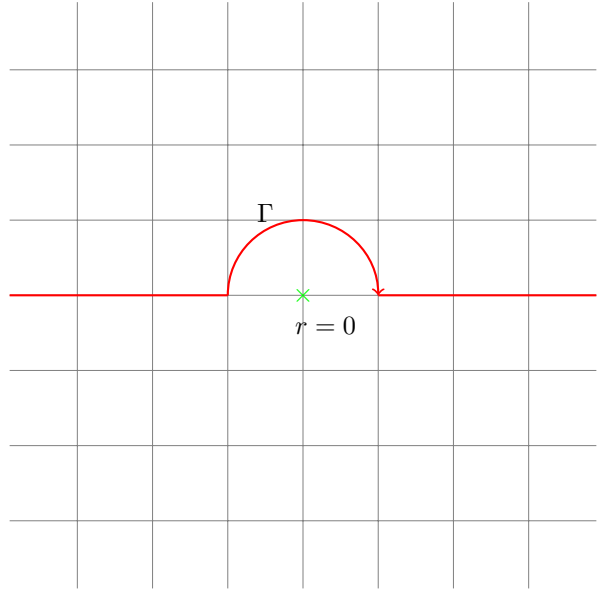
\begin{figure}
    \centering
    \begin{tikzpicture}
    \draw[step=1cm,gray,very thin] (-1.9,-1.9) grid (5.9,5.9);
    \draw[red,thick] (-1.9,2) -- (1,2);
    \draw[->,red,thick] (1,2) arc (180:0:1cm);
    \draw[red,thick] (3,2) -- (5.9,2);
    \node[green] at (2,2){$\times$};
    \node at (2.3, 1.6) {$r=0$};
    \node at (1.5,3.1){$\Gamma$};
    \end{tikzpicture}
    
    \caption{Example contour $\Gamma$ used for the $r_1, r_2$ integrals to keep away from the origin (denoted by the green cross).}
    \label{fig:r1r2_contour}
\end{figure}

\jstat{Lemma \ref{lem:fyod_general} can now be applied:}
\begin{align}
  \expectGOE  \Delta_{\epsilon}(M; x, S) = K^{(8)}_N e^{2N(x^2 - \epsilon^2)}&\left(1 + o(1)\right)\notag \\
  \iint_0^{\pi/2} &d\theta   d\hat{\theta}\int_{\text{Sym}_{\geq 0}(2)} dQ_B \iint_{\Gamma} dr_1dr_2 \cos^22\theta \sin2\theta   \sin2\hat{\theta}\notag\\
  &  \exp\left\{ -\frac{N}{2}\Tr Q_B^2+ ixN\Tr Q_B + \epsilon N\Tr Q_B\sigma\right\}\notag\\
    &\exp\left\{-N(r_1^2 + r_2^2) - 2Ni(x-i\epsilon\cos2\theta\cos2\hat{\theta})r_1 - 2Nix(x + i\epsilon\cos2\theta\cos2\hat{\theta}) \right\} \notag\\
    &\jstat{\prod_{j=1}^r\left(1 + is_j\Tr Q_B - p_{11}p_{22}s_j^2\right)^{-1/2}}\notag \\
    &\prod_{i,j=1}^{2} \left(r_j + \lambda^{(B)}_i\right)|r_1 - r_2|^4(r_1r_2)^{N-2}(\det Q_B)^{\frac{N-3}{2}}, \label{eq:nock12_prime}   
\end{align}

where $p_{ij}$ are the entries of the matrix $Q_B$ and $K^{(8)} = \frac{\pi^N \pi^{-1/2}}{\Gamma\left(\frac{N}{2}\right)\Gamma\left(\frac{N-1}{2}\right)} K^{(7)}_N$.

We now wish to diagonalise $Q_B$ and integrate out its eigenvectors, but as before (around (\ref{eq:nock9_prime})) the integrand is not invariant under the action of the orthogonal group on $Q_B$ and so we instead diagonalise $Q_B = O\text{diag}(p_1, p_2)O^T$ and parametrize $O$ as \begin{equation}
    O = \left(\begin{array}{cc} \cos\theta' &\sin\theta' \\ -\sin\theta' & \cos\theta'\end{array}\right)\label{eq:orthogonal_param}
\end{equation}
but we must be careful to choose domain of integration for $\theta$ and $(p_1, p_2)$ such that the transformation is a bijection. Consider a general positive semi-definite symmetric matrix $$Q_B = \left(\begin{array}{cc} a & c \\ c & b \end{array}\right).$$ Solving for the eigenvalues gives two choices for $(p_1, p_2)$ because of the arbitrary ordering of the eigenvalues. We want a simple product domain for the $(p_1, p_2)$ integrals and both eigenvalues are non-negative, so we choose $(p_1, p_2) \in (\mathbb{R}_{\geq 0})^2$.  One can easily find that \begin{align}
    c &= \frac{p_2 - p_1}{2}\sin2\theta\\
    a &= \frac{p_1 + p_2 + (p_1 - p_2)\cos2\theta}{2}\\
        b &= \frac{p_1 + p_2 + (p_2 - p_1)\cos2\theta}{2}
\end{align}
and so we see immediately that the domain of integration of $\theta$ must be restricted to an interval of length $\pi$ to obtain a bijection. But further, because of the chosen domain for $(p_1, p_2)$ the quantity $(p_1 - p_2)$ takes all values in $\mathbb{R}$ and thus we must in fact restrict $\theta$ to, say, $[0, \pi/2)$ to obtain a bijection.
 The Jacobian of this transformation is $|p_1 - p_2| $ and further \begin{align}
    p_{11}p_{22} & = (p_1\cos^2\theta' + p_2 \sin^2\theta')(p_2\cos^2\theta' + p_1\sin^2\theta')\notag\\
    &=(p_1^2 + p_2^2)(\cos\theta'\sin\theta')^2 + p_1p_2(\cos^4\theta' + \sin^4\theta')\notag\\
    &=  \frac{1}{4}\sin^22\theta' (p_1^2 + p_2^2) + \frac{1}{4}\left(3 + 4\cos4\theta'\right)p_1p_2 
\end{align}
and so we get 

\begin{align}
                  \expectGOE  \Delta_{\epsilon}(M; x, S) = K^{(8)}_N e^{2N(x^2 - \epsilon^2)}&\left(1 + o(1)\right)\iiint_0^{\pi/2} d\theta d\theta'd\hat{\theta}\iint_0^{\infty}dp_1dp_2 \iint_{\Gamma} dr_1dr_2\notag\\
                  &|r_1 - r_2|^4(r_1r_2)^{N-2}(p_1p_2)^{\frac{N-3}{2}}\cos^22\theta \sin2\theta \sin2\hat{\theta}\notag\\
  &  \exp\left\{ -\frac{N}{2}(p_1^2 + p_2^2)  + iN(x-i\epsilon\cos2\theta')p_1 + iN(x+i\epsilon\cos2\theta')p_2 \right\}\notag\\
    &\exp\left\{-N(r_1^2 + r_2^2) - 2Ni(x-i\epsilon\cos2\theta\cos2\hat{\theta})r_1 - 2Nix(x + i\epsilon\cos2\theta\cos2\hat{\theta}) \right\} \notag\\
    &\prod_{i,j=1}^{2} \left(r_j + p_i\right)J_1(p_1, p_2, \theta'; \{s_j\}_{j=1}^r, N)
 \label{eq:nock13_prime}   
\end{align}

where
 \jstat{
\begin{equation}
J_1(p_1, p_2, \theta'; \{s_j\}_{j=1}^r, N) =\prod_{j=1}^r\left(1 + is_j(p_1 + p_2) - s_j^2\left[\frac{1}{4}\sin^22\theta' (p_1^2 + p_2^2) + \frac{1}{4}\left(3 + 4\cos4\theta'\right)p_1p_2 \right]\right)^{-1/2}.
\end{equation}}

Now let us define the functions \begin{align}
    \psi^{(\pm)}_U(z; x; \epsilon) &= \frac{1}{2}z^2 \pm i(x - i\epsilon)z - \frac{1}{2}\log z\\
      \psi^{(\pm)}_L(z; x; \epsilon) &= \frac{1}{2}z^2 \pm i(x + i\epsilon)z - \frac{1}{2}\log z\\
\end{align}
and also \begin{equation}
    J_2(r_1, r_2, p_1, p_2) = |r_1 - r_2|^4 |p_1 - p_2| (r_1r_2)^{-2} (p_1p_2)^{-\frac{3}{2}} (r_1 + p_1)(r_1 + p_2)(r_2 + p_1)(r_2 + p_2)
\end{equation}
and then we finally rewrite (\ref{eq:nock13_prime}) as 
\begin{align}
                  \expectGOE  \Delta_{\epsilon}(M; x, S) = K^{(8)}_N e^{2N(x^2 - \epsilon^2)}\left(1 + o(1)\right)\iiint_0^{\pi/2} &d\theta d\theta'd\hat{\theta}\iint_0^{\infty}dp_1dp_2 \iint_{\Gamma} dr_1dr_2\notag\\
                  &J_1(p_1, p_2, \theta'; S, N)J_2(r_1,r_2, p_1, p_2)\cos^22\theta \sin2\theta \sin2\hat{\theta}\notag\\
& \exp\Bigg\{-N\Bigg(2\psi^{(+)}_L(r_1; x; \epsilon\cos2\theta\cos2\hat{\theta}) +2\psi^{(+)}_U(r_2; x; \epsilon\cos2\theta\cos2\hat{\theta})\notag\\& +\psi^{(-)}_L(p_1; x; \epsilon\cos2\theta')+\psi^{(-)}_U(p_2; x; \epsilon\cos2\theta')\Bigg)\Bigg\}.
 \label{eq:nock_lemma1_final}   
\end{align}
\end{proof}

We will need the asymptotic behaviour of the constant $K_N$ defined in Lemma \ref{lemma:nock_deformed}.

\begin{lemma}\label{lemma:K_N_asymp}
    As $N\rightarrow\infty$ \begin{equation}
      K_N\sim \frac{ (-i)^N N^{\frac{9}{2}}}{4\sqrtsign{2}\pi^{\frac{5}{2}}}(2e)^N.
    \end{equation}
\end{lemma}
\begin{proof}
Using Stirling's formula for the Gamma function gives \begin{align}
    K_N &\sim \frac{N^{N+3}(-i)^N}{\pi^{3/2}} N^{-\frac{N}{2} + \frac{1}{2}}(N-1)^{-\frac{N}{2} + 1} 2^{\frac{N}{2} - \frac{1}{2}} 2^{\frac{N}{2} - 1} e^{\frac{N}{2}} e^{\frac{N}{2} - \frac{1}{2}} 
    \left(2\pi\right)^{-1}\notag\\
  &=  \frac{N^{N+3}(-i)^N}{\pi^{3/2}} N^{-N} N^{\frac{3}{2}} 2^{N} 2^{-\frac{5}{2}} e^{N} e^{-\frac{1}{2}}\pi^{-1} \left(\frac{N-1}{N}\right)^{-\frac{N}{2} + 1}\notag\\
  &\sim \frac{ (-i)^N N^{\frac{9}{2}}}{4\sqrtsign{2}\pi^{\frac{5}{2}}}(2e)^N.
\end{align}
\end{proof}
Building on Lemma \ref{lemma:nock_deformed}, we can prove a generalisation of Theorem 2.8 from \cite{auffinger2013random}, namely Theorem \ref{thm:auff2.8}.
\auffindk*
\begin{proof}

Combining Lemmata \ref{lemma:conditional_dist} and \ref{lemma:kac_rice} and observing that the integrand in the Kac-Rice formula of Lemma \ref{lemma:kac_rice} is spherically symmetric, we obtain \begin{align}\label{eq:thm28_1}
\expect C_{N}^{h}(\sqrtsign{N}u) &= \underbrace{\left(2(N-1)(H-1)H\right)^{\frac{N-1}{2}} \omega_N \frac{\jstat{e^{-\frac{\vec{v}^2}{2H}}}}{(2\pi H)^{(N-1)/2}}}_{\defeq \Omega_N} \int_{-\infty}^{u_N} dx ~ \frac{1}{\sqrtsign{2\pi}t} e^{-\frac{x^2}{2t^2}}  \mathbb{E}^{N-1}_{GOE} |\det(M - xI + S)| 
\end{align} 
where $$ u_N = u\sqrtsign{\frac{HN}{2(N-1)(H-1)}}, $$ the variance $t^2 = \frac{H}{2(N-1)(H-1)}$,  $\omega_N = 2\pi^{N/2}/\Gamma(N/2)$ is the surface area of the $N-1$ sphere and $S$ and $\vec{v}$ are defined in Lemma \ref{lemma:conditional_dist}. Note that the first term in $\Omega_N$ comes from the expression (\ref{eq:hij_goe}) and the third term from \jstat{(\ref{eq:derivs_exp_hi}) and} (\ref{eq:derivs_cov_hihj}), i.e. this is the density of $\nabla h$ evaluated at $0$ as appears in Lemma \ref{lemma:kac_rice}. \jstat{The conditions for Lemma \ref{lemma:nock_deformed} are shown to be met in Lemma \ref{lemma:conditional_dist}}, so we obtain 
\begin{align}
    \expect C_{N}^{h}(\sqrtsign{N}u) = \Omega_N K_{N-1}\sqrtsign{\frac{2(N-1)(H-1)}{H}} &\left(1 + o(1)\right)\int_{-\infty}^{u_N}dx ~  \frac{1}{\sqrtsign{2\pi}}  \lim_{\epsilon\searrow 0}\iiint_0^{\pi/2} d\theta d\hat{\theta}d\theta' \iint_0^{\infty}dp_1dp_2 \iint_{\Gamma} dr_1dr_2\notag\\
                  &J_1(p_1, p_2, \theta'; \{s_j\}_{j=1}^r, N-1)J_2(r_1,r_2, p_1, p_2)\cos^22\theta \sin2\theta\sin2\hat{\theta}\notag\\
& \exp\Bigg\{-(N-1)\Bigg(2\psi^{(+)}_L(r_1; x; \epsilon\cos2\theta\cos2\hat{\theta}) +2\psi^{(+)}_U(r_2; x; \epsilon\cos2\theta\cos2\hat{\theta})\notag\\& +\psi^{(-)}_L(p_1; x; \epsilon\cos2\theta')+\psi^{(-)}_U(p_2; x; \epsilon\cos2\theta') - 
\frac{H+1}{H}x^2\Bigg)\Bigg\}\notag \\
 =c_{N,H}\int_{-\infty}^{u_N}dx ~   \lim_{\epsilon\searrow 0}\iiint_0^{\pi/2} d\theta d\hat{\theta}d\theta'\iint_0^{\infty}dp_1&dp_2 \iint_{\Gamma} dr_1dr_2 \notag\\
                  &J_1(p_1, p_2, \theta'; \{s_j\}_{j=1}^r, N-1)J_2(r_1,r_2, p_1, p_2)\cos^22\theta \sin2\theta\sin2\hat{\theta}\notag\\
& \exp\Bigg\{-(N-1)\Bigg(2\psi^{(+)}_L(r_1; x; \epsilon\cos2\theta\cos2\hat{\theta}) +2\psi^{(+)}_U(r_2; x; \epsilon\cos2\theta\cos2\hat{\theta})\notag\\& +\psi^{(-)}_L(p_1; x; \epsilon\cos2\theta')+\psi^{(-)}_U(p_2; x; \epsilon\cos2\theta') - \frac{H+1}{H}x^2\Bigg)\Bigg\}  \label{eq:thm28_mid}
\end{align}

where we have defined the constant \begin{equation}\label{eq:constant_final_defn}
    c_{N,H} = \frac{\Omega_N K_{N-1}\sqrtsign{(H-1)(N-1)}}{\sqrtsign{H\pi}}(1 + o(1)).
\end{equation}
We pause now to derive the asymptotic form of $c_{N,H}$. The vector $\vec{v}$ was defined in Lemma \ref{lemma:conditional_dist} and has entries of order $N^{-1/2}$, so $\vec{v}^2 = \mathcal{O}(1)$. Using Stirling's formula for the Gamma function \begin{align}
    \Omega_N &\sim 2 (N-1)^{\frac{N-1}{2}} (H-1)^{\frac{N-1}{2}}\pi^{1/2} N^{-\frac{N}{2} + \frac{1}{2}} 2^{\frac{N}{2} - \frac{1}{2}} e^{\frac{N}{2}} \left(2\pi\right)^{-1/2}\jstat{e^{-\frac{\vec{v}^2}{2H}}}\notag\\
    &=(H-1)^{\frac{N-1}{2}} (2e)^{\frac{N}{2}} \left(\frac{N-1}{N}\right)^{\frac{N-1}{2}}\jstat{e^{-\frac{\vec{v}^2}{2H}}}\notag\\
    &\sim (H-1)^{\frac{N-1}{2}} (2e)^{\frac{N}{2}} e^{-1/2}\jstat{e^{-\frac{\vec{v}^2}{2H}}}\notag\\ 
    \implies \frac{\Omega_N \sqrtsign{(H-1)(N-1)}}{\sqrtsign{H\pi}} &\sim (H-1)^{\frac{N}{2}} (2e)^{\frac{N}{2}} e^{-1/2} H^{-1/2}\pi^{-1/2} (N-1)^{1/2}\jstat{e^{-\frac{\vec{v}^2}{2H}}}
\end{align}
and so Lemma \ref{lemma:K_N_asymp} gives \begin{align}
    c_{N,H} &\sim \frac{ (-i)^{N-1} (N-1)^{\frac{9}{2}}}{4\sqrtsign{2}\pi^{\frac{5}{2}}}(2e)^{N-1}(H-1)^{\frac{N}{2}} (2e)^{\frac{N}{2}} e^{-1/2} H^{-1/2}\pi^{-1/2} (N-1)^{1/2}\jstat{e^{-\frac{\vec{v}^2}{2H}}}\notag\\
    &\sim \frac{(-i)^{N-1} N^5}{4\pi^3 H^{1/2}} (2e)^{\frac{3}{2}(N-1)} (H-1)^{\frac{N}{2}}\jstat{e^{-\frac{\vec{v}^2}{2H}}}.\label{eq:c_NH_asymp}
\end{align}
In the style of \cite{soton29213}, the multiple integral in (\ref{eq:thm28_mid}) can be written as an expansion over saddle points and saddle points of the integrand restricted to sections of the boundary. Recalling the form of $\psi^{(\pm)}_U$ and $\psi^{(\pm)}_L$, we see that the integrand vanishes on the boundary and so we focus on the interior saddle points. Let us define the exponent function \begin{equation}
        \Phi(r_1, r_2, p_1, p_2, x; S, \epsilon) = 2\psi^{(+)}_L(r_1; x, \epsilon)+ 2\psi^{(+)}_U(r_2; x, \epsilon) + \psi^{(-)}_L(p_1; x, \epsilon) + \psi^{(-)}_U(p_2; x, \epsilon) - \frac{(H+1)}{H} x^2
    \end{equation}
    It is clear that the $\cos\theta, \cos\hat{\theta}$ and $\cos\theta'$ terms in the exponent of (\ref{eq:thm28_mid}) do not affect the saddle point asymptotic analysis, since we take the limit $\epsilon\rightarrow 0$, and $\theta, \hat{\theta}, \theta'\in [0,\pi/2)$ and it is only the signs of the $\mathcal{O}(\epsilon)$ terms that are significant. Therefore, to simplify the exposition, we will suppress these terms.
    The $(r_1,r_2,p_1,p_2)$ components of $\grad\Phi$ are of the form \begin{equation}
        z\mapsto z \pm i(x\pm i\epsilon) - \frac{1}{2z}
    \end{equation}
    and so the only saddle in $\Phi$ restricted to those components is at \begin{align}
        r_1 &= \frac{-i(x+i\epsilon) + (2-(x+i\epsilon)^2)^{1/2}}{2}\defeq z^{(+)}_L\label{eq:saddle_loc_r1}\\
           r_2 &= \frac{-i(x-i\epsilon) + (2-(x-i\epsilon)^2)^{1/2}}{2}\defeq z^{(+)}_U\\
                   p_1 &= \frac{i(x+i\epsilon) + (2-(x+i\epsilon)^2)^{1/2}}{2}\defeq z^{(-)}_L\\
                           p_2 &= \frac{i(x-i\epsilon) + (2-(x-i\epsilon)^2)^{1/2}}{2}\defeq z^{(-)}_U\label{eq:saddle_loc_p2}.
    \end{align}
    To deform the $(r_1,r_2,p_1,p_2)$ contours through this saddle, we are required to choose a branch of the functions in (\ref{eq:saddle_loc_r1} - \ref{eq:saddle_loc_p2}). Each has branch points at  $\pm\sqrtsign{2} + i\epsilon$ or $\pm\sqrtsign{2} -i\epsilon$. Since the initial contour of $x$ integration lies along the real line, we take the following branch cuts in the complex $x$ plane and respective angle ranges (see Figure \ref{fig:branch_cut}) \begin{align}
    [\sqrtsign{2} + i\epsilon, \sqrtsign{2} + i\infty],~~ & [\pi/2, 5\pi/2] \label{eq:branch1}\\
    [\sqrtsign{2} - i\epsilon, \sqrtsign{2} - i\infty],~~ & [-\pi/2, 3\pi/2] \\
    [-\sqrtsign{2} + i\epsilon, -\sqrtsign{2} + i\infty],~~ & [\pi/2, 5\pi/2] \\
    [-\sqrtsign{2} - i\epsilon, -\sqrtsign{2} - i\infty],~~ & [-\pi/2, 3\pi/2]\label{eq:branch4}.    \end{align}

\begin{figure}
    \centering
    \begin{tikzpicture}
    \draw[step=1cm,gray,very thin] (-1.9,-1.9) grid (5.9,5.9);
    \def\R{0.4};
    \def\gapY{1.2};
    \def\gapX{2.83};
    \coordinate (A) at ($(2,2) + (-\gapX/2, \gapY/2)$);
    \coordinate (B) at ($(A) + (\gapX, 0)$);
    \coordinate (C) at ($(A) - (0,\gapY)$);
    \coordinate (D) at ($(B) - (0,\gapY)$);
    \node at (A){$\times$};
    \node at (B){$\times$};
    \node at (C){$\times$};
    \node at (D){$\times$};
    \draw[thick,red,branch cut] (A) to ($(A) + (0,4-\gapY/2)$);
    \draw[thick,red,branch cut] (B) to ($(B) + (0,4-\gapY/2)$);
    \draw[thick,red,branch cut] (C) to ($(C) + (0,-4+\gapY/2)$);
    \draw[thick,red,branch cut] (D) to ($(D) + (0,-4+\gapY/2)$);
    \node at ($(A) + (-1.2,0)$){$-\sqrtsign{2} + i\epsilon$};
    \node at ($(B) + (1.2,0)$){$\sqrtsign{2} + i\epsilon$};
    \node at ($(C) + (-1.2,0)$){$-\sqrtsign{2} - i\epsilon$};
    \node at ($(D) + (1.2,0)$){$\sqrtsign{2} - i\epsilon$};
    \draw[dashed, ->] ($(A) + (\R,0)$) arc (0:360:\R cm);
    \draw[dashed, ->] ($(B) + (\R,0)$) arc (0:360:\R cm);
    \draw[dashed, ->] ($(C) + (\R,0)$) arc (0:360:\R cm);
    \draw[dashed, ->] ($(D) + (\R,0)$) arc (0:360:\R cm);
    \end{tikzpicture}
    \caption{The choice of branch for the $x$ integral in the proof of Theorem \ref{thm:auff2.8}.}
    \label{fig:branch_cut}
\end{figure}
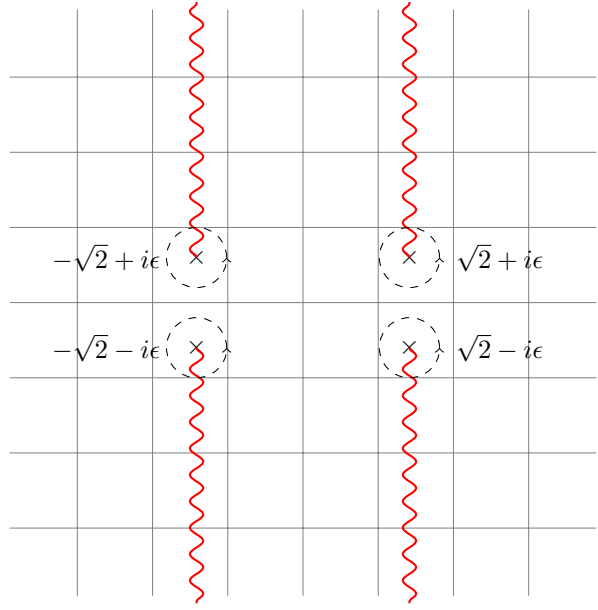

It is simple to compute $\psi_U^{(\pm)}(z^{(\pm)}_U)$ and  $\psi_L^{(\pm)}(z^{(\pm)}_L)$:
\begin{align}
    \psi_L^{(+)}(z^{(+)}_L) = \frac{1}{4}\left(1+(x+i\epsilon)^2 + \log 2\right) + \frac{1}{4}\log 2 +\frac{1}{4}i(x+i\epsilon)\left(2 - (x+i\epsilon)^2\right)^{1/2} - \frac{1}{2}\log\left[ -i(x+i\epsilon) + \left(2-(x+i\epsilon)^2\right)^{1/2}\right]\label{eq:psi_L_plus}\\
       \psi_U^{(+)}(z^{(+)}_U) = \frac{1}{4}\left(1+(x-i\epsilon)^2 + \log 2\right) + \frac{1}{4}\log 2 +\frac{1}{4}i(x-i\epsilon)\left(2 - (x-i\epsilon)^2\right)^{1/2} - \frac{1}{2}\log\left[ -i(x-i\epsilon) + \left(2-(x-i\epsilon)^2\right)^{1/2}\right]\\
           \psi_L^{(-)}(z^{(-)}_L) = \frac{1}{4}\left(1+(x+i\epsilon)^2 + \log 2\right) + \frac{1}{4}\log 2 -\frac{1}{4}i(x+i\epsilon)\left(2 - (x+i\epsilon)^2\right)^{1/2} - \frac{1}{2}\log\left[ i(x+i\epsilon) + \left(2-(x+i\epsilon)^2\right)^{1/2}\right]\\
       \psi_U^{(-)}(z^{(-)}_U) = \frac{1}{4}\left(1+(x-i\epsilon)^2 + \log 2\right) + \frac{1}{4}\log 2 -\frac{1}{4}i(x-i\epsilon)\left(2 - (x-i\epsilon)^2\right)^{1/2} - \frac{1}{2}\log\left[ i(x-i\epsilon) + \left(2-(x-i\epsilon)^2\right)^{1/2}\right]\label{eq:psi_U_minus}.
\end{align}
  Let us consider $x$ still restricted to the real line. We are free to restrict to $\epsilon>0$ and then $x\pm i\epsilon$ lies just above (below) the real line. For $x<-\sqrtsign{2}$ the angle from all four branch points is $\pi$ and so we obtain \begin{align}
      \Phi_{(4)}(x) \defeq \lim_{\epsilon\rightarrow 0} \Phi\left(z_L^{(+)}, z_U^{(+)}, z_L^{(-)}, z_U^{(-)}, x; \epsilon\right) &=              \frac{3}{2}\left(1+x^2+\log 2\right) + \frac{3}{2}\log 2 - \frac{1}{2}x\sqrtsign{x^2 -2}-2\log\left[-ix + i\sqrtsign{x^2 - 2}\right] \notag\\
      &- \log\left[ix + i\sqrtsign{x^2 - 2}\right] - \frac{H+1}{H}x^2 \notag\\
      &=\frac{3}{2}\left(1+\log 2\right) + \frac{H-2}{2H}x^2 + \frac{3}{2}\log 2 - \frac{1}{2}x\sqrtsign{x^2 -2}-\log\left[-ix + i\sqrtsign{x^2 - 2}\right]\notag\\
      & ~~~- \log 2\notag\\
      &=\frac{3}{2}\left(1+\log 2\right)+ \frac{H-2}{2H}x^2 + \frac{1}{2}\log 2 - \frac{1}{2}x\sqrtsign{x^2 -2}-\log\left[-x + \sqrtsign{x^2 - 2}\right]\notag\\ 
      & ~~~ - \log i\notag\\
      &= \frac{3}{2}\left(1+\log 2\right) +\frac{H-2}{2H}x^2+ I_1(x; \sqrtsign{2}) - \log i\label{eq:phi_4_below_bulk}
  \end{align}
 However for $-\sqrtsign{x} < x < \sqrtsign{2}$ the angles about the branch points are $\pi, \pi, 2\pi, 0$ in the order of (\ref{eq:branch1}-\ref{eq:branch4}). It follows that the square root terms in both of $\psi^{(\pm)}_L(z^{(\pm)}_L)$ and both of $\psi^{(\pm)}_U(z^{(\pm)}_U)$ have opposite signs and so \begin{align}
     \Phi_{(4)}(x) &= \frac{3}{2}\left( 1+ \log2\right) + \frac{H-2}{2H}x^2- \frac{3}{2}\log(-2) + \frac{3}{2}\log 2\notag\\
     &= \frac{3}{2}\left( 1 + \log2\right) + \frac{H-2}{2H}x^2
     - \frac{3}{2}\log(-1)\label{eq:phi_4_in_bulk}.
 \end{align}
 Finally, the above reasoning can be trivially extended to $x>\sqrtsign{2}$ to obtain \begin{align}\label{eq:phi_4_above_bulk}
     \Phi_{(4)}(x) = \frac{3}{2}\left(1 + \log{2}\right) + \frac{H-2}{2H}{x^2} + I_1(-x; \sqrtsign{2}) - \log{i}.
 \end{align}
  It is apparent from (\ref{eq:phi_4_below_bulk})\footnote{Note that $I_1(x;\sqrtsign{2})$ is monotonically decreasing on $(-\infty, -\sqrtsign{2}]$.}, (\ref{eq:phi_4_in_bulk}) and (\ref{eq:phi_4_above_bulk}) that the branch choice (\ref{eq:branch1}-\ref{eq:branch4}) and deforming through each of the saddles of in $(r_1, r_2, p_1, p_2)$ gives a contour of steepest descent in $x$ with the critical point being at $x=0$.

We are thus able to write down the leading order asymptotics for (\ref{eq:thm28_mid}) for all real $u$ coming either from the end-point $x=\sqrtsign{2}u/E_{\infty}$ or the critical point $x=0$. We begin with $u< -E_{\infty}$ by using (\ref{eq:phi_4_below_bulk}):
 \begin{align}
    \frac{1}{N}\log\mathbb{E}C^h_{N}(\sqrtsign{N}u) &\sim -\frac{3}{2}\log{2}  -\frac{3}{2} -\frac{H-2}{2H}\frac{Hu^2}{2(H-1)} - I_1(u; E_{\infty}) + \log{i}+ \frac{1}{N}\
    \log c_{N,H}\notag\\
    &\sim \frac{1}{2}\log(H-1) - \frac{H-2}{4(H-1)}u^2 - I_1(u; E_{\infty})
\end{align}
since by (\ref{eq:c_NH_asymp}) \begin{align}
    \log{c_{N,H}} \sim \frac{1}{2}N\log(H-1) + \frac{3}{2}(N-1)(1 + \log{2}) + (N-1)\log(-i).
\end{align}
For $-E_{\infty} \leq u < 0$ we use (\ref{eq:phi_4_in_bulk}):
\begin{align}
    \frac{1}{N}\log\mathbb{E}C^h_{N}(\sqrtsign{N}u) &\sim -\frac{3}{2}\log{2}  -\frac{3}{2} -\frac{H-2}{2H}\frac{Hu^2}{2(H-1)} + \frac{3}{2}\log(-1)+ \frac{1}{N}\
    \log c_{N,H}\notag\\
    &\sim \frac{1}{2}\log(H-1) - \frac{H-2}{4(H-1)}u^2  
\end{align}
since $\frac{3}{2}\log(-1) = \log\left((-1)^{1/2}\right) =\log{i}$. Finally, for $u\geq 0$ the leading contribution comes from the critical point, so 
\begin{align}
    \frac{1}{N}\log\mathbb{E}C^h_{N}(\sqrtsign{N}u) &\sim -\frac{3}{2}\log{2}  -\frac{3}{2} + \frac{3}{2}\log(-1)+ \frac{1}{N}\
    \log c_{N,H}\notag\\
    &\sim \frac{1}{2}\log(H-1).
\end{align}
\end{proof}

We are in-fact able to obtain the exact leading order term in the expansion of $\mathbb{E}C^h_{N}(\sqrtsign{N}u)$ in the case $u<-E_{\infty}$, namely Theorem \ref{thm:exact_term}.
\auffexact*

\begin{proof}
We begin by deriving an alternative form for $h$. For $v>\sqrtsign{2}$ \begin{align}
    h(v)^2 &= \frac{|v - \sqrtsign{2}| + | v + \sqrtsign{2}| + 2|v^2 - 2|^{\frac{1}{2}}}{|v^2 -2|^{\frac{1}{2}}}\notag\\
    &= 2\left( v + |v^2 - 2|^{\frac{1}{2}}\right)|v^2 - 2|^{-\frac{1}{2}}\notag\\
    \implies h(v) &= \sqrtsign{2} \left( v + |v^2 - 2|^{\frac{1}{2}}\right)^{\frac{1}{2}}|v^2 - 2|^{-\frac{1}{4}}\notag\\
   &=2|-v + |v^2 - 2|^{\frac{1}{2}}|^{-\frac{1}{2}}|v^2 - 2|^{-\frac{1}{4}}. \label{eq:alternate_h_form}
\end{align}
This proof now proceeds like that of Theorem \ref{thm:auff2.8} except that we are required to keep track of the exact factors in (\ref{eq:thm28_mid}) and evaluate the $\mathcal{O}(1)$ integrals arising from the saddle point approximation. First note that (using primes to denote $z$ derivatives) \begin{equation}
   { \psi^{(\pm)}_{U,L}}''(z ;x; \epsilon) = 1 + \frac{1}{2z^2}
\end{equation}
and so we abbreviate  ${\psi^{(\pm)}_{U,L}}''= \psi''$. We get the following useful relation (now letting $\epsilon \rightarrow 0$ implicitly for simplicity of exposition) \begin{align}
    \psi''(z^{(\pm)}_{U,L}) &= (z^{(\pm)}_{U,L})^{-2}\left(1 \mp ix z^{(\pm)}_{U,L}\right)\notag \\
    &=\frac{1}{2}(z^{(\pm)}_{U, L})^{-2}\left(2 - x^2  \pm x\sqrtsign{x^2 - 2}\right)\notag \\
    &= i\sqrtsign{x^2 - 2}(z^{(\pm)}_{U, L})^{-1}\label{eq:handy_psi_pp}
\end{align}
where, using our branch choice shown in Figure \ref{fig:branch_cut}, for $x<-\sqrtsign{2}$ the saddle points are \begin{align}
    z_{U,L}^{(\pm)} = \frac{\mp ix + i\sqrtsign{x^2 - 2}}{2}.
\end{align}
We recall the central expression (\ref{eq:thm28_mid}) from the proof of Theorem \ref{thm:auff2.8}:\begin{align}
    \expect C_{N}^{h}(\sqrtsign{N}u) = c_{N,H}\int_{-\infty}^{u_N}dx ~   \lim_{\epsilon\searrow 0}\iiint_0^{\pi/2}& d\theta d\hat{\theta}d\theta'\iint_0^{\infty}dp_1dp_2 \iint_{\Gamma} dr_1dr_2 \notag\\
                  &J_1(p_1, p_2, \theta'; \{s_j\}_{j=1}^r, N-1)J_2(r_1,r_2, p_1, p_2)\cos^22\theta \sin2\theta\sin2\hat{\theta}\notag\\
& \exp\Bigg\{-(N-1)\Bigg(2\psi^{(+)}_L(r_1; x; \epsilon\cos2\theta\cos2\hat{\theta}) +2\psi^{(+)}_U(r_2; x; \epsilon\cos2\theta\cos2\hat{\theta})\notag\\& +\psi^{(-)}_L(p_1; x; \epsilon\cos2\theta')+\psi^{(-)}_U(p_2; x; \epsilon\cos2\theta') - \frac{H+1}{H}x^2\Bigg)\Bigg\}\notag
\end{align}
and we recall the expressions for $J_1, J_2$ from Lemma \ref{lemma:nock_deformed}:
\begin{align}
    J_1(p_1, p_2, \theta'; \{s_j\}_{j=1}^r, N) &=\jstat{\left(1 + iN^{-1/2}s_2(p_1 + p_2) - N^{-1}s_2^2\left[\frac{1}{4}\sin^22\theta' (p_1^2 + p_2^2) + \frac{1}{4}\left(3 + 4\cos4\theta'\right)p_1p_2 \right]\right)^{-1/2}}\notag\\
    & ~~~~~~ \cdot\jstat{ \left(1 + is_1(p_1 + p_2) - s_1^2\left[\frac{1}{4}\sin^22\theta' (p_1^2 + p_2^2) + \frac{1}{4}\left(3 + 4\cos4\theta'\right)p_1p_2 \right]\right)^{-1/2},}\notag\\
        J_2(r_1, r_2, p_1, p_2) &= (r_1 + p_1)(r_2 + p_1)(r_1 + p_2)(r_2 + p_2)|r_1 - r_2|^4 |p_1-p_2| (r_1r_2)^{-2} (p_1p_2)^{-3/2}.\notag
    \end{align} 

We begin by evaluating $J_1$ to leading order at the saddle points:\begin{align}
    \frac{1}{2} \sin^2 2\theta' (z^{(-)})^2 + \frac{1}{4}\left(3+4\cos 4\theta'\right) (z^{(-)})^2 &\equiv q(\theta') (z^{(-)})^2\notag \\
    \implies J_1(z^{(-)}, z^{(-)}, \theta'; \{s_j\}_{j=1}^r, N) &\sim  \jstat{\left(1 + 2iz^{(-)}s_1 - q(\theta')\left(z^{(-)}\right)^2s_1^2\right)^{-1/2}\label{eq:J1_at_saddle}.}
\end{align}
\jstat{Recalling \begin{align}
    x + \sqrtsign{x^2 - 2} = \frac{-2}{-x + \sqrtsign{x^2 - 2}} = -\frac{h(x)^2}{2}\sqrtsign{x^2 - 2}, ~~~~~ (z^{(-)})^2 = -\frac{1}{2}\sqrtsign{x^2 - 2}\left( x + \sqrtsign{x^2 - 2}\right)
\end{align}
we obtain \begin{align}
    J_1 \sim 1 + \frac{1}{4}s_1\sqrtsign{x^2 - 2}h(x)^2 - \frac{1}{4}s_1^2 q(\theta')|x^2 - 2|h(x)^2 \equiv j(x, s_1, \theta').
\end{align}}

We see that $J_2(z^{(+)}, z^{(+)}, z^{(-)}, z^{(-)}) = 0$ and so we are required to expand $J_2$ in the region of $$(r_1, r_2, p_1, p_2) = (z^{(+)}, z^{(+)}, z^{(-)}, z^{(-)}).$$
Following standard steepest descents practice, the integration variables $r_1, r_2, p_1, p_2$ are replaced by scaled variables in the region of the saddle point, i.e. \begin{align}
    r_i &=  z^{(+)} + (N-1)^{-\frac{1}{2}}|{\psi^{(+)}}''(z^{(+)})|^{-\frac{1}{2}}\rho_i\label{eq:exact_term_scaling1}\\
    p_i &=  z^{(-)} + (N-1)^{-\frac{1}{2}}|{\psi^{(-)}}''(z^{(-)})|^{-\frac{1}{2}}\pi_i\label{eq:exact_term_scaling2}
\end{align}
and so \begin{align}
    J_2(r_1, r_2, p_1, p_2) &= (N-1)^{-\frac{5}{2}}|x^2 - 2|^2 (z^{(+)})^{-4}(z^{(-)})^{-3} |{\psi^{(-)}}''(z^{(-)})|^{-\frac{1}{2}}|{\psi^{(+)}}''(z^{(+)})|^{-2}|\rho_1 - \rho_2|^4 |\pi_1 -\pi_2| + o(N^{-\frac{5}{2}}).
\end{align}
Piecing these components together gives \begin{align}
    J_2 J_1 dr_1 dr_2 dp_1dp_2 &= (N-1)^{\jstat{-\frac{9}{2}}} \jstat{j(x, s_1, \theta')} |x^2 - 2|^2 \notag \\
    & ~~~~~~~~~~ |{\psi^{(-)}}''(z^{(-)})|^{-\frac{3}{2}}
    |{\psi^{(+)}}''(z^{(+)})|^{-3}(z^{(+)})^{-4}(z^{(-)})^{-3}\notag \\ & ~~~~~~~~~~ |\rho_1 - \rho_2|^4 |\pi_1 -\pi_2| d\rho_1 d\rho_2 d\pi_1 d\pi_2\notag \\
    &= (N-1)^{-\jstat{\frac{9}{2}}} \jstat{j(x, s_1, \theta')} |x^2 - 2|^{-\frac{1}{4}}  (z^{(+)})^{-1}(z^{(-)})^{-\frac{3}{2}}\notag \\ & ~~~~~~~~~~ |\rho_1 - \rho_2|^4 |\pi_1 -\pi_2| d\rho_1 d\rho_2 d\pi_1 d\pi_2\notag \\
        &= 2(N-1)^{-\jstat{\frac{9}{2}}} \jstat{j(x, s_1, \theta')} |x^2 - 2|^{-\frac{1}{4}}  (z^{(-)})^{-\frac{1}{2}}\notag \\ & ~~~~~~~~~~ |\rho_1 - \rho_2|^4 |\pi_1 -\pi_2| d\rho_1 d\rho_2 d\pi_1 d\pi_2\notag \\
                &= 2^{\jstat{ \frac{3}{2}}}(N-1)^{-\jstat{\frac{9}{2}}} \jstat{j(x, s_1, \theta')} |x^2 - 2|^{-\frac{1}{4}} \left(x + \sqrtsign{x^2 - 2}\right)^{-\jstat{\frac{1}{2}} }\notag \\ & ~~~~~~~~~~ |\rho_1 - \rho_2|^4 |\pi_1 -\pi_2| d\rho_1 d\rho_2 d\pi_1 d\pi_2\label{eq:exact_first_term_jacobian_JJ}.
    \end{align}
Recalling the expression (\ref{eq:alternate_h_form}), we can then write \begin{align}
     J_2 J_1 dr_1 dr_2 dp_1dp_2 &= 2^{\jstat{\frac{3}{2}}}(N-1)^{-\jstat{\frac{9}{2}}}  \jstat{j(x, s_1, \theta') h(-x)} 2^{- \jstat{1}}|\rho_1 - \rho_2|^4 |\pi_1 -\pi_2| d\rho_1 d\rho_2 d\pi_1 d\pi_2\notag \\
     &= 2^{\jstat{\jstat{\frac{1}{2}}}}(N-1)^{-\jstat{\frac{9}{2}}}  \jstat{j(x, s_1, \theta') h(-x)} |\rho_1 - \rho_2|^4 |\pi_1 -\pi_2| d\rho_1 d\rho_2 d\pi_1 d\pi_2
\end{align}
and so using (\ref{eq:c_NH_asymp}), we obtain \begin{align}
      \expect C_{N}^{h}(\sqrtsign{N}u) &\sim \frac{2^{-\jstat{\frac{3}{2}} }N^{\frac{\jstat{1}}{2}}}{\pi^3\sqrtsign{H}}\jstat{e^{-\frac{\vec{v}^2}{2H}}} \frac{Y_{2}^{(4)}}{8} Y_{2}^{(1)} \iint_0^{\pi/2}d\theta d\hat{\theta} ~ \cos^2 2\theta \sin 2\theta\sin2\hat{\theta}\notag\\
      & ~~~\sqrtsign{H-1} \int_{0}^{\pi/2} d\theta' \int_{-\infty}^{\frac{\sqrtsign{2}u}{E_{\infty}}\sqrtsign{\frac{N}{N-1}}} dx ~ h(-x)\jstat{j(x, s_1, \theta')}e^{(N-1)\Theta_H(2^{-\frac{1}{2}}E_{\infty}x)}\label{eq:exact_first_term_midway}
\end{align}
where we have defined the integrals \begin{align}
    Y_{n}^{(\beta)} = \int_{\mathbb{R}^n} d\vec{y} ~ e^{-\frac{1}{2}\vec{y}^2} |\Delta(\vec{y})|^{\beta}
\end{align}
and $\Delta$ is the Vandermonde determinant. Recall that, as in Theorem \ref{thm:auff2.8}, the $x$ integration contour in (\ref{eq:exact_first_term_midway}) is a steepest descent contour and so the leading order term comes from the end point.
Now \begin{align}
   (N-1) \Theta_H\left(\sqrtsign{\frac{N}{N-1}} u\right) &= (N-1)\frac{1}{2}\log(H-1) - N\frac{H-2}{4(H-1)}u^2 - (N-1)I_1\left(\sqrtsign{\frac{N}{N-1}}u; E_{\infty}\right) \notag\\
   &= (N-1)\frac{1}{2}\log(H-1) - N\frac{H-2}{4(H-1)}u^2 - (N-1)I_1\left(u; E_{\infty}\right) - \frac{N-1}{2N}uI_1'(u; E_{\infty}) + \mathcal{O}(N^{-1}) \notag\\
   &= N\Theta_H(u) - \frac{1}{2}\log(H-1) + I_1(u; E_{\infty}) - \frac{1}{2}uI_1'(u; E_{\infty}) + \mathcal{O}(N^{-1})
\end{align}
and so \begin{align}
     \expect C_{N}^{h}(\sqrtsign{N}u) &\sim \frac{2^{-\jstat{\frac{3}{2}}}N^{\frac{-\jstat{1}}{2}}}{24\pi^3\sqrtsign{H}} \jstat{e^{-\frac{\vec{v}^2}{2H}}} Y_{2}^{(4)} Y_{2}^{(1)} \left(\int_{0}^{\pi/2} d\theta'\jstat{j(-v, s_1, \theta')}\right) h(v) e^{N\Theta_H(u)} \frac{e^{I_1(u; E_{\infty}) - \frac{1}{2}u I_1'(u; E_{\infty})}}{\frac{H-2}{2(H-1)}u + I_1'(u; E_{\infty})}\label{eq:exact_first_term_almost}
\end{align}
where we have defined (c.f. \cite{auffinger2013random} Theorem 2.17) $v = -\sqrtsign{2}uE_{\infty}^{-1}.$
It now remains only to evaluate the various constants in (\ref{eq:exact_first_term_almost}) where possible. Firstly observe \begin{align}
    Y_2^{(1)} &= 2\pi \mathbb{E}_{X_1, X_2\overset{i.i.d.}{\sim}\mathcal{N}(0,1)} |X_1 - X_2| = 2\pi \mathbb{E}_{X\sim \mathcal{N}(0, 2)} |X| = 2\sqrtsign{\pi} \int_0^{\infty} xe^{-\frac{x^2}{4}} = 4\sqrtsign{\pi} 
\end{align}
and similarly \begin{align}
    Y_2^{(4)} &= 2\pi \mathbb{E}_{X_1, X_2\overset{i.i.d.}{\sim}\mathcal{N}(0,1)} (X_1 - X_2)^4 = 2\pi \mathbb{E}_{X\sim \mathcal{N}(0, 2)} X^4 = 24 \pi.
\end{align}
For convenience  we define 
\jstat{\begin{align}
    T(v, s_1) = \frac{2}{\pi}\int_{0}^{\pi/2}j(-v, s_1, \theta')d\theta',
\end{align}}

and then collating our results:
\begin{align}
     \expect C_{N}^{h}(\sqrtsign{N}u) &\sim \frac{N^{-\frac{\jstat{1}}{2}}}{\sqrtsign{2\pi H}} \jstat{e^{-\frac{\vec{v}^2}{2H}}}\jstat{T(v, s_1)} h(v) e^{N\Theta_H(u)} \frac{e^{I_1(u; E_{\infty}) - \frac{1}{2}u I_1'(u; E_{\infty})}}{\frac{H-2}{2(H-1)}u + I_1'(u; E_{\infty})}.\label{eq:exact_first_term_end}
\end{align}

\end{proof}

\begin{remark}\label{rem:exact_k_ind}
Having completed the proof of Theorem \ref{thm:exact_term}, we can now explain why this result generalises only part (a) of the analogous Theorem (2.17) from \cite{auffinger2013random}, namely only the case $u<-E_{\infty}$. Recall that, following standard steepest descent practice, we introduced scaled integration variables in the region of the saddle point (\ref{eq:exact_term_scaling1})-(\ref{eq:exact_term_scaling2}) and so arrived at (\ref{eq:exact_first_term_midway}) with the constant factors $Y_2^{(1)}, Y_2^{(4)}$ resulting from the Laplace approximation integrals over the scaled variables. If we take $-E_{\infty} < u < 0$, say, then $z^{(+)}_U + z^{(-)}_L = 0$ and $z^{(+)}_L + z^{(-)}_U = 0$ and so it is the terms $(r_1 + p_2), (r_2 + p_1)$ that vanish at the saddle point rather than $|r_1 - r_2|^4$ and $|p_1 - p_2|$. It follows that the terms $Y_2^{(1)}, Y_2^{(4)}$ are replaced by the integrals \begin{align}
    \int_{\mathbb{R}} d\pi_1 d\pi_2d\rho_1d\rho_2 ~ e^{-\frac{1}{2}(\pi_1^2 + \pi_2^2 + \rho_1^2 + \rho_2^2)} (\rho_1 + \pi_2)(\rho_2 + \pi_1) = 0.
\end{align} 
It is therefore necessary to keep terms to at least the first sub-leading order in the expansion of $J_1J_2$ around the saddle point, however we cannot do this owing the presence of the $o(1)$ term in the constant $c_{N,H}$ as defined in (\ref{eq:constant_final_defn}) which we cannot evaluate.\\
\end{remark}

\begin{remark}
\jstat{Note that setting all the  $\rho_{\ell}^{(N)}=0$ gives $\vec{v} = 0$, $S=0$, hence $s_1=0$ and so $T = 1$. Consequently (\ref{eq:exact_first_term_end}) recovers the exact spherical $H$-spin glass expression in part (a) of Theorem 2.17 in \cite{auffinger2013random}}. \\
\end{remark}

\begin{remark}
 The function $h(v)$ shows up in \cite{auffinger2013random} in the asymptotic evaluation of Hermite polynomials but arises here by an entirely different route.
\end{remark}

\subsection{Complexity results with  prescribed Hessian signature}
The next theorem again builds on Lemma \ref{lemma:nock_deformed} to prove a generalisation of Theorem 2.5 from \cite{auffinger2013random}. In fact, we will need a modified version of Lemma \ref{lemma:nock_deformed} which we now prove.

\begin{lemma}\label{lemma:nock_deformed_cond_k}
    Let $S$ be a rank $2$ $N\times N$ symmetric matrix with non-zero eigenvalues \jstat{$\{s_j\}_{j=1}^2$}, where and $s_j = \mathcal{O}(1)$.  Let $x<-\sqrtsign{2}$ and let $M$ denote an $N\times N$ GOE matrix with respect to whose law expectations are understood to be taken. Then
    \begin{align}
         &\expectGOE \left[|\det(M - xI + S)|\indic [\ind{x} (M+S)\in\{k-1, k, k+1\}]\right]\notag\\ \leq ~~  &\upsilon_U K_N e^{2Nx^2}\left(1 + o(1)\right)e^{-N(k-1)I_1(x;\sqrtsign{2})}\lim_{\epsilon\searrow 0}\iiint_0^{\pi/2} d\theta d\hat{\theta}d\theta'\iint_0^{\infty}dp_1dp_2 \iint_{\Gamma} dr_1dr_2\notag\\
                  &~~~~~~~~ J_1(p_1, p_2, \theta'; \{s_j\}_{j=1}^r, N)J_2(r_1,r_2, p_1, p_2)\cos^22\theta \sin2\theta\sin2\hat{\theta}\notag\\
& ~~~~~~~~ \exp\Bigg\{-N\Bigg(2\psi^{(+)}_L(r_1; x; \epsilon\cos2\theta\cos2\hat{\theta}) +2\psi^{(+)}_U(r_2; x; \epsilon\cos2\theta\cos2\hat{\theta})\notag\\& ~~~~~~~~~~~~~ +\psi^{(-)}_L(p_1; x; \epsilon\cos2\theta')+\psi^{(-)}_U(p_2; x; \epsilon\cos2\theta')\Bigg)\Bigg\}\end{align}
and 
    \begin{align}
         &\expectGOE \left[|\det(M - xI + S)|\indic [\ind{x} (M+S)\in\{k-1, k, k+1\}]\right]\notag\\ \geq ~~  &\upsilon_L K_N e^{2Nx^2}\left(1 + o(1)\right)e^{-N(k+1)I_1(x;\sqrtsign{2})}\lim_{\epsilon\searrow 0}\iiint_0^{\pi/2} d\theta d\hat{\theta}d\theta'\iint_0^{\infty}dp_1dp_2 \iint_{\Gamma} dr_1dr_2\notag\\
                  &~~~~~~~~ J_1(p_1, p_2, \theta'; \{s_j\}_{j=1}^r, N)J_2(r_1,r_2, p_1, p_2)\cos^22\theta \sin2\theta\sin2\hat{\theta}\notag\\
& ~~~~~~~~ \exp\Bigg\{-N\Bigg(2\psi^{(+)}_L(r_1; x; \epsilon\cos2\theta\cos\hat{\theta}) +2\psi^{(+)}_U(r_2; x; \epsilon\cos2\theta\cos\hat{\theta})\notag\\& ~~~~~~~~~~~~~ +\psi^{(-)}_L(p_1; x; \epsilon\cos2\theta')+\psi^{(-)}_U(p_2; x; \epsilon\cos2\theta')\Bigg)\Bigg\}\end{align}
where the functions $J_1, J_2$, the constant $K_N$ and the functions $\psi^{(\pm)}_{U,L}$ are defined as in Lemma \ref{lemma:nock_deformed}, and the $\upsilon_L, \upsilon_U$ are some constants independent of $N$.
\end{lemma}

\jstat{\begin{remark}
A more general version of this lemma holds with $S$ having any fixed rank $r$. In that case, one considers \begin{equation}
    \expectGOE \left[|\det(M - xI + S)|\indic [\ind{x} (M+S)\in\{k-(r-1),\ldots, k, \ldots, k+(r-1)\}]\right]
\end{equation}
and the statement and proof of the result are immediate extensions of what is given here. We omit this generality, since it is not required here.
\end{remark}}

\begin{proof}
This proof is largely the same as that of Lemma \ref{lemma:nock_deformed}. The first difference arises at (\ref{eq:goe_fourier}), where we are required to compute \begin{equation}\label{eq:goe_fourier_deform}
      \expectGOE \left[e^{-i\Tr MA}\indic[\ind{x}(M+S)=k]\right].
\end{equation} As will become apparent towards the end of this proof, we do not know how to maintain the exact equality constraint\footnote{See Remark \ref{rem:generating_func} below.} on index when $S\neq 0$, hence the slightly relaxed results that we are proving, however we will proceed by performing the calculation for $S=0$ and then show that $S$ can be reintroduced one eigendirection at a time. As in the proof of Theorem A.1 in \cite{auffinger2013random}, we split this expectation by fixing a bound, $R$, for the largest eigenvalue, i.e. \begin{align}
    &\expectGOE \left[e^{-i\Tr MA}\indic[\ind{x}(M)=k]\right] \notag\\ =& \expectGOE \left[e^{-i\Tr MA}\indic[\ind{x}(M)=k, \max\{|\lambda_i(M)|\}_{i=1}^N \leq R]\right]\notag\\ +& \expectGOE \left[e^{-i\Tr MA}\indic[\ind{x}(M)=k, \max\{|\lambda_i(M)|\}_{i=1}^N > R]\right] \label{eq:expect_cond_lemma_split}
\end{align}
We will focus initially on the first expectation on the RHS of (\ref{eq:expect_cond_lemma_split}) and deal with the second term later. Let us abbreviate the notation using $$\mathcal{I}_R(M) = \{\max\{|\lambda_i(M)|\}_{i=1}^N \leq R\}.$$ Recall that $A$ has finite rank and note that $A$ is symmetric without loss of generality, since \begin{equation}
 \Tr M\frac{A + A^T}{2} = \frac{1}{2}\left(\Tr MA + \Tr MA^T \right)= \frac{1}{2}\left( \Tr MA + \Tr AM^T\right) = \Tr MA
\end{equation} and hence $A=\text{diag}(a_1, \ldots, a_{r_A}, 0 \ldots, 0) $ without loss of generality.
We begin by factorising the symmetric matrix $M$ in the GOE integral: \begin{align}\label{eq:expect_cond_lem_1}
   & \expect_M\left[ e^{-i\Tr MA}\indic[\ind{x}(M)=k, \mathcal{I}_R(M)]\right]\notag\\ = &\int \frac{d\mu_E(\Lambda)}{Z_N} \indic[-R \leq \lambda_1 \ldots \leq \lambda_k \leq x\leq \lambda_{k+1} \leq \ldots \lambda_N\leq R] \int d \mu_{Haar}(O) e^{-i\sum_{j=1}^{r_A} a_j \vec{o}_j^T\Lambda\vec{o}_j}
\end{align}
where $\mu_E$ is the un-normalised joint density of ordered GOE eigenvalues, $\mu_{Haar}$ is the Haar measure on the orthogonal group $O(N)$, $\vec{o}_j$ are the rows of the orthogonal matrix $O$ and $Z_N$ is normalisation for the ordered GOE eigenvalues given by the Selberg integral: \begin{align}
    Z_N = \frac{1}{N!} (2\sqrtsign{2})^N N^{-N(N+1)/4} \prod_{i=1}^{N} \Gamma\left(1 + \frac{i}{2}\right).
\end{align}
Much like the proof of Theorem A.1 in \cite{auffinger2013random}, we proceed by
splitting the eigenvalues in (\ref{eq:expect_cond_lem_1}) to enforce the  constraint given by the indicator function:
\begin{align}
      &\expect_M\left[ e^{-i\Tr MA}\indic[\ind{x}(M)=k,\mathcal{I}_R(M)]\right]\notag\\ = &\int d\mu_{Haar}(O)\frac{1}{Z_N}\int_{[-R, x]^k} \prod_{i=1}^k\left( d\lambda_i e^{-N\lambda^2_i/2}\right)\Delta\left(\{\lambda_i\}_{i=1}^k\right)\indic\left[\lambda_1 \leq \ldots \leq \lambda_k\right]  \notag \\
      & \int_{(x,R]^{N-k}} \prod_{i=k+1}^N\left( d\lambda_i e^{-N\lambda^2_i/2}\right)\Delta\left(\{\lambda_i\}_{i=k+1}^N\right)\indic\left[\lambda_{k+1} \leq \ldots \leq \lambda_N\right] \notag\\& e^{-i\sum_{j=1}^{r_A} a_j \vec{o}_j^T\Lambda\vec{o}_j}
       \exp\left(\sum_{j=1}^k\sum_{\ell=k+1}^N \log|\lambda_j - \lambda_{\ell}|\right)\notag\\
        = &\int d\mu_{Haar}(O)\int_{[-R, x]^k} \prod_{i=1}^k\left( d\lambda_i e^{-N\lambda^2_i/2}\right)\Delta\left(\{\lambda_i\}_{i=1}^k\right) \frac{Z_{N-k}}{k!Z_N}  \notag \\
      &\frac{1}{Z_{N-k}(N-k)!} \int_{(x,R]^{N-k}} \prod_{i=k+1}^N\left( d\lambda_i e^{-N\lambda^2_i/2}\right)\Delta\left(\{\lambda_i\}_{i=k+1}^N\right)\notag\\& e^{-i\sum_{j=1}^{r_A} a_j \vec{o}_j^T\Lambda\vec{o}_j}
       \exp\left(\sum_{j=1}^k\sum_{\ell=k+1}^N \log|\lambda_j - \lambda_{\ell}|\right)\notag\\
      =&\int_{[-R_N, x_N]^k}\prod_{i=1}^k\left( d\lambda_i e^{-(N-k)\lambda^2_i/2}\right)\Delta\left(\{\lambda_i\}_{i=1}^k\right) \notag\\
    & \int_{(x_N,R_N]^{N-k}} d\bar{\mu}_E(\Lambda_{N-k}) \int d\mu_{Haar}(O)
    e^{-i\sum_{j=1}^{r_A}\sqrtsign{\frac{N-k}{N}} a_j \vec{o}_j^T\Lambda\vec{o}_j}
    \notag\\
     &\exp\left(\sum_{j=1}^k\sum_{\ell=k+1}^N \log|\lambda_j - \lambda_{\ell}|\right)
        \frac{Z_{N-k}}{k!Z_N} \left(\sqrtsign{\frac{N-k}{N}}\right)^{N + N(N+1)/2}\label{eq:expect_cond_lem_2}
\end{align}
where $x_N\defeq \sqrtsign{\frac{N}{N-k}}x$, $R_N\defeq \sqrtsign{\frac{N}{N-k}}R$ and $\bar{\mu}_E$ is the normalised joint density of un-ordered GOE eigenvalues.

We will first need to deal with the Itzykson-Zuber integral in (\ref{eq:expect_cond_lem_2}) before dealing with the eigenvalue integrals. We follow \cite{guionnet2005fourier}, in particular the proof of Theorem 7 therein. We have the well-known result (Fact 8 in \cite{guionnet2005fourier}) that in the sense of distributions\begin{equation}\label{eq:gram_s}
    (\vec{o}_1, \ldots, \vec{o}_{r_A}) \sim \left(\frac{\tilde{\vec{g}}_1}{||\tilde{\vec{g}}_1||},\ldots, \frac{\tilde{\vec{g}}_{r_A}}{||\tilde{\vec{g}}_{r_A}||}\right)
\end{equation}
where the $(\tilde{\vec{g}}_j)_{j=1}^{r_A}$ are constructed via the Gram-Schmidt process from $(\vec{g}_j)_{j=1}^{r_A}\overset{\text{i.i.d.}}{\sim} \mathcal{N}(\bm{0}, 1)$. (\ref{eq:gram_s}) exactly gives \begin{align}
    \int d\mu_{Haar}(O)
    e^{-i\sum_{j=1}^{r_A}\sqrtsign{\frac{N-k}{N}} a_j \vec{o}_j^T\Lambda\vec{o}_j} = \int \prod_{j=1}^{r_A} \frac{d\vec{g}_j}{\sqrtsign{2\pi}^N} e^{-\frac{\vec{g}_j^2}{2}} \exp\left(-i\sqrtsign{\frac{N-k}{N}}\sum_{j=1}^{r_A} a_j
    \frac{ \tilde{\vec{g}}_j^T\Lambda\tilde{\vec{g}}_j}{ ||\tilde{\vec{g}}_j||^2}\right)
\end{align}
and we will now seek to replace the $\tilde{\vec{g}}_j$ with $\vec{g}_j$ via appropriate approximations. Introduce the event \begin{equation}
    B_N(\upsilon) \defeq\left\{| N^{-1}\langle \vec{g}_i, \vec{g}_j\rangle - \delta_{ij}| \leq N^{-\upsilon}, ~~~ 1\leq i, j \leq r_A\right\}
\end{equation}
and then from \cite{guionnet2005fourier} we immediately conclude that under the i.i.d Gaussian law of the $(\vec{g}_j)_{j=1}^{r_A}$ the complementary event has low probability: \begin{equation}\label{eq:gram_s_bnk}
    \mathbb{P}(B_N(\upsilon)^c) =\mathcal{O}( C(\upsilon) e^{-\alpha N^{1-2\upsilon}})
\end{equation}
where $\alpha, C(\upsilon) > 0$ and we take $0<\upsilon < \frac{1}{2}$ to make this statement meaningful. This enables us to write \begin{equation}\label{eq:gram_s_practical}
\int d\mu_{Haar}(O)
    e^{-i\sum_{j=1}^{r_A}\sqrtsign{\frac{N-k}{N}} a_j \vec{o}_j^T\Lambda\vec{o}_j} = \left(1 + \mathcal{O}(e^{-\alpha N^{1-2\upsilon}})\right)\int \prod_{j=1}^{r_A} \frac{d\vec{g}_j}{\sqrtsign{2\pi}^N} e^{-\frac{\vec{g}_j^2}{2}} \exp\left(-i\sqrtsign{\frac{N-k}{N}}\sum_{j=1}^{r_A} a_j
    \frac{ \tilde{\vec{g}}_j^T\Lambda\tilde{\vec{g}}_j}{ ||\tilde{\vec{g}}_j||^2}\right)\indic\{B_N(\upsilon)\}.
\end{equation}

Again, directly from \cite{guionnet2005fourier}, given $B_N(\upsilon)$ we have \begin{align}
   ||\tilde{\vec{g}}_j - \vec{g}_j|| \leq N^{\frac{1}{2} - \frac{\upsilon}{2}}
\end{align}
and therefore \begin{align}
    ||\tilde{\vec{g}}_j||^2 = N\left[ 1 + N^{-1}\left(||\tilde{\vec{g}}_j||^2 - ||\vec{g}_j||^2\right) + \left(N^{-1}||\vec{g}_j||^2 - 1\right)\right] = N( 1 + \mathcal{O}(N^{-\upsilon}) )\label{eq:g_tilde_bound_1}
\end{align}
and \begin{align}\label{eq:g_tilde_bound_2}
    \tilde{\vec{g}}_j^T\Lambda \tilde{\vec{g}_j} &=\vec{g}^T\Lambda \vec{g}  + \sum_{i=1}^N (\tilde{g}_i - g_i)^2\lambda_i + 2\sum_{i=1}^N g_i(\tilde{g}_i - g_i)\lambda_i\notag\\
 \implies ~   \Bigg|\frac{\tilde{\vec{g}}_j^T\Lambda \tilde{\vec{g}}_j}{||\tilde{\vec{g}}_j||^2} - \frac{\vec{g}_j^T\Lambda \vec{g}_j}{||\vec{g}_j||^2}\Bigg|  &\lesssim N^{-\frac{\upsilon}{2}}||\Lambda||_{\infty}.
   \end{align}
   We see therefore that, in approximating the $\{\tilde{\vec{g}}_j\}_j$ by $\{\vec{g}_j\}_j$ in (\ref{eq:gram_s_practical}) we introduce an error term in the exponential that is uniformly small in the integration variables $\{\vec{g}_j\}_j.$
Combining (\ref{eq:gram_s_practical}), (\ref{eq:g_tilde_bound_1}) and (\ref{eq:g_tilde_bound_2}) and noting that $||\Lambda||_{\infty} = R_N \sim R$ under the eigenvalue integral in (\ref{eq:expect_cond_lem_2}) gives \begin{align}
        \int d\mu_{Haar}(O)
    e^{-i\sum_{j=1}^{r_A}\sqrtsign{\frac{N-k}{N}} a_j \vec{o}_j^T\Lambda\vec{o}_j} =  &\left(1 + 
      \mathcal{O}(N^{-\frac{\upsilon}{2}})\right)\int \prod_{j=1}^{r_A} \frac{d\vec{g}_j}{\sqrtsign{2\pi}^N} e^{-\frac{\vec{g}_j^2}{2}} \exp\left(-i\sqrtsign{\frac{N-k}{N}}\sum_{j=1}^{r_A} a_j
    \frac{ \vec{g}_j^T\Lambda\vec{g}_j}{N ( 1 + \mathcal{O}(N^{-\upsilon}))}\right)\notag\\
    =  &\prod_{j=1}^{r_A}\prod_{i=1}^{N}\left(1 + 2iN^{-1}a_j\lambda_i\right)^{-\frac{1}{2}}\left(1 + \mathcal{O}(N^{-\frac{\upsilon}{2}})\right)\notag\\
    =   &\exp\left\{-\frac{N-k}{2}\sum_{j=1}^{r_A} \int d\hat{\mu}_{N-k}(z) \log(1 + 2iN^{-1}a_j z)\right\}\notag\\
    &\exp\left\{-\frac{1}{2} \sum_{j=1}^{r_A}\sum_{i=1}^k\log(1+2iN^{-1}a_j\lambda_i)   \right\}\left(1 + \mathcal{O}(N^{-\frac{\upsilon}{2}})\right) \label{eq:it_zub_split}
\end{align}
where we have defined \begin{equation}
    \hat{\mu}_{N-k} = \frac{1}{N-k}\sum_{i=k+1}^N \delta_{\lambda_i}.
\end{equation}

Following \cite{auffinger2013random}, we now introduce the following function \begin{align}
    \Phi(z, \mu) = -\frac{z^2}{2} + \int d\mu(z') \log|z-z'|
\end{align}
and so and then (\ref{eq:expect_cond_lem_2}) and (\ref{eq:it_zub_split}) can be rewritten as 
\begin{align}
 &\expect_M\left[ e^{-i\Tr MA}\indic[\ind{x}(M)=k,\mathcal{I}_R(M)]\right]\notag\\ = &\int_{[-R_N, x_N]^k}\prod_{i=1}^k d\lambda_i ~ \Delta\left(\{\lambda_j\}_{j=1}^k\right)\exp\left\{-\frac{1}{2} \sum_{j=1}^{r_A}\sum_{i=1}^k\log(1+2iN^{-1}a_j\lambda_i)   \right\}\left(1 + \mathcal{O}(N^{-\frac{\upsilon}{2}})\right) \notag\\
 &\int_{(x_N, R_N]^{N-k}} d\bar{\mu}_E(\Lambda_{N-k}) \notag\exp\left\{-\frac{N-k}{2}\sum_{j=1}^{r_A} \int d\hat{\mu}_{N-k}(z) \log(1 + 2iN^{-1}a_j z)\right\}\\
 &\exp\left((N-k)\sum_{j=1}^k \Phi(\lambda_j, \hat{\mu}_{N-k})\right)    \frac{Z_{N-k}}{k!Z_N} \left(\sqrtsign{\frac{N-k}{N}}\right)^{N + N(N+1)/2}.\label{eq:expect_cond_lem_3}
 \end{align}

We now appeal to the Coulomb gas method \cite{cunden2016shortcut} and in particular the formulation found in \cite{majumdar2011many}. We replace the joint integral of $N-k$ eigenvalues in (\ref{eq:expect_cond_lem_3}) with a functional integral over the continuum eigenvalues density: 
\begin{align}
    &\int_{(x_N, R_N]^{N-k}} d\bar{\mu}_E(\Lambda_{N-K})\exp\left((N-k)\sum_{j=1}^k \Phi(\lambda_j, \hat{\mu}_{N-k})\right)\exp\left\{-\frac{N-k}{2}\sum_{j=1}^{r_A} \int d\hat{\mu}_{N-k}(z) \log(1 + 2iN^{-1}a_j z)\right\} \notag\\
    = & \int \mathcal{D}[\mu]e ^{-N^2 \mathcal{S}_x[\mu]} \exp\left((N-k)\sum_{j=1}^k \Phi(\lambda_j, \mu)\right)\exp\left\{-\frac{N-k}{2}\sum_{j=1}^{r_A} \int d\mu(z) \log(1 + 2iN^{-1}a_j z)\right\}\label{eq:expect_cond_lem_7}
\end{align}
where the action is defined as \begin{align}
    \mathcal{S}_x[\mu] = &\frac{1}{2}\int dz \mu(z) z^2 - \iint_{z\neq z} dzdz'\mu(z)\mu(z') \log|z-z'| \notag \\
    + &A_1\left(\int dz\theta(R_N - z)\mu(z) - 1\right) + A_2\left(\int dz \mu(z)\theta(z-x) - 1\right) - \Omega
\end{align}
where $\theta$ is the Heaviside step function, $\Omega$ is the constant resulting from the normalisation of the eigenvalue joint density and $A_1,A_2$ are Lagrange multipliers. 

Owing to the $N^2$ rate in (\ref{eq:expect_cond_lem_7}), the integral concentrates around the minimiser of the action. Since $x< -\sqrtsign{2}$ and we have chosen $R>|x|$, it is clear following \cite{majumdar2011many} that the semi-circle law $\mu_{SC}(z) = \pi^{-1}\sqrtsign{2 - z^2}$ minimises this action and further that $\mathcal{S}_x[\mu_{SC}] = 0$, so we have \begin{align}
   & \int \mathcal{D}[\mu]e ^{-N^2 \mathcal{S}_x[\mu]} \exp\left((N-k)\sum_{j=1}^k \Phi(\lambda_j, \mu)\right)\exp\left\{-\frac{N-k}{2}\sum_{j=1}^{r_A} \int d\mu(z) \log(1 + 2iN^{-1}a_j z)\right\}\notag\\
   = &\int_{B_{\delta}(\mu_{SC})} \mathcal{D}[\mu]e ^{-N^2 \mathcal{S}_x[\mu]}\exp\left((N-k)\sum_{j=1}^k \Phi(\lambda_j, \mu)\right)\exp\left\{-\frac{N-k}{2}\sum_{j=1}^{r_A} \int d\mu(z) \log(1 + 2iN^{-1}a_j z)\right\}
 \notag \\ & ~~~~ + e^{-N^2 c_{\delta}}\mathcal{O}(1) \label{eq:expect_cond_lem_8}
\end{align}
where $\delta=\mathcal{O}(N^{-1})$ and $c_{\delta}>0$ is some constant. Performing the usual Laplace method expansion of the action in (\ref{eq:expect_cond_lem_8}) and re-scaling the first non-vanishing derivative to be $\mathcal{O}(1)$, it is clear that the action only contributes a real factor of $\mathcal{O}(1)$ that is independent of the dummy integration variables $\vec{x}_1, \vec{x}_2, \zeta_1, \zeta_1^{\dagger}, \zeta_2, \zeta_2^{\dagger}$ and the other eigenvalues $\lambda_1,\ldots, \lambda_k$ and can therefore be safely summarised as $\mathcal{O}(1)$. Whence   \begin{align}
   & \int \mathcal{D}[\mu]e ^{-N^2 \mathcal{S}_x[\mu]} \exp\left((N-k)\sum_{j=1}^k \Phi(\lambda_j, \mu)\right)\exp\left\{-\frac{N-k}{2}\sum_{j=1}^{r_A} \int d\mu(z) \log(1 + 2iN^{-1}a_j z)\right\}\notag\\
   = &\mathcal{O}(1)\exp\left((N-k)\sum_{j=1}^k \Phi(\lambda_j, \mu_{SC})\right)\exp\left\{-\frac{N-k}{2}\sum_{j=1}^{r_A} \int d\mu_{SC}(z) \log(1 + 2iN^{-1}a_j z)\right\}
 \notag \\ & ~~~~ + e^{-N^2 c_{\delta}}\mathcal{O}(1). \label{eq:expect_cond_lem_9}\end{align}
Now elementary calculations give, noting that the integrand is uniformly convergent in $N$ owing to the compact support of $\mu_{SC}, $\begin{align}
     \int d\mu_{SC}(z) \log(1 + 2iN^{-1}a_j z) &= -\frac{2ia_j}{N}\int d\mu_{SC}(z) z + \frac{2a^2_j}{N^2}\int d\mu_{SC}(z) z^2 + \mathcal{O}(a_j^3N^{-3})\notag\\
     &= \frac{a^2_j}{N^2} ( 1 + \mathcal{O}(a_jN^{-1}))\notag\\
     \implies  \frac{N-k}{2}\sum_{j=1}^{r_A} \int d\mu_{SC}(z) \log(1 + 2iN^{-1}a_j z) &= \frac{\Tr A^2}{2N}  ( 1 + ||A||_{\infty}\mathcal{O}(N^{-1}))\label{eq:expect_cond_lem_10}
\end{align}
where we have implicitly assumed that the spectral radius $||A||_{\infty} \ll N$. This constraint can be introduced by restricting the domains of integration for $\vec{x}_1$ and $\vec{x}_2$ in the anaologue of  (\ref{eq:nock2}) from all of $\mathbb{R}^N$ to balls of radius $o(\sqrtsign{N})$. It is a standard result for Gaussian integrals that this can be achieved at the cost of an exponentially smaller term.
Summarising (\ref{eq:expect_cond_lem_3}), (\ref{eq:expect_cond_lem_7}), (\ref{eq:expect_cond_lem_9}) and (\ref{eq:expect_cond_lem_10}):\begin{align}
 &\expect_M\left[ e^{-i\Tr MA}\indic[\ind{x}(M)=k,\mathcal{I}_R(M)]\right]\notag\\ = &\int_{[-R_N, x_N]^k}\prod_{i=1}^k d\lambda_i ~ \Delta\left(\{\lambda_j\}_{j=1}^k\right)\exp\left\{-\frac{1}{2} \sum_{j=1}^{r_A}\sum_{i=1}^k\log(1+2iN^{-1}a_j\lambda_i)   \right\}\exp\left((N-k)\sum_{j=1}^k \Phi(\lambda_j, \mu_{SC})\right) \notag\\
 &e^{-\frac{\Tr A^2}{2N}}\left(\mathcal{O}(1) + \mathcal{O}(N^{-\frac{\upsilon}{2}}) + \mathcal{O}(N^{-1})||A||_{\infty}\right)   \frac{Z_{N-k}}{k!Z_N} \left(\sqrtsign{\frac{N-k}{N}}\right)^{N + N(N+1)/2}\notag\\
 = &\int_{[-R_N, x_N]^k}\prod_{i=1}^k d\lambda_i ~ \Delta\left(\{\lambda_j\}_{j=1}^k\right)\exp\left((N-k)\sum_{j=1}^k \Phi(\lambda_j, \mu_{SC})\right) \notag\\
 &e^{-\frac{\Tr A^2}{2N}}\left(\mathcal{O}(1) + \mathcal{O}(N^{-\frac{\upsilon}{2}}) + \mathcal{O}(N^{-1})||A||_{\infty}\right)   \frac{Z_{N-k}}{k!Z_N} \left(\sqrtsign{\frac{N-k}{N}}\right)^{N + N(N+1)/2}\notag\\
 = &\int_{[-R_N, x_N]^k}\prod_{i=1}^k d\lambda_i ~ \Delta\left(\{\lambda_j\}_{j=1}^k\right)\exp\left((N-k)\sum_{j=1}^k \Phi(\lambda_j, \mu_{SC})\right) \notag\\
 &e^{-\frac{\Tr A^2}{2N}}\mathcal{O}(1)   \frac{Z_{N-k}}{k!Z_N} \left(\sqrtsign{\frac{N-k}{N}}\right)^{N + N(N+1)/2}\label{eq:expect_cond_lem_11}
 \end{align}
where in the second equality we have Taylor expanded the remaining logarithm and summarised the result with another factor of $(1 + \mathcal{O}(N^{-1})||A||_{\infty})$.

We now wish to follow the proof of Theorem A.1 in \cite{auffinger2013random} and use $\Delta(\{\lambda_j\}_{j=1}^{k}) \leq (2R_N)^k \leq (3R)^k$ for $\lambda_j \in [-R_N, R_N]$ with bound (\ref{eq:expect_cond_lem_11}), however the expectation on the left hand side of (\ref{eq:expect_cond_lem_11}) is not necessarily real. We do however know that the $\mathcal{O}(1)$ term in (\ref{eq:expect_cond_lem_11}) is real to leading order and so we can write  \begin{equation}\label{eq:expect_cond_real_imag}
    \expect_M\left[ e^{-i\Tr MA}\indic[\ind{x}(M)=k,\mathcal{I}_R(M)]\right] = \Re \expect_M\left[ e^{-i\Tr MA}\indic[\ind{x}(M)=k,\mathcal{I}_R(M)]\right] \left( 1 + io(1)\right)
\end{equation}

and thence focus on bounding the real part of the expectation to obtain \begin{align}
     & \Re \expect_M\left[ e^{-i\Tr MA}\indic[\ind{x}(M)=k, \mathcal{I}_R(M)]\right] \notag\\
     \leq &K(3R)^k \frac{Z_{N-k}}{k!Z_N} \left(\sqrtsign{\frac{N-k}{N}}\right)^{N + N(N+1)/2}
    e^{-\frac{\Tr A^2}{2N}} \left(\int_{-R_N}^{x_N} dz e^{(N-k)\Phi(z, \mu)}\right)^k\label{eq:expect_cond_lem_12}
\end{align}
where we have exchanged $\mathcal{O}(1)$ terms for some appropriate constant $K$. Continuing to bound (\ref{eq:expect_cond_lem_12}): \begin{align}
 &\Re \expect_M\left[ e^{-i\Tr MA}\indic[\ind{x}(M)=k, \mathcal{I}_R(M)]\right] \notag\\
     \leq &K(3R)^{2k}\frac{Z_{N-k}}{k!Z_N} \left(\sqrtsign{\frac{N-k}{N}}\right)^{N + N(N+1)/2}
    e^{-\frac{\Tr A^2}{2N}} \exp\left(k(N-k)\sup\limits_{\substack{z\in[-2R, x] \\ \nu\in B_{\delta}(\mu_{SC})}}\Phi(z, \nu)\right)\notag\\
      \leq &K(3R)^{2k} \frac{Z_{N-k}}{k!Z_N} \left(\sqrtsign{\frac{N-k}{N}}\right)^{N + N(N+1)/2}
    e^{-\frac{\Tr A^2}{2N}} e^{-k(N-k)(1/2 + I_1(x; \sqrtsign{2})} \label{eq:expect_cond_lem_upper_bound}
\end{align}
where we have used the same result as used around (A.18) in \cite{auffinger2013random} to take the supremum.

Recalling (\ref{eq:expect_cond_lemma_split}), we can now use (\ref{eq:expect_cond_lem_upper_bound}) and the GOE large deviations principle \cite{arous2001aging} as in \cite{auffinger2013random} to obtain \begin{align}
    \Re\expect_M\left[ e^{-i\Tr MA}\indic[\ind{x}(M)=k]\right] 
      \leq &K''(3R)^k  \frac{Z_{N-k}}{k!Z_N} \left(\sqrtsign{\frac{N-k}{N}}\right)^{N + N(N+1)/2} e^{-k(N-k)(1/2 + I_1(x; \sqrtsign{2}))}
     e^{-\frac{1}{2N}\Tr A^2} + e^{-NR^2}\label{eq:expect_cond_lem_upper_bound_full}
\end{align}
We now seek to obtain a complementary lower bound and again follow \cite{auffinger2013random} in choosing some $y$ and $R'$ such that $y < x < R' < -\sqrtsign{2}$. We then, following a similar procedure as above, find \begin{align}
       \Re\expect_M\left[ e^{-i\Tr MA}\indic[\ind{x}(M)=k]\right] \geq
 & \tilde{K} \frac{Z_{N-k}}{k!Z_N} \left(\sqrtsign{\frac{N-k}{N}}\right)^{N + N(N+1)/2}
     e^{-\frac{1}{2N}\Tr A^2} \exp\left(k(N-k)\sup\limits_{\substack{z\in[y, x] \\ \nu\in B_{\delta}(\mu_{SC})}}\Phi(z, \nu)\right)\end{align}
and taking $y\nearrow x$ we obtain the complement to (\ref{eq:expect_cond_lem_upper_bound_full}): \begin{align}
      \Re\expect_M\left[ e^{-i\Tr MA}\indic[\ind{x}(M)=k]\right] \geq
 & \tilde{K}\frac{Z_{N-k}}{k!Z_N} \left(\sqrtsign{\frac{N-k}{N}}\right)^{N + N(N+1)/2}
    e^{-k(N-k)(1/2 + I_1(x; \sqrtsign{2}))} e^{-\frac{1}{2N}\Tr A^2}.\label{eq:expect_cond_lem_lower_bound}
\end{align}

Next we need the asymptotic beahviour of the Selberg term in  (\ref{eq:expect_cond_lem_upper_bound_full}) and (\ref{eq:expect_cond_lem_lower_bound})

\begin{align}
 T_{N,k} \defeq \frac{Z_{N-k}}{k!Z_N} \left(\sqrtsign{\frac{N-k}{N}}\right)^{N + N(N+1)/2} =   \underbrace{ \frac{Z_{N-k}(N-k)!}{Z_N N!}\left(\frac{N-k}{N}\right)^{ \frac{(N-k)(N-k+1)}{4} }}_{T_{N,k}'} \frac{N!}{(N-k)!k!}\left(\frac{N-k}{N}\right)^{\frac{N}{2} + \frac{N(N+1)- (N-k)(N-k+1)}{4} }\label{eq:k_theorem_selberg}.
\end{align}
The term $T_{N,k}'$ appears in \cite{auffinger2013random} (defined in A.13) and it is shown there that \begin{equation}
 \lim_{N\rightarrow\infty}   N^{-1}\log T_{N,k}' = \frac{k}{2}.
\end{equation}
Clearly \begin{equation}
   \lim_{N\rightarrow\infty}N^{-1} \log \frac{N!}{(N-k)!k!} = 0
\end{equation}
and it is simple to show that \begin{align}
    \lim_{N\rightarrow\infty} \left(\frac{N-k}{N}\right)^{\frac{N}{2} + \frac{N(N+1)- (N-k)(N-k+1)}{4} } = e^{-\frac{k(k+1)}{2}}
\end{align}
and so we have overall \begin{equation}
    \lim_{N\rightarrow\infty}N^{-1}\log T_{N,k} = \frac{k}{2}.
\end{equation}

So absorbing any $\mathcal{O}(1)$ terms into constants $K_L$ and $K_U$ we have \begin{align}
   K_L e^{-kN(1 + o(1)) I_1(x; \sqrtsign{2})}
     e^{-\frac{1}{2N}\Tr A^2}  \leq  \Re\expect_M\left[ e^{-i\Tr MA}\indic[\ind{x}(M)=k]\right] 
      \leq K_U e^{-kN(1+o(1)) I_1(x; \sqrtsign{2})}
     e^{-\frac{1}{2N}\Tr A^2} \label{eq:expect_cond_lemma_bounds}
\end{align}

Set $S=s_1\vec{e}_1\vec{e}_1^T + s_2\vec{e}_2\vec{e}_2^T$ and $S_1 = s_1\vec{e}_1\vec{e}_1^T$. Suppose $s_1>0$ and $s_2>0$.  By the interlacing property of eigenvalues, we have \begin{equation}\label{eq:interlacing}
    \lambda^{(M)}_1 \leq  \lambda^{(M+S_1)}_1 \leq  \lambda^{(M)}_2 \leq \ldots \leq  \lambda^{(M)}_k \leq  \lambda^{(M+S_1)}_{k} \leq  \lambda^{(M)}_{k+1} \leq \lambda^{(M+S_1)}_{k+1} \leq \ldots \leq \lambda^{(M)}_{N} \leq \lambda^{(M+S_1)}_{N}
\end{equation}

Therefore we have \begin{align}\label{eq:interlace_set}
  \begin{cases}  \{ \ind{x}(M) =k \} \subset \{\ind{x}(M+S_1) \in\{k-1, k\}\} \subset \{\ind{x}(M) \in\{k-1,k, k+1\} \}=\bigsqcup\limits_{j=-1}^1\{\ind{x}(M) = k + j \} ~ &\text{for } k>0,\\
    \{ \ind{x}(M) =k \} \subset \{\ind{x}(M+S_1) =k \} \subset \{\ind{x}(M) \in\{k, k+1\} \}=\bigsqcup\limits_{j=0}^1\{\ind{x}(M) = k + j \} &\text{for } k=0,\\
  \end{cases}
\end{align}
and so (\ref{eq:expect_cond_lemma_bounds}) gives \begin{equation}
\begin{aligned}
      K_L e^{-kN (1+o(1))I_1(x; \sqrtsign{2})}
     e^{-\frac{1}{2N}\Tr A^2}\leq&\Re\expect_M\left[ e^{-i\Tr MA}\indic[\ind{x}(M+S_1)\in\{k-1,k\}]\right] \leq  3K_U e^{-(k-1)N (1+o(1))I_1(x; \sqrtsign{2})}
     e^{-\frac{1}{2N}\Tr A^2},\\
     e^{-\frac{1}{2N}\Tr A^2}\leq&\Re\expect_M\left[ e^{-i\Tr MA}\indic[\ind{x}(M+S_1)=0]\right] \leq  2K_U
     e^{-\frac{1}{2N}\Tr A^2}.
\end{aligned}    
\end{equation}
We can then extend to $S$ likewise by observing that interlacing gives \begin{align}
       \{ \ind{x}(M+S_1) \in\{k, k+1\} \} &\subset \{\ind{x}(M+S) \in\{k-1, k, k+1\}\} \subset \{\ind{x}(M+S_1) \in\{k-1,k, k+1, k+2\} \}
\end{align} and iterating using (\ref{eq:interlace_set}) yields 
\begin{align}\label{eq:M_S_interlace}
\begin{cases}
       \{ \ind{x}(M) = k+1 \} \subset \{\ind{x}(M+S) \in\{k-1, k, k+1\}\} \subset \bigsqcup\limits_{j=-1}^3\{\ind{x}(M) = k + j \}, ~ &\text{for } k>0\\
      \{ \ind{x}(M) = k+1 \} \subset \{\ind{x}(M+S) \in\{ k, k+1\}\} \subset \bigsqcup\limits_{j=0}^3\{\ind{x}(M) = k + j \}, &\text{for } k=0
\end{cases}
\end{align}
and (\ref{eq:expect_cond_lemma_bounds}) then gives
\begin{equation}\label{eq:expect_cond_lemma_deform_bounds}
\begin{aligned}
    K_L e^{-(k+1)N(1+o(1)) I_1(x; \sqrtsign{2})}
     e^{-\frac{1}{2N}\Tr A^2}&\leq\Re\expect_M\left[ e^{-i\Tr MA}\indic[\ind{x}(M+S)\in\{k-1,k, k+1\}]\right] \\
     &\leq  5K_U e^{-(k-1)N(1+o(1)) I_1(x; \sqrtsign{2})}
     e^{-\frac{1}{2N}\Tr A^2}\\
      K_L e^{-N(1+o(1)) I_1(x; \sqrtsign{2})}
     e^{-\frac{1}{2N}\Tr A^2}&\leq\Re\expect_M\left[ e^{-i\Tr MA}\indic[\ind{x}(M+S)\in\{0, 1\}]\right] 
     \leq  4K_U
     e^{-\frac{1}{2N}\Tr A^2}.
\end{aligned}     
\end{equation}
If instead the \jstat{signs of $s_1, s_2$ are different, then the interlacing will be in the reverse orders}, but the conclusion of (\ref{eq:expect_cond_lemma_deform_bounds}) will be unchanged.
Finally using (\ref{eq:expect_cond_real_imag}) in the analogue of (\ref{eq:nock_initial})

\begin{align}
 &\expect_M[|\det(M - xI + S)|\indic[\ind{x}(M+S)\in\{k-1,k, k+1\}] \notag \\ = &\Re\expect_M[|\det(M - xI + S)|\indic[\ind{x}(M+S)\in\{k-1,k, k+1\}] \notag \\ 
 =&\Re \Bigg\{  K^{(1)}_N\lim_{\epsilon\searrow 0}\int d\vec{x}_1 d\vec{x}_2 d\zeta_1 d\zeta_1^{\dagger} d\zeta_2 d\zeta_2^{\dagger} \exp\left\{-i\vec{x}_1^T(M-(x + i\epsilon)I+S)\vec{x}_1 - i\vec{x}_2^T(M-(x-i\epsilon)I + S)\vec{x}_2\right\}\notag \\
 & ~~~~~~\exp\left\{ i \zeta_1^{\dagger}(M-(x+i\epsilon) I+S)\zeta_1
      + i \zeta_2^{\dagger}(M-(x - i\epsilon)I+S)\zeta_2\right\} \\
  &~~~~~~\expect_M\left[e^{-i\Tr MA}\indic[\ind{x}(M+S)\in\{k-1,k, k+1\}]\right]\Bigg\}\notag \\
 =&\Re \Bigg\{  K^{(1)}_N\lim_{\epsilon\searrow 0}\int d\vec{x}_1 d\vec{x}_2 d\zeta_1 d\zeta_1^{\dagger} d\zeta_2 d\zeta_2^{\dagger} \exp\left\{-i\vec{x}_1^T(M-(x + i\epsilon)I+S)\vec{x}_1 - i\vec{x}_2^T(M-(x-i\epsilon)I + S)\vec{x}_2\right\}\notag \\
 & ~~~~~~\exp\left\{ i \zeta_1^{\dagger}(M-(x+i\epsilon) I+S)\zeta_1
      + i \zeta_2^{\dagger}(M-(x - i\epsilon)I+S)\zeta_2\right\} \\
  &~~~~~~\Re\expect_M\left[e^{-i\Tr MA}\indic[\ind{x}(M+S)\in\{k-1,k, k+1\}]\right](1+io(1))\Bigg\}\notag \\
  =&\Re \Bigg\{  K^{(1)}_N\lim_{\epsilon\searrow 0}\int d\vec{x}_1 d\vec{x}_2 d\zeta_1 d\zeta_1^{\dagger} d\zeta_2 d\zeta_2^{\dagger} \exp\left\{-i\vec{x}_1^T(M-(x + i\epsilon)I+S)\vec{x}_1 - i\vec{x}_2^T(M-(x-i\epsilon)I + S)\vec{x}_2\right\}\notag \\
 & ~~~~~~\exp\left\{ i \zeta_1^{\dagger}(M-(x+i\epsilon) I+S)\zeta_1
      + i \zeta_2^{\dagger}(M-(x - i\epsilon)I+S)\zeta_2\right\} \\
  &~~~~~~\Re\expect_M\left[e^{-i\Tr MA}\indic[\ind{x}(M+S)\in\{k-1,k, k+1\}]\right]\Bigg\}(1+io(1))
\label{eq:expect_cond_lemma_real_bound_expr}
\end{align}

From this point on, the proof proceeds, \emph{mutatis mutandis}, as that for Lemma \ref{lemma:nock_deformed} but applied to the upper and lower bounds on (\ref{eq:expect_cond_lemma_real_bound_expr}) obtained from (\ref{eq:expect_cond_lemma_deform_bounds}). The final range of integration for $p_1$ and $p_2$ will be some intervals $(0, o(1))$ owing to the change of variables used around (\ref{eq:nock7_prime}), but this does not affect the ensuing asymptotics in which the $p_1,p_2$ integration contours are deformed through the saddle point at $z^{(-)}_{U,L}$.
\end{proof}

\begin{remark}\label{rem:generating_func}
We note that if an appropriate generating function for $\indic\left[\ind{x}(M + S)=k\right]$ could be found, that would allow for a straightforward taking of the expectation in (\ref{eq:goe_fourier_deform}), then the calculations of Lemma \ref{lemma:nock_deformed} could be modified to include this extra term and then the desired expectation $\expectGOE \left[|\det(M - xI + S)|\indic [\ind{x} (M+S)=k]\right]$ could be read-off in comparison with the result of Lemma \ref{lemma:nock_deformed}.
\end{remark}

We have established all we need to prove Theorem \ref{thm:auff2.5}.
\auffdepk*
\begin{proof}

First consider $u < -E_{\infty}$. The proof proceeds just as that of Theorem \ref{thm:auff2.8} but applying Lemma \ref{lemma:nock_deformed_cond_k} instead of Lemma \ref{lemma:nock_deformed} and working identically on the upper and lower bounds from Lemma \ref{lemma:nock_deformed_cond_k}.

Now consider $u> -E_{\infty}$. By the interlacing property as used around (\ref{eq:interlacing}), $\ind{x}(M)$ and $\ind{x}(M+S)$ differ by no more than 2. Hence \begin{equation}
    \ind{x}(M+S) \in \mathcal{K} \implies \ind{x}(M) = \mathcal{O}(1)
\end{equation} but for $0 > x > -\sqrtsign{2}$, and $M\sim GOE_N$, the large deviations principle for the GOE \cite{arous1997large} gives \begin{equation}
    \mathbb{P}(\ind{x}(M) = \mathcal{O}(1)) \leq e^{-cN^2}
\end{equation} 
for some constant $c$, hence the $x$ integral analogous to (\ref{eq:thm28_mid}) is exponentially suppressed with quadratic speed in $N$ for $x>-\sqrtsign{2}$. But we have already seen that the integral is only suppressed with linear speed in $N$ for $x < -\sqrtsign{2}$, and further that $\Theta_{H,k}(u)$ is increasing on $(-\infty, -E_{\infty})$ and so, by the Laplace principle, the leading order contribution is from around $x=-\sqrtsign{2}$ and so \begin{equation}
    \lim_{N\rightarrow\infty} \frac{1}{N}\log\expect C_{N,\mathcal{K}}^{h}(\sqrtsign{N}u) = \lim_{N\rightarrow\infty} \frac{1}{N}\log\expect C_{N,\mathcal{K}}^{h}(-\sqrtsign{N}E_{\infty})
\end{equation} for $u > - E_{\infty}$, which completes the proof.

\end{proof}

\begin{remark}
We are clearly unable to provide an exact leading term for $C^h_{N, \mathcal{K}}(\sqrtsign{N}u)$ for any value of $u$ as we did for $C^h_N(\sqrtsign{N}u)$ for $u< -E_{\infty}$ in Theorem \ref{thm:exact_term} because  the presence of $S$ in $\ind{x}(M+S)$ has forced us in Lemma \ref{lemma:nock_deformed_cond_k} to resort to upper and lower bounds on the leading order term. We note that in \cite{auffinger2013random} the authors are also not able to obtain the exact leading term in this case by their rather different methods. Recalling Remark \ref{rem:generating_func}, we conjecture that this term could be obtained by variants of our methods if only a suitable (perhaps approximate) generating function for $\indic[\ind{x}(M+S) = k] $ could be discovered.
\end{remark}

\section{Conclusions and future work}
The interpretation of the results we have presented here is largely the same as that first given in \cite{choromanska2015loss}. \jstat{Under the chosen modeling assumptions}, the local optima of the the neural network loss surface are arranged so that, above a critical value $-\sqrtsign{N}E_{\infty}$, it is overwhelmingly likely that gradient descent will encounter high-index optima and so `escape' and descend to lower loss.  Below $-\sqrtsign{N}E_{\infty}$, the low-index optima are arranged in a `banded' structure, however, due to the imprecision of Theorem \ref{thm:auff2.5}, the bands are slightly blurred when compared with \cite{choromanska2015loss}. We display the differences in Table \ref{tab:banded}.
\begin{table}[]
    \centering
    \begin{tabular}{c|c|c}
       band  & possible indices \cite{choromanska2015loss} & \textbf{possible indices}  \\
       \hline
        $(-\sqrtsign{N}E_{0},-\sqrtsign{N}E_{1})$  &0&0,1,2\\
        $(-\sqrtsign{N}E_{1},-\sqrtsign{N}E_{2})$  &0,1 &0,1,2,3\\
        $(-\sqrtsign{N}E_{2},-\sqrtsign{N}E_{3})$ & 0,1,2 & 0,1,2,3,4\\
        $(-\sqrtsign{N}E_{3},-\sqrtsign{N}E_{4})$ & 0,1,2,3 & 0,1,2,3,4,5\\
    \end{tabular}
        \caption{Illustration of the banded low-index local optima structure obtained here for neural networks with general activation functions and compared to the analogous results in \cite{choromanska2015loss}.}
    \label{tab:banded}
\end{table}


\jstat{Our results have plugged a gap in the analysis of \cite{choromanska2015loss} by demonstrating that the specific \texttt{ReLU} activation function required by the technicalities of their derivation is not, in fact, a requirement of the results themselves, which we have shown to hold for any reasonable choice of activation function. At the same time, experimental results imply that a sufficiently precise model for deep neural network loss surfaces should display some non-trivial dependence on the choice of activation function, but we have shown that no dependence at all is seen at the relevant level of logarithmic asymptotic complexity, but is visible in the sharp leading order complexity. In defense of \cite{choromanska2015loss}, we have reduced the scope for their results to be some spurious apparition of an intersection of several unrealistic simplifications. However, with the same result, we have demonstrated an important aspect of neural network architectural design to which the multi-spin glass correspondence is entirely insensitive, so limiting the precision of any statements about real neural networks that can be made using this analysis.}

In the pursuit of our aims, we have been forced to approximately reproduce the work of \cite{auffinger2013random} by means of the supersymmetric method of Random Matrix Theory, which we believe is quite novel and have also demonstrated how various steps in these supersymmetric calculations can be adapated to the setting of a GOE matrix deformed by some low-rank fixed matrix including utilising Gaussian approximations to orthogonal matrices in ways we have not previously seen in the literature. \jstat{We believe some of our intermediate results and methods may be of use in other contexts in Random Matrix Theory.}

As highlighted in the main text, there are a few areas for future work that stem immediately from our calculations. We list them here along with other possibilities.
\begin{enumerate}
    \item Constructing an appropriate indicator function (or approximate indicator function) for the index of a matrix so that Theorem \ref{thm:auff2.5} can be precised and to obtain exact leading order terms for $C^h_{N,k}$ that could not be obtained in \cite{auffinger2013random} (see Remark \ref{rem:exact_k_ind}).
    \item The `path-independence' assumption (Section \ref{subsec:modelling_assumptions}, assumption \ref{item: assumption_bernoulli}) is the weakest link in this work (and that of \cite{choromanska2015loss}) and we have shed further light on its validity through experimentation (Section \ref{subsec:discussion_assumptions}). The supersymmetric calculations used here have shown themselves to be powerful and quite adaptable. We therefore suggest that it may be possible to somehow encapsulate the failure of assumption \ref{item: assumption_bernoulli} as a first-order correlation term and repeat the presented analysis in an expansion when this term is small. 
    \item Further, this work and others mentioned in the introduction have shown that studying spin glass like objects in this context is a fruitful area of research and so we would like to study more exotic glassy objects inspired by different neural network architectures and applications and hope to be able to adapt the calculations presented here to such new scenarios. 
\end{enumerate}

\section{Acknowledgements}
FM is grateful for support from the University Research Fellowship
of the University of Bristol. JPK is pleased to acknowledge support from a Royal Society
Wolfson Research Merit Award and ERC Advanced Grant 740900 (LogCorRM). \jstat{We are grateful to two anonymous referees for their most helpful comments and suggestions.}
\printbibliography

\appendix
\section{Specific expression for the low-rank perturbation matrix}\label{ap:S_specific}
The the rank-2  $N-1\times N-1$ matrix $S$ arises throughout the course of Sections \ref{sec:nns_random_funcs} and \ref{sec:statement_results} and Lemma \ref{lemma:conditional_dist}. The specific value of $S$ is not required at any point during our calculations and, even though its eigenvalues appear in the result of Theorem \ref{thm:exact_term}, it is not apparent that explicit expressions for its eigenvalues would affect the practical implications of the theorem. These considerations notwithstanding, in this supplementary section we collate all the expressions involved in the development of $S$ from the modeling of the activation function in Section \ref{sec:nns_random_funcs} through to Lemma \ref{lemma:conditional_dist}. Beginning at the final expression for $S$ in Lemma \ref{lemma:conditional_dist} 
\begin{align}
            S_{ij} = \frac{1}{ \sqrtsign{2(N-1)H(H-1)} }\left(\xi_3 + \xi_2(\delta_{i1} + \delta_{j1}) + \xi_1\delta_{i1}\delta_{j1}\right),\end{align}
    where, recalling the re-scaling (\ref{eq:rho_N_redef}), 
 \begin{align}
          \xi_0 &= \sum_{\ell=1}^HN^{-\ell/2}\rho_{\ell}^{(N)}\\
                    \xi_1 &=  \sum_{\ell=1}^{H-2}N^{-\ell/2}\rho_{\ell}^{(N)} \left[(H-\ell)(H-\ell -1) +1 \right]\\
          \xi_2 & =\sum_{\ell=1}^{H-2}N^{-\ell/2}\rho_{\ell}^{(N)}(H-\ell - 2) \\
          \xi_3 &= \sum_{\ell=1}^{H-2}N^{-\ell/2}\rho_{\ell}^{(N)}
      \end{align}
The $\rho_{\ell}$ were defined originally in (\ref{eq:rho_def}) and re-scaled around (\ref{eq:g_def}) so that  \begin{align}
    \rho_{\ell} = \frac{ \expect A_{i,j}^{(\ell)}}{\expect A_{i,j}}
\end{align}
where $A_{i,j}$ are discrete random variables taking values in \begin{equation}
   \mathcal{A} \defeq \left\{\prod_{i=1}^H \alpha_{j_i}\ ~:~ j_1,\ldots, j_H \in \{1,\ldots, L\}\right\}
\end{equation} and $ A^{(\ell)}_{i,j}$ take values in \begin{equation}
    \mathcal{A}^{(\ell)} \defeq\left\{\beta_k\prod_{r=1}^{H-\ell} \alpha_{j_r} ~:~ j_1,\ldots, j_{H-\ell}, k \in \{1,\ldots, L\}\right\}\end{equation}
but we have not prescribed the mass function of the $A_{i,j}$ or $A_{i,j}^{(\ell)}.$
Lastly recall that the $\alpha_j, \beta_j$ are respectively the slopes and intercepts of the piece-wise linear function chosen to approximate the activation function $f$.

\section{Low rank perturbation of a matrix identity}\label{ap:low_rank_fyod}
\jstat{In this section we establish a modified version of Theorem I from \cite{fyodorov2002characteristic} required in the proof in Lemma \ref{lemma:nock_deformed}. In that Lemma, we are faced with an integral of the form \begin{equation}\label{eq:ap_fyod_deformed}
    \mathcal{I}_N(F; S) = \iint_{\mathbb{R}^N} d\vec{x}_1 d\vec{x}_2 F(Q_B) e^{-iN\Tr SB}
\end{equation}
where the $N\times N$ matrix $B$ is defined as $B=\vec{x}_1\vec{x}_1^T + \vec{x}_2\vec{x}_2^T$, the $2\times 2$ matrix $Q_B$ is given by \begin{equation}
    Q_B = \left(\begin{array}{cc} \vec{x}_1^T\vec{x}_1 & \vec{x}_1^T\vec{x}_2 \\ \vec{x}_2^T\vec{x}_1 & \vec{x}_2^T\vec{x}_2\end{array}\right),
\end{equation}
 $F$ is some suitably nice function and $S$ is some real symmetric matrix of rank $r=\mathcal{O}(1)$ as $N\rightarrow \infty$ and with non-zero eigenvalues $\{N^{-1/2}s_i\}_{i=1}^r$ for $s_i=\mathcal{O}(1)$.
It is sufficient to be able to evaluate a leading order term of $\mathcal{I}_N$ in an expansion for large $N$. \cite{fyodorov2002characteristic} proves the following related result:
\begin{lemma}[\cite{fyodorov2002characteristic}\label{lemma:fyod_lemma} Theorem I]
Given $m$ vectors in $\mathbb{R}^N$ $\vec{x}_1, \ldots, \vec{x}_m$, denote by $Q(\vec{x}_1, \ldots, \vec{x}_m)$ the $m\times m$ matrix whose entries are given by $Q_{ij} = \vec{x}_i^T\vec{x}_j$. Let $F$ be any function of an $m\times m$ matrix such that the integral \begin{equation}
     \int_{\mathbb{R}^N}\ldots\int_{\mathbb{R}^N}d\vec{x}_1\ldots d\vec{x}_m |F(Q)|
\end{equation}
exists and define the integral \begin{equation}
     \mathcal{J}_{N, m}(F) \defeq \int_{\mathbb{R}^N}\ldots\int_{\mathbb{R}^N}d\vec{x}_1\ldots d\vec{x}_m F(Q).
\end{equation} Then we have \begin{equation}
    \mathcal{J}_{N,m}(F) = \frac{\pi^{\frac{m}{2}\left(N - \frac{m-1}{2}\right)}}{\prod_{k=0}^{m-1}\Gamma\left(\frac{N-k}{2}\right)}\int_{\text{Sym}_{\geq 0}(m)}d\hat{Q} \left(\det \hat{Q}\right)^{\frac{N-m-1}{2}}F(\hat{Q}).
\end{equation}
 \end{lemma}
We will prove the following perturbed version of this result and in greater generality than is required in the present work.
\fyodgeneral*}

\jstat{The proof of Lemma \ref{lemma:fyod_lemma} presented in Appendix D of \cite{fyodorov2002characteristic} proceeds by induction on $m$ and relies on writing the integration vector $\vec{x}_m$ as $\vec{x}_m = \rho_m O_m\vec{e}_N$ where $\vec{e}_N$ is the $N$-th basis vector in the chosen orthonormal basis, $\rho_m>0$ is a scalar variable and $O_m$ is an orthogonal matrix. The proof proceeds by making a change of variables for the first $m-1$ integration vectors and then finding that the integrand does not depend on $O_m$ and so the integral over $O_m$ with respect to the Haar measure just contributes a volume factor of \begin{equation}
    \frac{2\pi^{N/2}}{\Gamma(N/2)}.
\end{equation}
It is at this point where the $e^{-iN\Tr SB}$ term in (\ref{eq:ap_fyod_deformed}) causes problems because a dependence on $O_m$ remains. Indeed, we have \begin{align}
  \vec{x}_m^T S\vec{x}_m = \rho_m \vec{e}_N^TO^T_m S O_m \vec{e}_N.
\end{align}
Since $S$ is real symmetric we may take, wlog, $S = N^{\alpha}\text{diag}(s_1, \ldots, s_r, 0, \ldots, 0).$ Then \begin{align}
     e^{-iN\vec{x}_m^T S\vec{x}_m} = e^{-iN^{1 + \alpha}\rho_m \sum_{j=1}^r s_j (o_{Nj})^2}
\end{align}
where $o_{Nj}$ is the $j$-th component of the $N$-th column of $O$. Proceeding with an evaluation of an integral like (\ref{eq:ap_fyod_deformed}) then requires the evaluation of the integral \begin{equation}\label{eq:ap_hciz}
    \int_{O(N)} d\mu_{\text{Haar}}(O_m) e^{-iN^{1 + \alpha}\rho_m \sum_{j=1}^r s_j (o_{Nj})^2}.
\end{equation}
We can now follow \cite{guionnet2005fourier}, in particular the proof of Theorem 7 therein.  We have the well-known result (Fact 8 in \cite{guionnet2005fourier}) that in the sense of distributions\begin{equation}
    (\vec{o}_1, \ldots, \vec{o}_{p}) \sim \left(\frac{\tilde{\vec{g}}_1}{||\tilde{\vec{g}}_1||},\ldots, \frac{\tilde{\vec{g}}_{p}}{||\tilde{\vec{g}}_{p}||}\right)
\end{equation}
for any $p=\mathcal{O}(1)$ and where the $(\tilde{\vec{g}}_j)_{j=1}^{p}$ are constructed via the Gram-Schmidt process from $(\vec{g}_j)_{j=1}^{r_A}\overset{\text{i.i.d.}}{\sim} \mathcal{N}(\bm{0}, 1)$. So in particular \begin{align}\label{eq:ap_gram_s}
    \vec{o}_N \sim \frac{\vec{g}}{||\vec{g}||}, ~~ \vec{g}\sim \mathcal{N}(0,1).
\end{align}}

\jstat{(\ref{eq:ap_gram_s}) then exactly gives 
\begin{align}
    \int_{O(N)} d\mu_{\text{Haar}}(O_m) e^{-iN^{1 + \alpha}\rho_m \sum_{j=1}^r s_j (o_{Nj})^2} &= \int_{\mathbb{R}^N} \frac{d\vec{g}}{(2\pi)^{N/2}} e^{-\frac{\vec{g}^2}{2}} \exp\left(-iN^{1+\alpha}\rho_m \sum_{j=1}^r s_j \frac{g_j^2}{||\vec{g}||^2}\right)
\end{align}}

 \jstat{Introduce the event \begin{equation}
    B_N(\upsilon) \defeq\left\{| N^{-1}\langle \vec{g}, \vec{g}\rangle - 1| \leq N^{-\upsilon} \right\}
\end{equation}
and then from \cite{guionnet2005fourier} we immediately conclude that under the i.i.d Gaussian law of  $\vec{g}$ the complementary event has low probability: \begin{equation}\label{eq:gram_s_bnk2}
    \mathbb{P}(B_N(\upsilon)^c) =\mathcal{O}( C(\upsilon) e^{-\beta N^{1-2\upsilon}})
\end{equation}
where $\beta, C(\upsilon) > 0$ and we take $0<\upsilon < \frac{1}{2}$ to make this statement meaningful. This enables us to write \begin{align}
  \int_{O(N)} d\mu_{\text{Haar}}(O_m) e^{-iN^{1 + \alpha}\rho_m \sum\limits_{j=1}^r s_j (o_{Nj})^2} =  \left(1 + \mathcal{O}(e^{-\beta N^{1-2\upsilon}})\right)\int_{\mathbb{R}^N} &\frac{d\vec{g}}{(2\pi)^{N/2}} e^{-\frac{\vec{g}^2}{2}} \exp\left(-iN^{1+\alpha}\rho_m \sum_{j=1}^r s_j \frac{g_j^2}{||\vec{g}||^2}\right)\indic\{B_N(\upsilon)\}\notag \\
  = \left(1 + \mathcal{O}(e^{-\beta N^{1-2\upsilon}})\right)\int_{\mathbb{R}^N}& \frac{d\vec{g}}{(2\pi)^{N/2}}\indic\{B_N(\upsilon)\}\notag \\ &  e^{-\frac{\vec{g}^2}{2}} \exp\left(-iN^{\alpha}(1+\mathcal{O}(N^{-\upsilon}))\rho_m \sum_{j=1}^r s_j g_j^2\right)\label{eq:ap_gs_practical_fail}
\end{align}
but given $B_N(\upsilon)$ we have $g_j^2 \lesssim N$ for all $j=1,\ldots, N$ and so we do not, as it stands, have uniformly small error terms. We can circumvent this by introducing the following event for $0<\eta<\frac{1}{2}$: \begin{equation}
   E_N^{(r)}(\eta) = \{|g_j|\leq N^{\frac{1}{2} - \eta} ~ \text{ for } j=1,\ldots, r\}.
\end{equation} 
Let us use $\hat{\vec{g}}$ to denote the $N-r$ dimensional vector with components $\left(g_{r+1}, \ldots, g_N\right)$. }

\jstat{Then we have \begin{align}
 \Bigg| |N^{-1}||\hat{\vec{g}}||^2 - 1 | - N^{-1}\sum_{i=1}^r g_j^2  \Bigg| \leq  |N^{-1}||\vec{g}||^2 - 1 | \leq |N^{-1}||\hat{\vec{g}}||^2 - 1 | + N^{-1}\sum_{i=1}^r g_j^2
\end{align}
so if $\eta > \frac{\upsilon}{2}$ then it follows that \begin{equation}
    B_N(\upsilon) ~|~ E_N^{(r)}(\eta) = B_{N-r}(\upsilon).
\end{equation}
But we also have (e.g. \cite{andrews_askey_roy_1999} Appendix C) \begin{align}
    \mathbb{P}( E_N^{(r)}(\eta)) = \left[\text{erf}\left(N^{\frac{1}{2}-\eta}\right)\right]^r = \left[1 - \mathcal{O}(N^{\frac{1}{2} - \eta} e^{-N^{1-2\eta}})\right]^r = 1 - \mathcal{O}(N^{\frac{1}{2} - \eta} e^{-N^{1-2\eta}})
\end{align}
and so  (taking $\eta> \upsilon$, say)\begin{equation}
    \mathbb{P}\left(B_N(\upsilon)\cap E^{(r)}_N(\eta)\right) =     \mathbb{P}\left(B_N(\upsilon) ~|~ E^{(r)}_N(\eta)\right)    \mathbb{P}\left(E^{(r)}_N(\eta)\right) =1 - \mathcal{O}(e^{-\alpha N^{1-2\upsilon}})
\end{equation}
and thus we can replace (\ref{eq:ap_gs_practical_fail}) with 
\begin{align}
  \int_{O(N)} d\mu_{\text{Haar}}(O_m) e^{-iN^{1 + \alpha}\rho_m \sum\limits_{j=1}^r s_j (o_{Nj})^2} = \left(1 + \mathcal{O}(e^{-\beta N^{1-2\upsilon}})\right)\int_{\mathbb{R}^N}& \frac{d\vec{g}}{(2\pi)^{N/2}}\indic\{B_N(\upsilon)\cap E^{(r)}_N(\eta)\}\notag \\ &  e^{-\frac{\vec{g}^2}{2}} \exp\left(-iN^{\alpha}(1+\mathcal{O}(N^{-\upsilon}))\rho_m \sum_{j=1}^r s_j g_j^2\right)\notag\\
  \label{eq:ap_gs_practical_success}
\end{align}
but now $N^{\alpha-\upsilon} g_j^2 \leq N^{\alpha + 1 -\upsilon - 2\eta} \leq N^{\alpha + 1 -3\upsilon} \rightarrow 0 $ as $N\rightarrow \infty$ so long as we choose $\upsilon > \frac{\alpha + 1}{3}$. Given that $\alpha< 1/2$, this choice is always possible for $0< \upsilon < 1/2$. Thus the error term in the exponent of (\ref{eq:ap_gs_practical_success}) is in fact uniformly small in $\vec{g}$ and so we obtain 
\begin{align}
  \int_{O(N)} d\mu_{\text{Haar}}(O_m) e^{-iN^{1 + \alpha}\rho_m \sum\limits_{j=1}^r s_j (o_{Nj})^2} &= \left(1 + o(1)\right)\int_{\mathbb{R}^N} \frac{d\vec{g}}{(2\pi)^{N/2}}\indic\{B_N(\upsilon)\cap E^{(r)}_N(\eta)\}  \exp\left(-\frac{\vec{g}^2}{2}-iN^{\alpha}\rho_m \sum_{j=1}^r s_j g_j^2\right)\notag\\
   &= \left(1 + o(1)\right)\int_{\mathbb{R}^r} \frac{dg_1\ldots dg_r}{(2\pi)^{r/2}} \exp\left(-\frac{1}{2}\sum_{j=1}^r\left\{1+iN^{\alpha}\rho_m  s_j\right\} g_j^2\right)\notag \\
  &= (1 + o(1)) \prod_{j=1}^r \left( 1+ iN^{\alpha} \rho_m s_j\right)^{-\frac{1}{2}}.\label{eq:ap_haar_finished}
\end{align}
In the induction step in the proof of \cite{fyodorov2002characteristic}, $\rho_m$ becomes the new diagonal entry of the expanded $\hat{Q}$ matrix. Combining (\ref{eq:ap_haar_finished}) with that proof gives the result
 \begin{align}\label{eq:fyod_lemma_result}
    \mathcal{I}_N(F; S) =(1+ o(1))\frac{\pi^{N - \frac{1}{2}}(1+ o(1))}{\Gamma\left(\frac{N}{2}\right)\Gamma\left(\frac{N-1}{2}\right)}\int_{\text{Sym}_{\geq 0}(m)}d\hat{Q} \left(\det \hat{Q}\right)^{\frac{N-3}{2}}F(\hat{Q}) \prod_{j=1}^r\prod_{i=1}^N \left( 1+ iN^{\alpha}\hat{Q}_{ii}s_j\right).
\end{align}}

\section{Experimental details}\label{ap:experiments}
In this section we give further details of the experiments presented in Section \ref{subsec:discussion_assumptions}.\\

The MLP architecture used consists of hidden layers of sizes $1000, 1000, 500, 250$. The CNN architecture used is a standard LeNet style architecture: \begin{enumerate}
    \item 6 filters of size $4\times 4$.
    \item Activation.
    \item Max pooling of size $2\times 2$ and stride $2$.
    \item 16 filters of size $4\times 4$.
    \item Activation.
    \item Max pooling of size $2\times 2$ and stride $2$.
        \item 120 filters of size $4\times 4$.
    \item Activation.
    \item Dropout.
    \item Fully connected to size 84.
        \item Activation
            \item Dropout.
    \item Fully connected to size 10.
\end{enumerate} 
The activation functions used were the ubiquitous $\texttt{ReLU}$ defined by \begin{equation}
    \texttt{ReLU}(x) = \max(0, x),
\end{equation}
and \texttt{HardTanh} defined by \begin{equation}
    \texttt{HardTanh}(x) = \begin{cases}
    x ~~ &\text{for } x\in(-1,1),\\
     -1 ~~ &\text{for } x\leq -1,\\  
     1 ~~ &\text{for } x\geq 1,\\   
    \end{cases}
\end{equation}
and a custom 5 piece function $f_5$ with gradients $0.01,0.1, 1, 0.3, 0.03$ on $(-\infty, -2), (-2,-1), (-1,1), (1,2), (2, \infty)$ respectively, and $f_5(0) = 0 $. We implemented all the networks and experiments in PyTorch \cite{paszke2017automatic} and our code is made available in the form of a Python notebook capable of easily reproducing all plots\footnote{\url{https://github.com/npbaskerville/loss-surfaces-general-activation-functions}.}.

\end{document}